\documentclass[11pt]{article}

\usepackage{graphicx,wrapfig,lipsum}
\usepackage{float}    
\usepackage{verbatim} 
\usepackage{amsmath,amssymb, amsthm,mathtools}
\usepackage{subfig}   
\usepackage[colorlinks=true,citecolor=blue,linkcolor=blue,urlcolor=blue]{hyperref}
\usepackage{bookmark}
\usepackage{fullpage}
\usepackage{enumerate}
\usepackage{paralist}
\usepackage{xspace}
\usepackage{caption}
\usepackage{bbm}
\usepackage{bm}
\usepackage{url}
\usepackage{mathrsfs}
\usepackage{algorithm}
\usepackage{algorithmic}
\usepackage{xcolor}
\usepackage{microtype}
\usepackage[numbers]{natbib}
\usepackage[section]{placeins}
\usepackage{lscape}
\usepackage{multirow}
\usepackage{tikz}

\usepackage[nohyperref]{jmlr2e}

\makeatletter

\AtBeginDocument{\providecommand\secref[1]{Section \ref{sec:#1}}}
\AtBeginDocument{\providecommand\eqref[1]{\ref{eq:#1}}}
\AtBeginDocument{\providecommand\propref[1]{Proposition \ref{prop:#1}}}
\AtBeginDocument{\providecommand\thmref[1]{Theorem \ref{thm:#1}}}
\AtBeginDocument{\providecommand\corref[1]{Corollary \ref{cor:#1}}}
\AtBeginDocument{\providecommand\lemref[1]{Lemma \ref{lem:#1}}}
\AtBeginDocument{\providecommand\defref[1]{Definition \ref{def:#1}}}
\AtBeginDocument{}
\AtBeginDocument{\providecommand\appref[1]{Appendix  \ref{app:#1}}}

\usepackage{lastpage}
\jmlrheading{24}{2023}{1-\pageref{LastPage}}{8/22; Revised 4/23}{4/23}{22-0882}{Gavin Zhang and Salar Fattahi and Richard Y. Zhang} 
\ShortHeadings{PrecGD for Overparameterized Burer--Monteiro}{Zhang, Fattahi, and Zhang}
\hypersetup{ hidelinks }

\firstpageno{1}

\begin{document}

\title{Preconditioned Gradient Descent for Overparameterized Nonconvex Burer--Monteiro
Factorization with Global Optimality Certification
\thanks{Financial support for this work was provided by NSF CAREER Award ECCS-2047462,
NSF Award DMS-2152776, ONR Award N00014-22-1-2127.}}

\author{\name Gavin Zhang \email jialun2@illinois.edu \\
       \addr Electrical and Computer Engineering\\
       University of Illinois at Urbana-Champaign, Urbana, IL 61801, USA
       \AND
       \name Salar Fattahi \email fattahi@umich.edu \\
       \addr Industrial and Operations Engineering\\
       University of Michigan, Ann Arbor, MI 48109, USA
       \AND
       \name Richard Y.\ Zhang \email ryz@illinois.edu \\
       \addr Electrical and Computer Engineering\\
       University of Illinois at Urbana-Champaign, Urbana, IL 61801, USA}
  
\editor{Prateek Jain}
\maketitle

\begin{abstract}%
We consider using gradient descent to minimize the nonconvex function
$f(X)=\phi(XX^{T})$ over an $n\times r$ factor matrix $X$, in which
$\phi$ is an underlying smooth convex cost function defined over
$n\times n$ matrices. While only a second-order stationary point
$X$ can be provably found in reasonable time, if $X$ is additionally
\emph{rank deficient}, then its rank deficiency certifies it as being
globally optimal. This way of certifying global optimality necessarily
requires the search rank $r$ of the current iterate $X$ to be \emph{overparameterized}
with respect to the rank $r^{\star}$ of the global minimizer $X^{\star}$.
Unfortunately, overparameterization significantly slows down the convergence
of gradient descent, from a linear rate with $r=r^{\star}$ to a sublinear
rate when $r>r^{\star}$, even when $\phi$ is strongly convex. In
this paper, we propose an inexpensive preconditioner that restores
the convergence rate of gradient descent back to linear in the
overparameterized case, while also making it agnostic to possible
ill-conditioning in the global minimizer $X^{\star}$. 
\end{abstract}


\global\long\def\A{\mathbf{A}}%
\global\long\def\J{\mathbf{J}}%
\global\long\def\e{\mathbf{e}}%
\global\long\def\P{\mathbf{P}}%
\global\long\def\Pb{\mathbb{P}}%
\global\long\def\E{\mathbb{E}}%
\global\long\def\V{\mathbb{V}}%
\global\long\def\N{\mathcal{N}}%
\global\long\def\vect{\mathrm{vec}}%
\global\long\def\vbrack#1{\langle#1\rangle}%
\global\long\def\norm#1{\left\Vert #1\right\Vert }%
\global\long\def\bigvbrack#1{\left\langle #1\right\rangle }%
\global\long\def\inner#1#2{\left\langle #1, #2 \right\rangle }%
\global\long\def\S{\mathcal{S}}%
\global\long\def\D{\mathcal{D}}%
\global\long\def\d{\mathrm{d}}%
\global\long\def\EE{\mathbf{E}}%
\global\long\def\AA{\mathcal{A}}%
\global\long\def\JJ{\mathcal{J}}%
\global\long\def\R{\mathbb{R}}%
\global\long\def\B{\mathbb{B}}%
\global\long\def\ub{\mathrm{ub}}%
\global\long\def\lb{\mathrm{lb}}%
\global\long\def\op{\mathrm{op}}%
\global\long\def\eqdef{\overset{\text{def}}{=}}%
\global\long\def\PP{\mathcal{P}}%
\global\long\def\TT{\mathcal{T}}%

\global\long\def\rank{\operatorname{rank}}%

\global\long\def\tr{\operatorname{tr}}%
\global\long\def\orth{\operatorname{orth}}%
\global\long\def\range{\operatorname{range}}%
\global\long\def\cXs{\mathcal{X}_{\text{stuck}}}%
\global\long\def\poly{\operatorname{poly}}%

\section{Introduction}
Numerous state-of-the-art algorithms in statistical and machine learning
can be viewed as gradient descent applied to the nonconvex Burer--Monteiro~\citep{burer2003nonlinear,burer2005local}
problem
\begin{equation}
X^{\star}=\quad\mathrm{minimize}\quad f(X)\eqdef\phi(XX^{T})\text{ over }X\in\R^{n\times r},\tag{BM}\label{eq:ncvx}
\end{equation}
in which $\phi$ is an underlying convex cost function defined over
$n\times n$ matrices. Typically, the search rank $r\ll n$ is set
significantly smaller than $n$, and an efficient gradient oracle
$X\mapsto\nabla f(X)$ is available due to problem structure that
costs $n\cdot\poly(r)$ time per query. 
Under these two assumptions,
each iteration of gradient descent $X_{+}=X-\alpha\nabla f(X)$ costs
$O(n)$ time and memory. 

Gradient descent has become widely popular for problem (\ref{eq:ncvx})
because it is simple to implement but works exceptionally well in
practice~\citep{sun2016guaranteed,bhojanapalli2016dropping,bhojanapalli2016global,park2017non,chen2017solving,park2018finding,chen2019gradient}.
Across a broad range of applications, gradient descent is consistently observed to converge from an arbitrary, possibly random initial guess $X_{0}$ to the global minimum $X^{\star}$, as if the function $f$ were \emph{convex}. In fact, in many cases, the observed convergence rate is even \emph{linear}, meaning that gradient descent converges to $\epsilon$ global suboptimality in $O(\log(1/\epsilon))$ iterations, as if the function $f$ were \emph{strongly convex}. When this occurs, the resulting empirical complexity of $\epsilon$-accuracy in $O(n\cdot\log(1/\epsilon))$ time matches the best figures achievable by algorithms for convex optimization. 

However, due to the nonconvexity of $f$, it is always possible for
gradient descent to fail by getting stuck at a \emph{spurious} local
minimum---a local minimum that is strictly worse than that of the
global minimum. This is a particular concern for safety-critical
applications like electricity grids~\citep{zhang2019spurious} and
robot navigation~\citep{rosen2019se,rosen2020certifiably}, where
mistaking a clearly suboptimal $X$ for the globally optimal $X^{\star}$
could have serious rammifications. Recent authors have derived conditions under which $f$ is guaranteed not to admit spurious local minima,
but such \emph{a priori} global optimality guarantees, which are valid
for all initializations before running the algorithm, can be much
stronger than what is needed for gradient descent to succeed in practice.
For example, it may also be the case that spurious local minima do
generally exist, but that gradient descent is frequently able to avoid
them without any rigorous guarantees of doing so. 

In this paper, we consider   \emph{overparameterizing} the search rank
$r$, choosing it to be large enough so that $\rank(X^{\star})<r$
holds for all globally optimal $X^{\star}$. We are  motivated by the ability to guarantee global optimality \emph{a posteriori}, that is,
\emph{after} a candidate $X$ has already been computed.  To explain,
it has been long suspected and recently rigorously shown~\cite{ge2015escaping,jin2017escape,jin2021nonconvex} that gradient descent can be made to converge to an approximate second-order stationary point
$X$ that satisfies
\begin{equation}
|\inner{\nabla f(X)}V|\le\epsilon_{g}\|V\|_{F},\quad\inner{\nabla^{2}f(X)[V]}V\ge-\epsilon_{H}\|V\|_{F}^{2}\quad\text{ for all }V\in\R^{n\times r}\label{eq:soc-F}
\end{equation}
with arbitrarily small accuracy parameters $\epsilon_{g},\epsilon_{H}>0$.
By evoking an argument first introduced by \citet[Theorem~4.1]{burer2005local}
(see also \citet{journee2010low} and \citet{boumal2016non,boumal2020deterministic}) one can show that an $X$ that
satisfies (\ref{eq:soc-F}) has global suboptimality:
\begin{equation}
f(X)-f(X^{\star})\le\underbrace{C_{g}\cdot\epsilon_{g}}_{\text{gradient norm}}+\underbrace{C_{H}\cdot\epsilon_{H}}_{\text{Hessian curvature}}+\underbrace{C_{\lambda}\cdot\lambda_{\min}(X^{T}X)}_{\text{rank deficiency}}\label{eq:certF}
\end{equation}
where $\lambda_{\min}(\cdot)$ denotes the smallest eigenvalue, and
$C_{g},C_{H},C_{\lambda}>0$ are absolute constants under standard
assumptions. By overparameterizing the search rank so that $r>r^\star$ holds, where $r^\star$ denotes the \emph{maximum rank} over all globally optimal $X^{\star}$, it follows from (\ref{eq:certF})
that the global optimality of an $X$ with $\epsilon_{g}\approx0$
and $\epsilon_{H}\approx0$ is conclusively determined by 
its rank deficiency term $\lambda_{\min}(X^{T}X)$:
\begin{enumerate}
\item (Near globally optimal) If $X\approx X^{\star}$, then $X$ must be
nearly rank deficient with $\lambda_{\min}(X^{T}X)\approx0$. In this
case, the near-global optimality of $X$ can be rigorously \emph{certified}
by evoking (\ref{eq:certF}) with $\epsilon_{g}\approx0$ and $\epsilon_{H}\approx0$
and $\lambda_{\min}(X^{T}X)\approx0$.
\item (Stuck at spurious point) If $f(X)\gg f(X^{\star})$, then by contradiction
$X$ must be nearly full-rank with $\lambda_{\min}(X^{T}X)\approx C_{\lambda}^{-1}(f(X)-f(X^{\star}))\not\approx 0$ bounded away from zero.
\end{enumerate}
As we describe in \secref{Cert}, the three parameters $\epsilon_{g},\epsilon_{H},$
and $\lambda_{\min}(X^{T}X)$ for a given $X$ can all be numerically
evaluated in $O(n)$ time and memory, using a small number of calls
to the gradient oracle $X\mapsto\nabla f(X)$. (In \secref{Cert},
we formally state and prove \eqref{eq:certF} as \propref{euclid_cert}.)

Aside from the ability to certify global optimality, a second benefit
of overparameterization is that $f$ tends to admit fewer spurious
local minima as the search rank $r$ is increased beyond the maximum
rank $r^{\star}\ge\rank(X^{\star})$. Indeed, it is commonly observed
in practice that any local optimization algorithm seem to globally
solve problem (\ref{eq:ncvx}) as soon as $r$ is slightly larger
than $r^{\star}$; see \citep{burer2003nonlinear,journee2010low,rosen2019se}
for numerical examples of this behavior. Towards a rigorous explanation,
\citet{boumal2016non,boumal2020deterministic} pointed out that if
the search rank is overparameterized as $r\ge n$, then the function
$f$ is guaranteed to contain \emph{no spurious local minima}, in
the sense that every second-order stationary point $Z$ satisfying
$\nabla f(Z)=0$ and $\nabla^{2}f(Z)\succeq0$ is guaranteed to be
global optimal $f(Z)=f(X^{\star})$. This result was recently sharpened
by \citet{zhang2022improved}, who proved that if the underlying convex
cost $\phi$ is $L$-gradient Lipschitz and $\mu$-strongly convex,
then overparameterizing the search rank by a constant factor as $r>\max\{r^{\star},\frac{1}{4}(L/\mu-1)^{2}r^{\star}\}$
is enough to guarantee that $f$ contains no spurious local minima.

Unfortunately, overparameterization significantly slows down the
convergence of gradient descent, both in theory and in practice. Under
suitable strong convexity and optimality assumptions on $\phi$, \citet{zheng2015convergent,tu2016low}
showed that gradient descent $X_{+}=X-\alpha\nabla f(X)$ locally
converges as follows
\[
f(X_{+})-f(X^{\star})\le\left[1-\alpha\cdot c\cdot\lambda_{\min}(X^{T}X)\right]\cdot\left[f(X)-f(X^{\star})\right]
\]
where $\alpha>0$ is the corresponding step-size, and $c>0$ is a
constant (see also \secref{sublin} for an alternative derivation). In the exactly parameterized regime $r=r^\star$, this inequality implies linear convergence, because $\lambda_{\min}(X^TX) > 0$ holds within a local neighborhood of the minimizer $X^\star$.
In the overparameterized regime $r>r^{\star}$, however, the iterate $X$ becomes
increasingly singular $\lambda_{\min}(X^{T}X)\to0$ as it makes progress
towards the global minimizer $X^{\star}$, and the convergence quotient
$Q=1-\alpha\cdot c\cdot\lambda_{\min}(X^{T}X)$ approaches 1. In practice,
gradient descent slows down to \emph{sublinear} convergence, now requiring
$\mathrm{poly}(1/\epsilon)$ iterations to yield an $\epsilon$ suboptimal
solution. This is a dramatic, exponential slow-down compared to the
$O(\log(1/\epsilon))$ figure associated with linear convergence under
exact rank parameterization $r=r^{\star}$.

For applications of gradient descent with very large values of $n$,
this exponential slow-down suggests that the ability to certify global
optimality via overparameterization can only come by dramatically worsening the quality of
the computed solution. In most cases, it remains better to exactly
parameterize the search rank as $r=r^{\star}$, in order to compute
a high-quality solution without a rigorous proof of quality. For safety-critical
applications for which a proof of quality is paramount, rank overparameterization
$r>r^{\star}$ is used alongside much more expensive trust-region
methods~\citep{rosen2019se,rosen2020certifiably,boumal2020deterministic}.
These methods can be made immune to the progressive ill-conditioning
$\lambda_{\min}(X^{T}X)\to0$ of the current iterate $X$, but have
per-iteration costs of $O(n^{3})$ time and $O(n^{2})$ memory that
limit $n$ to modest values. 

\subsection{Summary of results}

In this paper, we present an inexpensive \emph{preconditioner} for
gradient descent that restores the convergence rate of gradient descent
back to linear in the overparameterized case, both in theory and
in practice. We propose the following iterations 
\begin{equation}
X_{+}=X-\alpha\nabla f(X)(X^{T}X+\eta I)^{-1},\tag{PrecGD}\label{PrecGD}
\end{equation}
where $\alpha\in(0,1]$ is a fixed step-size, and $\eta\ge0$ is a
regularization parameter that may be changed from iteration to iteration.
We call these iterations \emph{preconditioned} gradient descent or
\ref{PrecGD}, because they can be viewed as gradient descent applied
with a carefully chosen $r\times r$ preconditioner. 

It is easy to see that \ref{PrecGD} should maintain a similar per-iteration $O(n)$
cost to regular gradient descent  in most applications where the Burer-Monteiro approach is used, where $r$ is typically set orders of magnitude smaller than $n$. The method induces an additional cost of $O(r^3)$ each iteration to form and compute the preconditioner $(X^TX+\eta I)^{-1}$. But in applications with very large values of $n$  and very small values of $r$, the small increase in the per-iteration cost, from $O(r)$ to $O(r^3)$, is completely offset by the exponential reduction in the number of iterations, from $O(1/\epsilon)$ to $O(\log(1/\epsilon))$.  Therefore, \ref{PrecGD} can serve
as a plug-in replacement for gradient descent, in order to provide the ability to
certify global optimality without sacrificing the high quality of
the computed solution. 

Our results are summarized as follows:

\begin{figure}[t]
\includegraphics[width=0.5\textwidth]{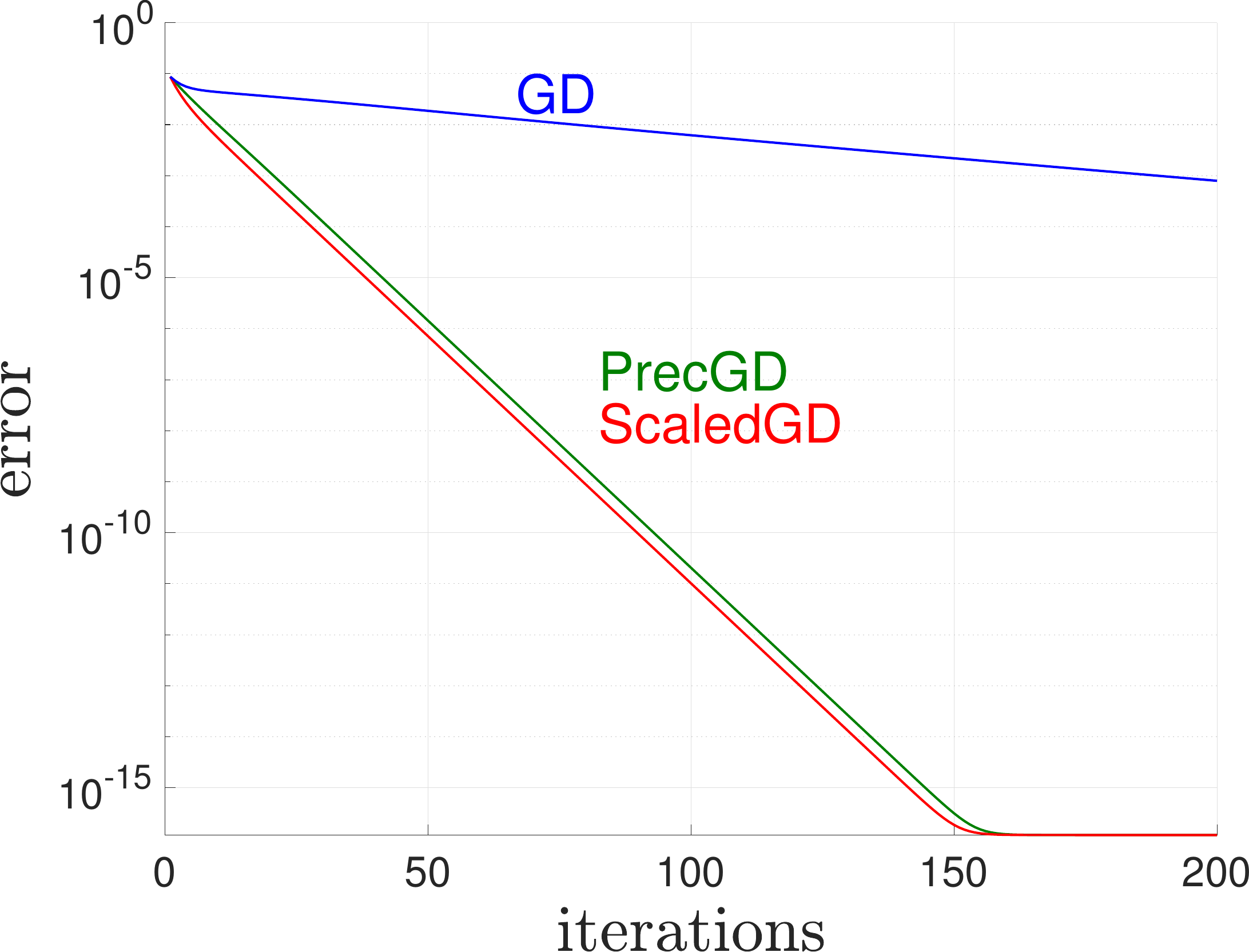}
\hfill{}\includegraphics[width=0.5\textwidth]{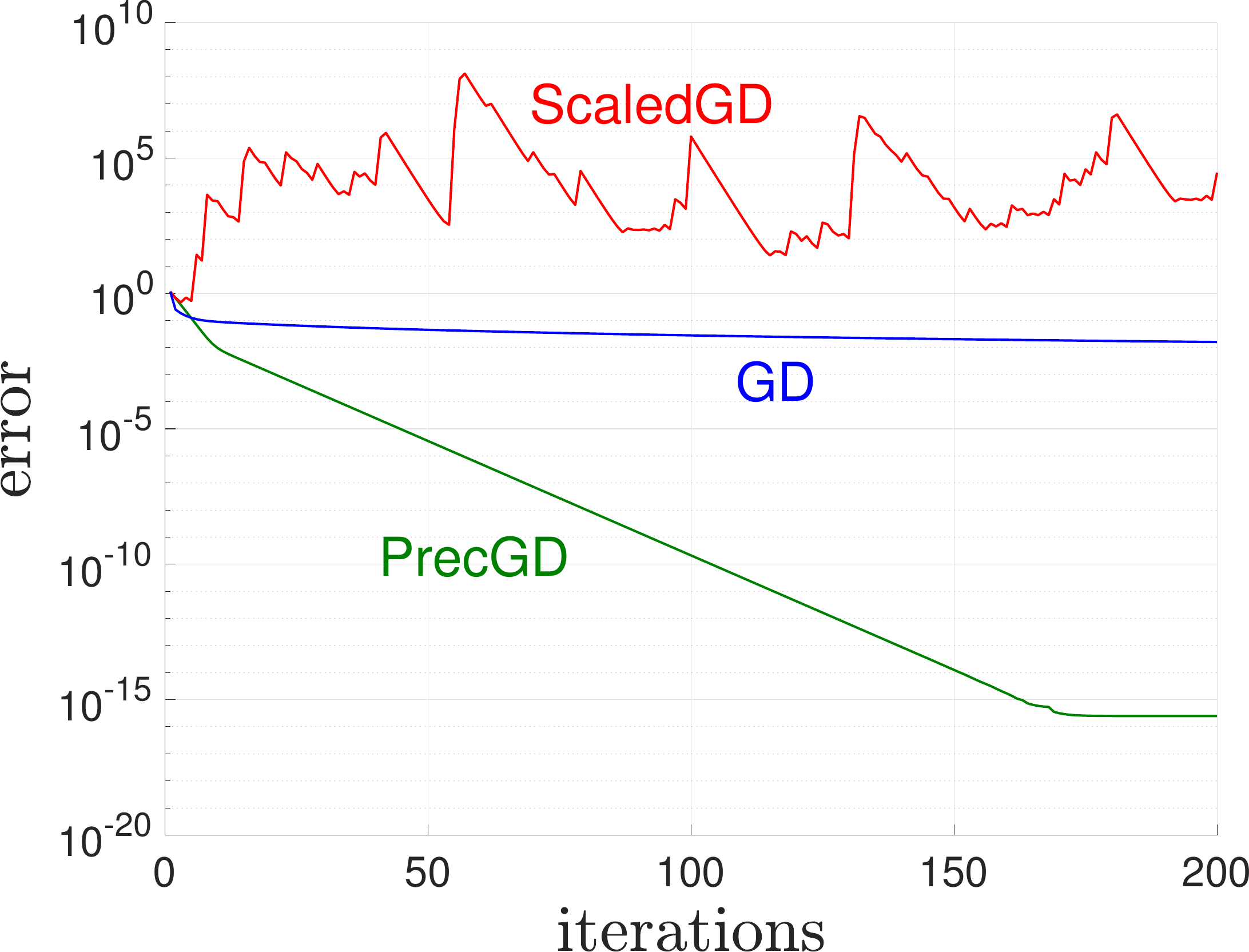}
 \caption{\textbf{PrecGD converges linearly in the overparameterized regime.}
Comparison of \ref{PrecGD} against regular gradient descent (GD),
and the ScaledGD algorithm of \citet{tong2020accelerating} for an
instance of (\ref{eq:ncvx}) taken from~\citep{NEURIPS2018_f8da71e5,zhang2019sharp}.
The same initial points and the same step-size $\alpha=2\times10^{-2}$
was used for all three algorithms. \textbf{(Left $r=r^{*}$)} Set
$n=4$ and $r^{*}=r=2$. All three methods convergence at a linear
rate, though GD converges at a slower rate due to ill-conditioning
in the ground truth. \textbf{(Right $r>r^{*}$)} With $n=4$, $r=4$
and $r^{*}=2$, overparameterization causes gradient descent to slow
down to a sublinear rate. ScaledGD also behaves sporadically. Only
PrecGD converges linearly to the global minimum.\label{fig:linconv}}
\end{figure}

\paragraph{Local convergence.}

Starting within a neighborhood of the global minimizer $X^{\star}$,
and under suitable strong convexity and optimality assumptions on
$\phi$, classical gradient descent converges to $\epsilon$ suboptimality
in $O(1/\lambda_{r}\log(1/\epsilon))$ iterations where $\lambda_{r}=\lambda_{\min}(X^{\star T}X^{\star})$
is the rank deficiency term of the global minimizer~\citep{zheng2015convergent,tu2016low}.
This result breaks down in the overparameterized regime, where $r>r^{\star}=\rank(X^{\star})$
and $\lambda_{r}=0$ holds by definition; instead, gradient descent
now requires $\poly(1/\epsilon)$ iterations to converge to $\epsilon$
suboptimality~\citep{zhuo2021computational}. 

Under the same strong
convexity and optimality assumptions, we prove that \ref{PrecGD}
with the parameter choice $\eta=\|\nabla f(X)(X^{T}X)^{-1/2}\|_{F}$
converges to $\epsilon$ global suboptimality in $O(\log(1/\epsilon))$
iterations, independent of $\lambda_{r}^{\star}=0$. In fact, we prove
that the convergence rate of \ref{PrecGD} also becomes independent
of the smallest nonzero singular value $\lambda_{r^{\star}}=\lambda_{r^{\star}}(X^{\star T}X^{\star})$
of the global minimizer $X^{\star}$. In practice, this often allows
\ref{PrecGD} to converge faster in the overparameterized regime $r>r^{\star}$
than regular gradient descent in the exactly parameterized regime
$r=r^{\star}$ (see Fig.~\ref{fig:linconv}). In our numerical results,
we observe that the linear convergence rate of \ref{PrecGD} for all
values of $r\ge r^{\star}$ and $\lambda_{r^{\star}}>0$ is the same
as regular gradient descent with a perfectly conditioned global minimizer
$X^{\star}$, i.e. with $r=r^{\star}$ and $\lambda_{r^{\star}}=\lambda_{1}(X^{\star T}X^{\star})$.
In fact, linear convergence was observed even for choices of $\phi$
that do not satisfy the notions of strong convexity considered in
our theoretical results. 

\paragraph{Global convergence.}

If the function $f$ can be assumed to admit no spurious local minima,
then under a strict saddle assumption~\citep{ge2015escaping,ge2017no,jin2017escape},
classical gradient descent can be augmented with random perturbations
to globally converge to $\epsilon$ suboptimality in $O(1/\lambda_{r}\log(1/\epsilon))$
iterations, starting from any arbitrary initial point. In the overparameterized
regime, however, this global guarantee worsens by an exponential factor
to $\poly(1/\epsilon)$ iterations, due to the loss of local linear
convergence. 

Instead, under the same benign landscape assumptions
on $f$, we show that \ref{PrecGD} can be similarly augmented with
random perturbations to globally converge to $\epsilon$ suboptimality
in $O(\log(1/\epsilon))$ iterations, independent of $\lambda_{r}=0$
and starting from any arbitrary initial point. A major difficulty
here is the need to account for a preconditioner $(X^{T}X+\eta I)^{-1}$
that changes after each iteration. We prove an $\tilde{O}(1/\delta^{2})$
iteration bound to $\delta$ approximate second-order stationarity
for the perturbed version of \ref{PrecGD} with a fixed $\eta=\eta_{0}$,
by viewing the preconditioner as a local norm metric that is both
well-conditioned and Lipschitz continuous. 

\paragraph{Optimality certification.}

Finally, a crucial advantage of the overparameterizing the search
rank $r>r^{\star}$ is that it allows \emph{a posteriori} certification
of convergence to a global minimum. We give a short proof that if
$X$ is $\epsilon$ suboptimal, then this fact can be explicitly verified
by appealing to its second-order stationarity and its rank deficiency.
Conversely, we prove that if $X$ is stuck at a spurious second-order
critical point, then this fact can also be explicitly detected via
its lack of rank deficiency.

\subsection{Related work}

\paragraph{Benign landscape.}

In recent years, there has been significant progress in developing
rigorous guarantees on the global optimality of local optimization
algorithms like gradient descent~\citep{ge2016matrix,bhojanapalli2016dropping,sun2016complete,ge2017no,sun2018geometric}.
For example, \citet{bhojanapalli2016global} showed that if the underlying
convex function $\phi$ is $L$-gradient Lipschitz and $\mu$-strongly
convex with a sufficiently small condition number $L/\mu$, then $f$
is guaranteed to have no spurious local minima and satisfy the\emph{
}strict saddle property of \citep{ge2015escaping} (see also \citet{ge2017no}
for an exposition of this result). Where these properties hold, \citet{jin2017escape,jin2021nonconvex}
showed that gradient descent is rigorously guaranteed (after minor
modifications) to converge to $\epsilon$ global suboptimality in
$O(\log(1/\epsilon))$ iterations, starting from any arbitrary initial
point. 

Unfortunately, \emph{a priori} global optimality guarantees, which
must hold for all initializations before running the algorithm, can
often require assumptions that are too strong to be widely applicable
in practice \cite{ma2022blessing, ma2022global}. For example, \citet{zhang2018much,zhang2019sharp} found
for a global optimality guarantee
to be possible, the underlying convex function $\phi$ must have a condition number
of at most $L/\mu<3$, or else the claim is false due to the existence
of a counterexample. And while \citet{zhang2021sharp} later extended
this global optimality guarantee to arbitrarily large condition numbers
$L/\mu$ by overparameterizing the search rank $r>\max\{r^{\star},\frac{1}{4}(L/\mu-1)^{2}r^{\star}\}$,
the result does require suitable strong convexity and optimality assumptions
on $\phi$. Once these assumptions are lifted, \citet{waldspurger2020rank}
showed that a global optimality guarantee based on rank overparameterization
would necessarily require $r\ge n$ in general; of course, with such
a large search rank, gradient descent would no longer be efficient.

In this paper, we rigorously certify the global optimality of a point
$X$ after it has been computed. This kind of \emph{a posteriori}
global optimality guarantee may be more useful in practice, because
it makes no assumptions on the landscape of the nonconvex function $f$,
nor the algorithm used to compute $X$. In particular, $f$ may admit
many spurious local minima, but an \emph{a posteriori} guarantee will
continue to work so long as the algorithm is eventually able to compute
a rank deficient second-order point $X^{\star}$, perhaps after many
failures. Our numerical results find that \ref{PrecGD} is able to
broadly achieve an exponential speed-up over classical gradient descent,
even when our theoretical assumptions do not hold. 

\paragraph{Ill-conditioning and Over-parameterization}
 When minimizing the function $\phi(XX^T)$, ill-conditioning in this problem can come from two separate sources: ill-conditioning of the ground truth $M^\star$ and ill-conditioning of the loss function $\phi$. Both can cause gradient descent to slow down \citep{tu2016low, zhuo2021computational}. In this work, we focus on the former kind of ill-conditioning, because it is usually the more serious issue in practice. In applications like matrix completion or matrix sensing, the condition number of the loss function $\phi$ is entirely driven by the number of samples that the practitioner has collected---the more samples, the better the condition number. Accordingly, any ill-conditioning in $\phi$ can usually be overcome by collecting more samples. On the other hand, the ill-conditioning in $M^\star$ is inherent to the underlying nature of the data. It cannot be resolved, for example, by collecting more data. For these real-world applications, it was recently noted that the condition number of $M^\star$ can be as high as $10^8$ \citep{cloninger2014solving}. Indeed, if the rank of $M^\star$ is unknown or ill-defined, as in the over-parameterized case, the condition number is essentially infinite, and it was previously not known how to make gradient descent converge quickly.

\paragraph{ScaledGD.}

Our algorithm is closely related to the \emph{scaled gradient descent}
or \emph{ScaledGD }algorithm of \citet{tong2020accelerating}, which
uses a preconditioner of the form $(X^{T}X)^{-1}$. They prove that
ScaledGD is able to maintain a constant-factor decrement after each
iteration, even as $\lambda_{r}=\lambda_{\min}(X^{\star T}X^{\star})$
becomes small and $X^{\star}$ becomes ill-conditioned. However, applying
ScaledGD to the overparameterized problem with $\lambda_{r}=0$ and
a rank deficient $X^{\star}$ leads to sporadic and inconsistent behavior.
 The issue is that the admissible step-sizes needed to maintain a
constant-factor decrement also shrinks to zero as $\lambda_{r}$ goes
to zero (we elaborate on this point in detail in \secref{local}).
If we insist on using a constant step-size, then the method will on
occasion \emph{increment }after an iteration (see Fig.~\ref{fig:linconv}). 

Our main result is that regularizing the preconditioner as $(X^{T}X+\eta I)^{-1}$
with an identity perturbation $\eta I$ on the same order of magnitude
as the matrix error norm $\|XX^{T}-X^{\star}X^{\star}\|_{F}$ will
maintain the constant-factor decrement of ScaledGD, while also keeping
a constant admissible step-size. The resulting iterations, which we
call \ref{PrecGD}, is able to converge linearly, at a rate that is
independent of the rank deficiency term $\lambda_{r}$, even as it
goes to zero in the overparameterized regime.

\paragraph{Riemann Staircase.}

An alternative approach for certifying global optimality, often known
in the literature as the \emph{Riemann staircase}~\citep{boumal2015riemannian,boumal2016non,boumal2020deterministic},
is to progressively increase the search rank $r$ only after a second
order stationary point has been found. The essential idea is to keep
the search rank exactly parameterized $r=r^{\star}$ during the local
optimization phase, and to overparameterize only for the purpose of
certifying global optimality. After a full-rank $\epsilon$ second
order stationary point $X$ has been found in as few as $O(\log(1/\epsilon))$
iterations, we attempt to certify it as $\epsilon$ globally suboptimal
by increasing the search rank $r_{+}=r+1$ and augmenting $X_{+}=[X,0]$
with a column of zeros. If the augmented $X_{+}$ remains $\epsilon$
second order stationary under the new search rank $r_{+}$, then it
is certifiably $\epsilon$ globally suboptimal. Otherwise, $X_{+}$
is a saddle point; we proceed to reestablish $\epsilon$ second order
stationarity under the new search rank $r_{+}$ by performing another
$O(\log(1/\epsilon))$ iterations. 

The main issue with the Riemann staircase is that choices of $f$
based on real data often admit global minimizers $X^{\star}$ whose
singular values trail off slowly, for example like a power series
$\sigma_{i}(X^{\star})\approx1/i$ for $i\in\{1,2,\dots,r\}$ (see
e.g.~\citet[Fig.~S3]{kosinski2013private} for a well-cited example).
In this case, the search rank $r$ is always exactly parameterized
$r=r^{\star}$, but the corresponding $\epsilon$ second order stationary
point $X$ becomes progressively ill-conditioned as $r$ is increased.
In practice, ill-conditioning can cause a similarly dramatic slow-down
to gradient descent as overparameterization, to the extent that it
becomes indistinguishable from sublinear convergence. Indeed, existing
implementations of the Riemann staircase are usually based on much
more expensive trust-region methods; see e.g. \citet{rosen2019se,rosen2020certifiably}.

\textbf{Notations.} We denote $\lambda_{i}(M)$ as the $i$-th eigenvalue
of $M$ in descending order, as in $\lambda_{1}(M)\ge\lambda_{2}(M)\ge\cdots\ge\lambda_{n}(M)$. Similarly we use $\lambda_{\max}$ and $\lambda_{\min}$ to denote the largest and smallest singular value of a matrix. The matrix inner product is defined $\inner XY\eqdef\tr(X^{T}Y)$,
and that it induces the Frobenius norm as $\|X\|_{F}=\sqrt{\inner XX}$.
The vectorization $\vect(X)$ is the usual column-stacking operation
that turns a matrix into a column vector and $\otimes$ denote the Kronecker product. 
 Moreover, we use $\|X\|$ to denote the spectral norm (i.e. the induced 2-norm) of a matrix. We use $\nabla f(X)$ to denote the gradient at $X$, which is itself a matrix of same dimensions as $X$. The Hessian $\nabla^2 f(X)$ is defined as the linear operator that satisfies 
$
\nabla^2 f(X)[V] = \lim_{t\to 0} \frac{1}{t} [\nabla f(X+tV)-\nabla f(X)]
$
for all $V$. The symbol $\mathbb{B}(d)$ shows the Euclidean ball of radius $d$
centered at the origin. The notation $\tilde{O}(\cdot)$ is used to hide logarithmic terms in the usual big-O notation.

We always use $\phi(\cdot)$ to denote the original convex objective and $f(X)=\phi(XX^{T})$ to denote the factored objective function. We use $M^\star$ to denote the global minimizer of $\phi(\cdot)$. The dimension of $M^\star$ is $n$, and its rank is $r^\star$. Furthermore, the search rank is denoted by $r$, which means that $X$ is $n\times r$. We always assume that $\phi(\cdot)$ Lipschitz gradients, and is $(\mu,r)$ restricted strongly convex (see next section for precise definition). When necessary, we will also assume that $\phi(\cdot)$ has $L_2$-Lipschitz Hessians.

\section{Convergence Guarantees}

\subsection{Local convergence}

Let $f(X)\eqdef\phi(XX^{T})$ denote the a Burer--Monteiro cost function
defined over $n\times r$ factor matrices $X$. Under gradient Lipschitz
and strong convexity assumptions on $\phi$, it is a basic result
that convex gradient descent $M_{+}=M-\alpha\nabla\phi(M)$ has a
linear convergence rate. Under these same assumptions on $\phi$,
it was shown by \citet{NIPS2015_32bb90e8,tu2016low} that nonconvex
gradient descent $X_{+}=X-\alpha\nabla f(X)$ also has a linear convergence
rate within a neighborhood of the global minimizer $X^{\star}$, provided
that the unique unconstrained minimizer $M^{\star}=\arg\min\phi$ is positive
semidefinite $M^{\star}\succeq0$, and has a rank $r^{\star}=\rank(M^{\star})=r$
that matches the search rank. 
\begin{definition}[Gradient Lipschitz]
\label{def:gradlip}The differentiable function $\phi:\R^{n\times n}\to\R$
is said to be \emph{$L_{1}$-gradient Lipschitz} if 
\[
\|\nabla\phi(M+E)-\nabla\phi(M)\|_{F}\le L_{1}\cdot\|E\|_{F}
\]
holds for all $M,E\in\R^{n\times n}$.
\end{definition}
\begin{definition}[Strong convexity]
\label{def:strcvx}The twice differentiable function $\phi:\R^{n\times n}\to\R$
is said to be\emph{ $\mu$-strongly convex} if 
\[
\inner{\nabla^{2}\phi(M)[E]}E\ge\mu\|E\|_{F}^{2}
\]
holds for all $M,E\in\R^{n\times n}$. It is said to be \emph{$(\mu,r)$-restricted
strongly convex} if the above holds for all matrices $M,E\in\R^{n\times n}$ with rank $\leq r$.
\end{definition}
\begin{remark}
Note that \citet{NIPS2015_32bb90e8,tu2016low} actually assumed restricted
strong convexity, which is a \emph{milder} assumption than the usual
notion of strong convexity. In particular, if $\phi$ is $\mu$-strongly
convex, then it is automatically $(\mu,r)$-restricted strongly convex
for all $r\le n$. In the context of low-rank matrix optimization,
many guarantees made for a strongly convex $\phi$ can be trivially
extended to a restricted strongly convex $\phi$, because queries
to $\phi(M)$ and its higher derivatives are only made with respect
to a low-rank matrix argument $M=XX^{T}$.  We also note that in the context of our work, the conditions in Definitions \ref{def:gradlip} and \ref{def:strcvx} actually only needs to be imposed on symmetric matrices $E$. However, for clarity we will follow the standard definition.
\end{remark}
 If $r^{\star}$, the rank of the unconstrained minimizer $M^{\star}\succeq0$,
is strictly less than the search rank $r$, however, nonconvex gradient
descent slows down to a \emph{sublinear} local convergence rate, both
in theory and in practice. We emphasize that the sublinear rate manifests
in spite of the strong convexity assumption on $\phi$; it is purely
a consequence of the fact that $r^{\star}<r$. In this paper, we prove
that \ref{PrecGD} $X_{+}=X-\alpha\nabla f(X)(X^{T}X+\eta I)^{-1}$
has a \emph{linear} local convergence rate, irrespective of the rank
$r^{\star}\le r$ of the minimizer $M^{\star}\succeq0$. Note that a preliminary version of this result restricted to the nonlinear least-squares cost $f(X)=\|\mathcal{A}(XX^{T})-b\|^{2}$ had appeared in a conference paper by the same authors~\citep{zhang2021preconditioned}.
\begin{theorem}[Linear convergence]
\label{thm:local}Let $\phi$ be $L_{1}$-gradient Lipschitz and
$(\mu,2r)$-restricted strongly convex, and let $M^{\star}=\arg\min\phi$
satisfy $M^{\star}=X^{\star}X^{\star T}$ and $r^{\star}=\rank(M^{\star})\le r$.
Define $f(X)\eqdef\phi(XX^{T})$; if $X$ is sufficiently close to
global optimality
\[
f(X)-f(X^{\star})\le\frac{\mu}{2(1+\mu/L_{1})}\cdot\frac{\lambda_{r^{\star}}^{2}(M^{\star})}{2}
\]
and if $\eta$ is bounded from above and below by the distance to
the global optimizer 
\[
C_{\lb}\cdot\|XX^{T}-M^{\star}\|_{F}\le\eta\le C_{\ub}\cdot\|XX^{T}-M^{\star}\|_{F}
\]
then PrecGD $X_{+}=X-\alpha\nabla f(X)(X^{T}X+\eta I)^{-1}$ converges
linearly
\[
f(X_{+})-f(X^{\star})\le\left(1-\alpha\cdot\tau\right)\left[f(X)-f(X^{\star})\right]\text{ for }\alpha\le\min\{1,1/\ell\},
\]
with constants
\begin{gather*}
\tau=\frac{\mu^{2}}{2L_{1}}\left(1+C_{\ub}\cdot\left(1+\sqrt{2}+\frac{L_{1}+\mu}{\sqrt{L_{1}\mu}}\cdot\sqrt{r-r^{\star}}\right)\right)^{-1}\\
\ell=4L_{1}+(2L_{1}+8L_{1}^{2})\cdot C_{\lb}^{-1}+4L_{1}^{3}\cdot C_{\lb}^{-2}.
\end{gather*}
\end{theorem}

\thmref{local} suggests choosing the size of the identity perturbation
$\eta$ in the preconditioner $(X^{T}X+\eta I)^{-1}$ to be within
a constant factor of the error norm $\|XX^{T}-M^{\star}\|_{F}$. This
condition is reminiscent of trust-region methods, which also requires
a similar choice of $\eta$ to ensure fast convergence towards an
optimal point with a degenerate Hessian (see in particular \citealt[Assumption~2.2]{yamashita2001rate}
and also \citealt{fan2005quadratic}). The following provides an explicit
choice of $\eta$ that satisfies the condition in \thmref{local}
in closed-form.
\begin{corollary}[Optimal parameter]
\label{cor:param}Under the same condition as \thmref{local}, we
have
\begin{gather*}
\frac{\mu}{\sqrt{2}}\cdot\|XX^{T}-M^{\star}\|_{F}\le\|\nabla f(X)(X^{T}X)^{-1/2}\|_{F}\le2L_{1}\cdot\|XX^{T}-M^{\star}\|_{F}
\end{gather*}
\end{corollary}
We provide a proof of \thmref{local} and \corref{param} in \secref{local}.

Here, we point out that while \ref{PrecGD} becomes immune to $\kappa=\lambda_1(M^\star)/\lambda(M^\star)$, the condition number of the ground truth $M^\star$, its dependence on $\chi = L_1/\mu$, the condition number of the convex loss function $\phi$, is apparently much worse. Concretely, gradient descent is known to have an iteration complexity of $O(\kappa \cdot \chi \cdot \log(1/\epsilon))$, while \thmref{local} says that \ref{PrecGD} has an iteration complexity of $O(\chi^4 \cdot \log(1/\epsilon))$. (Note that the constants $C_\mathrm{lb}$ and $C_\mathrm{ub}$ in Theorem \ref{thm:local} have ``units'' of $\mu$ and $L_1$ respectively, and therefore 
$\ell = O(L_1\chi^2)$.)

In our numerical experiments, however, both methods have the same dependence on $\chi$. In other words, the iteration complexity of PrecGD is strictly better than GD in practice. Therefore, we believe that with a more refined analysis, our dependence on the latter can also be improved. However, as we discussed in the introduction, the ill-conditioning of $M^\star$ is usually much more serious in practice, so it is the main focus of our work. 

\subsection{Global convergence}

If initialized from an arbitrary point $X_{0}$, gradient descent
can become stuck at a suboptimal point $X$ with small gradient norm
$\|\nabla f(X)\|_{F}\le\delta$. A particularly simple way to escape
such a point is to perturb the current iterate $X$ by a small amount
of random noise. Augmenting classical gradient descent with random
perturbations in this manner, \citet{jin2017escape,jin2021nonconvex}
proved convergence to an $\delta$ approximate second-order stationary
point $X$ satisfying $\|\nabla f(X)\|_{F}\le\delta$ and $\nabla f(X)\succeq-\sqrt{\delta}\cdot I$
in at most $O(1/\delta^{2}\log^{4}(nr/\delta))$ iterations, assuming
that the convex function $\phi$ is gradient Lipschitz and also Hessian
Lipschitz. 
\begin{definition}[Hessian Lipschitz]
\label{def:hesslip}The twice-differentiable function $\phi:\R^{n\times n}\to\R$
is said to be $L_{2}$-\emph{Hessian Lipschitz} if 
\[
\|\nabla^{2}\phi(M)[E]-\nabla^{2}\phi(M')[E]\|_{F}\le L_{2}\cdot\|E\|_{F}\cdot\|M-M'\|_{F}
\]
holds for all $M,M',E\in\R^{n\times n}$. 
\end{definition}
It turns out that certain choices of $f$ satisfy the property that
every $\delta$ approximate second-order stationary point $X$ lies
$\poly(\delta)$-close to a global minimum~\citep{sun2016complete,sun2018geometric,bhojanapalli2016global}.
The following definition is adapted from \citet{ge2015escaping};
see also \citep{ge2017no,jin2017escape}.
\begin{definition}[Strict saddle property]\label{def:strictsaddle}
The function $f$ is said to be $(\epsilon_{g},\epsilon_{H},\rho)$-\emph{strict
saddle} if at least one of the following holds for every $X$:
\begin{itemize}
\item $\|\nabla f(X)\|\ge\epsilon_{g}$;
\item $\lambda_{\min}(\nabla f(X))\le-\epsilon_{H}$;
\item There exists $Z$ satisfying $\nabla f(Z)=0$ and $\nabla^{2}f(Z)\succeq0$
such that $\|X-Z\|_{F}\le\rho$.
\end{itemize}
The function is said to be $(\epsilon_{g},\epsilon_{H},\rho)$-\emph{global
strict saddle} if it is $(\epsilon_{g},\epsilon_{H},\rho)$-strict
saddle, and that all $Z$ that satisfy $\nabla f(Z)=0$ and $\nabla^{2}f(Z)\succeq0$
also satisfy $f(Z)=f(X^{\star})$.
\end{definition}
Assuming that $f(X)\eqdef\phi(XX^{T})$ is $(\epsilon_{g},\epsilon_{H},\rho)$-global
strict saddle, \citet{jin2017escape} used perturbed gradient descent
to arrive within a $\rho$-local neighborhood, after which it takes
gradient descent another $O(1/\lambda_{r}\log(\rho/\epsilon))$ iterations
to converge to $\epsilon$ global suboptimality. Viewing $\epsilon_{g},\epsilon_{H},\rho$
as constants with respect to $\epsilon$, the combined method globally
converges to $\epsilon$ suboptimality in $O(1/\lambda_{r}\log(1/\epsilon))$
iterations, as if $f$ were a smooth and strongly convex function.

If the rank $r^{\star}<r$ is strictly less than the search rank $r$,
however, the global guarantee for gradient descent worsens by an exponential
factor to $\poly(1/\epsilon)$ iterations, due to the loss of local
linear convergence. Inspired by \citet{jin2017escape}, we consider
augmenting \ref{PrecGD} with random perturbations, in order to arrive
within a local neighborhood for which our linear convergence result
(\thmref{local}) becomes valid. Concretely, we consider \textit{perturbed}
PrecGD or PPrecGD, defined as 
\begin{equation}
X_{k+1}=X_{k}-\alpha[\nabla f(X_{k})(X_{k}^{T}X_{k}+\eta I)^{-1}+\zeta_{k}],\tag{PPrecGD}\label{PPrecGD}
\end{equation}
where we \emph{fix} the value of the regularization parameter $\eta>0$
and apply a random perturbation $\zeta_{k}$ whenever the gradient
norm becomes small: 
\[
\begin{cases}
\zeta_{k}\sim\mathbb{B}(\beta) & \|\nabla f(X_{k})(X_{k}^{T}X_{k}+\eta_{\mathrm{fix}}I)^{-1/2}\|_{F}\le\epsilon, \text{ and } k \ge k_{\mathrm{last}} + \mathcal{T},\\
\zeta_{k}=0 & \text{otherwise.}
\end{cases}
\]

Here, $k_{\mathrm{last}}$ denotes the last iteration index for which a random perturbation was made. The condition $k \ge k_{\mathrm{last}} + \mathcal{T}$ ensures that the algorithm takes at least $\mathcal{T}$ iterations before making a new perturbation.

The algorithm parameters are the step-size $\alpha>0$, the perturbation
radius $\beta>0$, the period of perturbation $\mathcal{T}$, the
fixed regularization parameter $\eta>0$, and the accuracy threshold
$\epsilon>0$. We show that \ref{PPrecGD} is guaranteed to converge
to an $\epsilon$ second-order stationary point, provided that the
following sublevel set of $\phi$ (which contains all iterates $X_{0},X_{1},\dots,X_{k}$)
is bounded:
\begin{gather*}
\mathcal{X}=\{X\in\R^{n\times r}:\phi(XX^{T})\leq\phi(X_{0}X_{0}^{T})+2\sqrt{\|X_{0}\|_{F}^{2}+\eta}\cdot\alpha\beta\epsilon\}.
\end{gather*}
Let $\Gamma=\max_{X\in\mathcal{X}}\|X\|_{F}$. Below, the notation
$\tilde{O}(\cdot)$ hides polylogarithmic factors in the algorithm
and function parameters $\eta,$ $L_{1},$ $L_{2},$ $\Gamma,$ the
dimensionality $n,$ $r,$ the final accuracy $1/\epsilon,$ and the
initial suboptimality $f(X_{0})-f(X^{\star})$. The proof is given in \secref{global}.

\begin{theorem}[Approximate second-order optimality]
\label{cor_pprecgd} Let $\phi$ be $L_{1}$-gradient and $L_{2}$-Hessian
Lipschitz. Define $f(X)\eqdef\phi(XX^{T})$ and let $X^{\star}=\arg\min f$.
For any $\epsilon=O(1/(L_{d}\sqrt{\Gamma^{2}+\eta}))$ and with an
overwhelming probability, \ref{PPrecGD} with parameters $\alpha=\eta/\ell_{1},$
$\beta=\tilde{O}(\epsilon/L_{d}),$ and $\mathcal{T}=\tilde{O}(L_{1}\Gamma^{2}/(\eta\sqrt{L_{d}\epsilon}))$
converges to a point $X$ that satisfies 
\begin{equation}
\inner{\nabla f(X)}V\le\epsilon\cdot\|V\|_{X,\eta},\quad\inner{\nabla^{2}f(X)[V]}V\ge-\sqrt{L_{d}\epsilon}\cdot\|V\|_{X,\eta}^{2}\quad\text{for all }V,\label{eq_SOC}
\end{equation}
where $\|V\|_{X,\eta}\eqdef\|V(X^{T}X+\eta I)^{1/2}\|_{F}$ in at
most 
\[
\tilde{O}\left(\frac{\ell_{1}\cdot[f(X_{0})-f(X^{\star})]}{\eta^{2}\cdot\epsilon^{2}}\right)\text{ iterations}
\]
where $L_{d}={5\max\{\ell_{2},2\Gamma\ell_{1}\sqrt{\Gamma^{2}+\eta}\}}/{\eta^{2.5}}$,
$\ell_{1}=9\Gamma^{2}L_{1}$, $\ell_{2}=(4\Gamma+2)L_{1}+4\Gamma^{2}L_{2}$. 
\end{theorem}

In Theorem~\ref{cor_pprecgd}, the total number of iterations it takes to converge to an approximate second-order stationary point depends additionally on $\Gamma$, the maximal radius of the iterates. The factor of $\Gamma$ is largely an artifact of the proof technique, and does not appear in the practical performance of the algorithm. Previous work on naive gradient descent (see Theorem 8 of \cite{jin2017escape}) also introduced a similar factor. 
Clearly, $\Gamma$ is finite if $\phi(\cdot)$ is coercive, i.e., it diverges to $\infty$ as $\|X\|_F \to \infty.$ (See Lemma~\ref{lem_Lip}.) In many statistical and machine learning applications, the loss function is purposely chosen to be coercive.

Assuming that $f$ is $(\epsilon_{g},\epsilon_{H},\rho)$-global strict
saddle, we use \ref{PPrecGD} to arrive within a $\rho$-local neighborhood
of the global minimum, and then switch to \ref{PrecGD} for another
$O(\log(\rho/\epsilon))$ iterations to converge to $\epsilon$ global
suboptimality due to \thmref{local}. (If the search rank is overparameterized
$r>r^{\star}$, then the switching condition can be explicitly detected
using \propref{local_cert} in the following section.) Below, we use
$\tilde{O}(\cdot)$ to additionally hide polynomial factors in $L_{1},L_{2},\mu,\Gamma$,
while exposing all dependencies on final accuracy $1/\epsilon$ and
the smallest nonzero eigenvalue $\lambda_{r^{\star}}(M^{\star})$.
\begin{corollary}[Global convergence]
\label{cor:global}Let $\phi$ be $L_{1}$-gradient Lipschitz and
$L_{2}$-Hessian Lipschitz and $(\mu,2r)$-restricted strongly convex,
and let $M^{\star}=\arg\min\phi$ satisfy $M^{\star}=X^{\star}X^{\star T}$
and $r^{\star}=\rank(M^{\star})<r$. Suppose that $f(X)\eqdef\phi(XX^{T})$
satisfies $(\epsilon_{g},\epsilon_{H},\rho)$-global strict saddle
with 
\[
\frac{1}{\Gamma^{2}}\cdot\epsilon_{g}+\epsilon_{H}+4L_{1}\cdot\rho^{2}\le\frac{\mu}{1+\mu/L_{1}}\cdot\frac{\lambda_{r^{\star}}^{2}(M^{\star})}{\tr(M^{\star})},
\]
Then, do the following:
\begin{enumerate}
\item (Global phase) Run \ref{PPrecGD} with a fixed $\eta=\eta_{0}\le\Gamma^{2}$
until $\|\nabla f(X_{k})\|_{F}\le\epsilon_{g},$ and $\lambda_{\min}(\nabla^{2}f(X_{k}))\ge-\epsilon_{H},$
and $\lambda_{\min}(X_{k}^{T}X_{k})\le\rho$;
\item (Local phase) Run \ref{PrecGD} with $\eta=\|\nabla f(X)(X^{T}X)^{-1}\|_{F}$
and $\alpha=1/\ell$. 
\end{enumerate}
The combined algorithm arrives at a point $X$ satisfying $f(X)-f(X^{\star})\le\epsilon$
in at most
\[
\tilde{O}\left(\frac{f(X_{0})-f(X^{\star})}{\eta_{0}^{2}}\cdot\left(\frac{1}{\epsilon_{g}^{2}}+\frac{1}{\epsilon_{H}^{4}}\right)+\log\left(\frac{\lambda_{r^{\star}}^{2}(M^{\star})}{\epsilon}\right)\right)\text{ iterations.}
\]
\end{corollary}

Corollary \ref{cor:global} follows by running \ref{PPrecGD} until it arrives at an ($\epsilon_{g},\epsilon_{H}$)-second-order stationary point $X_k$ (Theorem~\ref{cor_pprecgd}), and then using the global strict saddle property (Definition~\ref{def:strictsaddle}) to argue that $X_k$ is also rank deficient, with $\lambda_{\min}(X_k^T X_k) \leq \rho^2$ via Weyl's inequality. It follows from second-order optimality and rank deficiency that $X_k$ is sufficiently close to global optimality (Proposition \ref{prop:euclid_cert} below), and therefore switching to \ref{PrecGD} results in linear convergence (Theorem~\ref{thm:local}). Viewing $\epsilon_{g},\epsilon_{H},\rho$ as constants, the combined method globally converges to $\epsilon$ suboptimality
in $O(\log(1/\epsilon))$ iterations, as if $f$ were a smooth and
strongly convex function, even in the overparameterized regime with
$r>r^{\star}$.

Here, we point out that the strict saddle property (Definition~\ref{def:strictsaddle}) is usually defined for a second-order point measured in the Euclidean norm $\|\cdot\|_F$, but that Theorem~\ref{cor_pprecgd} proves convergence to a second-order point measured in the local norm $\|\cdot\|_{X,\eta}$. Clearly, for a fixed $\eta=\eta_{\mathrm{fix}}$, the two notions are equivalent up to a conversion factor. In deciding when to switch from \ref{PPrecGD} to \ref{PrecGD}, Corollary \ref{cor:global} uses the Euclidean norm (via Proposition \ref{prop:euclid_cert} below) to remain consistent with the strict saddle property. In practice, however, it should be less conservative to decide using the local norm (via Proposition \ref{prop:local_cert} below), as this is the preferred norm that the algorithm tries to optimize.

\section{\label{sec:Cert}Certifying Global Optimality via Rank Deficiency}

We now turn to the problem of certifying the global optimality of
an $X$ computed using \ref{PrecGD} by appealing to its rank deficiency.
We begin by rigorously stating the global optimality guarantee previously
quoted in (\ref{eq:certF}). The core argument actually dates back
to \citet[Theorem~4.1]{burer2005local} (and has also appeared in
\citet{journee2010low} and \citet{boumal2016non,boumal2020deterministic})
but we restate it here with a shorter proof in order to convince the
reader of its correctness. 
\begin{proposition}[Certificate of global optimality]
\label{prop:euclid_cert}Let $\phi$ be twice differentiable and
convex and let $f(X)\eqdef\phi(XX^{T})$. If $X$ satisfies $\lambda_{\min}(X^{T}X)\le\epsilon_{\lambda}$
and
\[
\inner{\nabla f(X)}V\le\epsilon_{g}\cdot\|V\|_{F},\quad\inner{\nabla^{2}f(X)[V]}V\ge-\epsilon_{H}\cdot\|V\|_{F}^{2}\quad\text{for all }V,
\]
where $\epsilon_{g},\epsilon_{H},\epsilon_{\lambda}\ge0$, then $X$
has suboptimality
\begin{gather*}
f(X)-f(X^{\star})\le C_{g}\cdot\epsilon_{g}+C_{H}\cdot\epsilon_{H}+C_{\lambda}\cdot\epsilon_{\lambda}.
\end{gather*}
where $C_{g}=\frac{1}{2}\|X\|_{F}$ and $C_{H}=\frac{1}{2}\|X^{\star}\|_{F}^{2}$
and $C_{\lambda}=2\|\nabla^{2}\phi(XX^{T})\|\|X^{\star}\|_{F}^{2}$. 
\end{proposition}
\begin{proof}
Let $(u_{r},v_{r},\sigma_{r})$ the $r$-th singular value triple
of $X$, i.e. we have $Xv_{r}=\sigma_{r}u_{r}$ with $\|v_{r}\|=\|u_{r}\|=1$
and $\sigma_{r}^{2}=\lambda_{\min}(X^{T}X)$. For $M=XX^{T}$ and
$M^{\star}=X^{\star}X^{\star T}$, the convexity of $\phi$ implies
$\phi(M^{\star})\ge\phi(M)+\inner{\nabla\phi(M)}{M^{\star}-M}$ and
therefore 
\[
f(X)-f(X^{\star})=\phi(M)-\phi(M^{\star})\le\inner{\nabla\phi(M)}M-\lambda_{\min}[\nabla\phi(M)]\cdot\tr(M^{\star}).
\]
Substituting $V=X$ into the first-order optimality conditions, as
in
\begin{gather*}
\inner{\nabla f(X)}V=2\inner{\nabla\phi(XX^{T})X}X\le\epsilon_{g}\|V\|_{F}=\epsilon_{g}\cdot2C_{g}
\end{gather*}
yields $\inner{\nabla\phi(M)}M\le C_{g}\cdot\epsilon$. Substituting
$V=yv_{r}^{T}$ with an arbitrary $y\in\R^{n}$ with $\|y\|=1$ into
the second-order conditions yields
\begin{align*}
\inner{\nabla^{2}f(X)[V]}V & \le2\inner{\nabla\phi(XX^{T})}{VV^{T}}+\|\nabla^{2}\phi(XX^{T})\|\cdot\|XV^{T}+VX^{T}\|_{F}^{2}\\
 & =2y^{T}\nabla\phi(XX^{T})y+\|\nabla^{2}\phi(XX^{T})\|\cdot\sigma_{r}^{2}\cdot\|u_{r}y^{T}+yu_{r}^{T}\|_{F}^{2}
\end{align*}
which combined with $\inner{\nabla^{2}f(X)[V]}V\ge -\epsilon_{H}\|V\|_{F}^{2}=-\epsilon_{H}$
gives 
\[
-y^{T}\nabla\phi(XX^{T})y\le\frac{1}{2}\epsilon_{H}+2\|\nabla^{2}\phi(XX^{T})\|\cdot\sigma_{r}^{2}
\]
and therefore $-\lambda_{\min}[\nabla\phi(M)]\le\frac{1}{2}\epsilon_{H}+C_{\lambda}\cdot\lambda_{\min}(X^{T}X)$.
\end{proof}

\propref{euclid_cert} can also be rederived with respect to the local
norm in Theorem~\ref{cor_pprecgd}. We omit the proof of the following
as it is essentially identical to that of \propref{euclid_cert}.
\begin{proposition}[Global certificate in local norm]
\label{prop:local_cert}Let $\phi$ be twice differentiable and convex
and let $f(X)\eqdef\phi(XX^{T})$. If $X$ satisfies $\lambda_{\min}(X^{T}X)\le\epsilon_{\lambda}$
and
\[
\inner{\nabla f(X)}V\le\epsilon_{g}\cdot\|V\|_{X,\eta},\quad\inner{\nabla^{2}f(X)[V]}V\ge-\epsilon_{H}\cdot\|V\|_{X,\eta}^{2}\quad\text{for all }V,
\]
where $\|V\|_{X,\eta}\eqdef\|V(X^{T}X+\eta I)^{-1/2}\|_{F}$ and $\epsilon_{g},\epsilon_{H},\epsilon_{\lambda}\ge0$,
then $X$ has suboptimality
\begin{gather*}
f(X)-f(X^{\star})\le C_{g}\cdot\epsilon_{g}+C_{H}\cdot\epsilon_{H}\cdot(\epsilon_{\lambda}+\eta)+C_{\lambda}\cdot\epsilon_{\lambda}
\end{gather*}
where $C_{g}=\frac{1}{2}\sqrt{\|X^{T}X\|_{F}^{2}+\eta\|X\|_{F}^{2}}$
and $C_{H}=\frac{1}{2}\|X^{\star}\|_{F}^{2}$ and $C_{\lambda}=2\|\nabla^{2}\phi(XX^{T})\|\|X^{\star}\|_{F}^{2}$. 
\end{proposition}
\begin{remark}
Under the same conditions as Theorem~\ref{cor_pprecgd}, it
follows that $\|X\|_{F},\|X^{\star}\|_{F}\le\Gamma$ and $\|\nabla^{2}\phi(XX^{T})\|\le L_{1}$. 
\end{remark}

Now let us explain how we can use \ref{PrecGD} to solve an instance
of (\ref{eq:ncvx}) to an $X$ with provable global optimality via
either \propref{euclid_cert} or \propref{local_cert}. First, after
overparameterizing the search rank $r>r^{\star}$, we run \ref{PPrecGD}
with a fixed parameter $\eta>0$ until we reach the neighborhood of
a global minimizer $X^{\star}$ where \thmref{local} holds. The following
result says that this condition can always be detected by checking
\propref{euclid_cert}. Afterwards, we can switch to \ref{PrecGD}
with a variable parameter $\eta=\|\nabla f(X)(X^{T}X)^{-1}\|_{F}$
and expect linear convergence towards to global minimum.

\begin{corollary}[Certifiability of near-global minimizers]
Under the same condition as \thmref{local}, let $X$ satisfy $f(X)-f(X^{\star})\le\frac{1}{2}\mu\epsilon^{2}$.
If $r>r^{\star}$, then $X$ also satisfies
\[
\|\nabla f(X)\|_{F}\le2L_{1}\|X\|_{F}\cdot\epsilon,\quad\lambda_{\min}(\nabla^{2}f(X))\ge-L_{1}\cdot\epsilon,\quad\lambda_{\min}(X^{T}X)\le\epsilon.
\]
\end{corollary}
\begin{proof}
It follows immediately from $\frac{1}{2}\mu\epsilon^{2}\ge f(X)-f(X^{\star})\ge\frac{1}{2}\mu\|XX^{T}-M^{\star}\|_{F}$
in \lemref{riperr}, which yields $\|\nabla\phi(XX^{T})\|_{F}\le L_{1}\epsilon$
via gradient Lipschitzness and $\lambda_{\min}(X^{T}X)=\lambda_{r}(XX^{T})\le\epsilon$
via Weyl's inequality.  Finally, to see why the second statement holds, note that the Hessian of $f(X)$ can be written as 
\[
\inner{V}{\nabla^{2} f(X)[V]}= \inner{V}{\nabla \phi (XX^{T}) V}+ \inner{V}{\nabla^{2} \phi (XX^{T})\left[X V^{T}+V X^{T}\right] X}.
\]
Since  $\phi$ is convex, the second term is always non-negative. Thus the second statement follows from the fact that $\|\nabla \phi(XX^T)\|_F \leq L_1\epsilon$.
\end{proof}

On the other hand, if \ref{PPrecGD} becomes stuck within a neighborhood
of a spurious local minimum or nonstrict saddle point $Z$, then this
fact can also be explicitly detected by numerically evaluating the
rank deficiency parameter $\epsilon_{\lambda}=\lambda_{\min}(X^{T}X)$.
Note that if $\|X-Z\|_{F}\le\rho$, then it follows from Weyl's inequality
that $\lambda_{\min}^{1/2}(Z^{T}Z)\ge\lambda_{\min}^{1/2}(X^{T}X)-\rho$. 
\begin{corollary}[Spurious points have high rank]
Under the same condition as \thmref{local}, let $Z$ satisfy $\nabla f(Z)=0$
and $\nabla^{2}f(Z)\succeq0$. If $r>r^{\star}$, then we have
\begin{gather*}
f(Z)>f(X^{\star})\quad\iff\quad\lambda_{\min}(Z^{T}Z)>\frac{\mu}{4\cdot(L_{1}+\mu)}.\frac{\lambda_{r^{\star}}^{2}(M^{\star})}{\tr(M^{\star})}.
\end{gather*}
\end{corollary}
\begin{proof}
It follows from \thmref{local} that any point $Z$ that satisfies
$\nabla f(Z)=0$ within the neighborhood $f(Z)-f(X^{\star})\le R=\frac{\mu}{2\cdot(1+\mu/L_{1})}\lambda_{r^{\star}}^{2}(M^{\star})$
must actually be globally optimal $f(Z)=f(X^{\star})$. Therefore,
any suboptimal $Z$ with $\nabla f(Z)=0$ and $\nabla^{2}f(Z)\succeq0$
and $f(Z)>f(X^{\star})$ must lie outside of this neighborhood, as
in $f(Z)-f(X^{\star})>R$. It follows from \propref{euclid_cert}
that $Z$ must satisfy:
\begin{align*}
R<f(Z)-f(X^{\star}) & \le C_{\lambda}\cdot\lambda_{\min}(Z^{T}Z)\le2L_{1}\tr(M^{\star})\cdot\lambda_{\min}(Z^{T}Z).
\end{align*}
Conversely, if $r>r^{\star}=\rank(M^{\star})$, then $Z$ is globally
optimal $f(Z)=f(X^{\star})$ if and only if it is rank deficient,
as in $\lambda_{\min}(Z^{T}Z)=0$.  
\end{proof}

Finally, we turn to the practical problem of evaluating the parameters
in \propref{euclid_cert}. It is straightforward to see that it costs
$O(nr^{2}+r^{3})$ time to compute the gradient norm term $\epsilon_{g}=\|\nabla f(X)\|_{F}$
and the rank deficiency term $\lambda_{\min}(X^{T}X)$, after computing
the nonconvex gradient $\nabla f(X)$ in $n\cdot\poly(r)$ time via
the gradient oracle. To compute the Hessian curvature $\epsilon_{H}=-\lambda_{\min}[\nabla^{2}f(X)]$
without explicitly forming the $nr\times nr$ Hessian matrix, we suggest
using a shifted power iteration
\[
V_{k+1}=\tilde{V}_{k}/\|\tilde{V}_{k}\|_{F}\qquad\text{ where }\tilde{V}_{k}=\lambda V_{k}-\nabla^{2}f(X)[V_{k}],
\]
where we roughly choose the shift parameter $\lambda$ so that $\lambda\ge\lambda_{\max}[\nabla^{2}f(X)]$
and approximate each Hessian matrix-vector product using finite differences
\[
\nabla^{2}f(X)[V]\approx\frac{1}{t}[\nabla f(X+tV)-\nabla f(X)].
\]
The Rayleigh quotient converges linearly, achieving $\delta$-accuracy
in $O(\log(1/\delta))$ iterations \citep{kaniel1966estimates,paige1971computation,saad1980rates}.
Each iteration requires a single nonconvex gradient evaluation $\nabla f(X+tV)$,
which we have assumed to cost $n\cdot\poly(r)$ time. Technically,
linear convergence to $\lambda_{\min}(\nabla^{2}f(X))$ requires the
eigenvalue to be simple and well separated. If instead the eigenvalue
has multiplicity $b>1$ (or lies within a well-separated cluster of
$b$ eigenvalues), then we use a block power iteration with block-size
$b$ to recover linear convergence to $\lambda_{\min}[\nabla^{2}f(X)]$,
with an increased per-iteration cost of $O(nb)$ time~\citep{saad1980rates}.

\section{Preliminaries}

Our analysis will assume that $\phi$ is $L$-gradient Lipschitz and
$(\mu,2r)$-restricted strongly convex, meaning that
\begin{equation}
\mu\|E\|_{F}^{2}\le\inner{\nabla^{2}\phi(M)[E]}E\le L\|E\|_{F}^{2}\label{eq:rip}
\end{equation}
in which the lower-bound is restricted over matrices $M,E$ whose
$\rank(M)\le2r$ and $\rank(E)\le2r$. (See \defref{gradlip} and
\defref{strcvx}.) The purpose of these assumptions is to render the
function $\phi$ well-conditioned, so that its suboptimality can serve
as a good approximation for the matrix error norm
\[
f(X)-f(X^{\star})\approx\|XX^{T}-M^{\star}\|_{F}^{2}\text{ up to a constant.}
\]
In turn, we would also expect the nonconvex gradient $\nabla f(X)$
to be closely related to the gradient of the matrix error norm $\|XX^{T}-M^{\star}\|_{F}^{2}$
taken with respect to $X$. To make these arguments rigorous, we will
need the following lemma from \citet[Proposition~2.1]{li2019non}.
The proof is a straightforward extension of \citet[Lemma~2.1]{candes2008restricted}.
\begin{lemma}[Preservation of inner product]
\label{lem:ripinner}Let $\phi$ be $L$-gradient Lipschitz and
$(\mu,r)$-restricted strongly convex. Then, we have
\[
\left|\frac{2}{\mu+L}\inner{\nabla^{2}\phi(M)[E]}F-\inner EF\right|\le\frac{L-\mu}{L+\mu}\|E\|_{F}\|F\|_{F}
\]
for all $\rank(M)\le r$ and $\rank(E+F)\le r$.
\end{lemma}
\begin{lemma}[Preservation of error norm]
\label{lem:riperr}Let $\phi$ be $L$-gradient Lipschitz and
$(\mu,2r)$-restricted strongly convex. Let $M^{\star}=\arg\min\phi$
satisfy $M^{\star}\succeq0$ and $\rank(M^{\star})\le r$. Define
$f(X)\eqdef\phi(XX^{T})$ and let $X^{\star}=\arg\min f$. Then $f$
satisfies
\[
\frac{1}{2}\mu\|XX^{T}-M^{\star}\|_{F}^{2}\leq f(X)-f(X^{\star})\leq\frac{1}{2}L\|XX^{T}-M^{\star}\|_{F}^{2}
\]
for all $\rank(M)\le r$.
\end{lemma}
\begin{lemma}[Preservation of error gradient]
\label{lem:prelimgrad}Under the same conditions as \lemref{riperr},
we have \begin{subequations}
\begin{align}
\|\nabla f(X)\|_{F} & \ge\nu\cdot\max_{\|Y\|_{F}=1}\left[\inner E{XY^{T}+YX^{T}}-\delta\|E\|_{F}\|XY^{T}+YX^{T}\|_{F}\right],\label{eq:gradbnd}
\end{align}
\end{subequations} where $\nu=\frac{1}{2}(\mu+L)$ and $\delta=\frac{L-\mu}{L+\mu}$
and $E=XX^{T}-M^{\star}$.
\end{lemma}
\begin{proof}
Let $Y^{\star}$ denote a maximizer for the right-hand side of (\ref{eq:gradbnd}),
and let $\Pi$ denote the orthogonal projector onto 
\[
\range(X)+\range(X^{\star})=\{Xu+X^{\star}v:u,v\in\R^{r}\}.
\]
(Explicitly, $\Pi=QQ^{T}$ where $Q=\orth([X,X^{\star}])$.) We claim
that the projected matrix $Y=\Pi Y^{\star}$ is also a maximizer.
Note that, by the definition of $\Pi$, we have $X=\Pi X$ and $E=\Pi E\Pi$.
It follows that 
\begin{align*}
\inner{XY^{T}+YX^{T}}E & =\inner{\Pi\left[XY^{\star T}+Y^{\star}X^{T}\right]\Pi}E\\
 & =\inner{XY^{\star T}+Y^{\star}X^{T}}{\Pi E\Pi}=\inner{XY^{\star T}+Y^{\star}X^{T}}E,
\end{align*}
and 
\[
\|XY^{T}+YX^{T}\|_{F}=\|\Pi\left[XY^{\star T}+Y^{\star}X^{T}\right]\Pi\|_{F}\le\|XY^{\star T}+Y^{\star}X^{T}\|_{F},
\]
and $\|Y\|_{F}=\|\Pi Y^{\star}\|_{F}\le\|Y^{\star}\|_{F}\le1.$ Therefore,
we conclude that $Y$ is feasible and achieves the same optimal value
as the maximizer $Y^{\star}$. 

Now, let $Y^{\star}=\Pi Y^{\star}$ without loss of generality due
to the above. We evoke the lower-bound in \lemref{ripinner} and $\nabla\phi(M^{\star})=0$
to obtain the following
\begin{align*}
\inner{\nabla f(X)}{Y^{\star}} & =\inner{\nabla\phi(XX^{T})-\nabla\phi(M^{\star})}{XY^{\star T}+Y^{\star}X^{T}}\\
 & =\int_{0}^{1}\inner{\nabla^{2}\phi(M^{\star}+tE)[E]}{XY^{\star T}+Y^{\star}X^{T}}\d t\\
 & \ge\nu\cdot[\inner E{XY^{\star T}+Y^{\star}X^{T}}-\delta\cdot\|E\|_{F}\cdot\|XY^{\star T}+Y^{\star}X^{T}\|_{F}]
\end{align*}
where we crucially note that $\rank(XY^{\star T}+Y^{\star}X^{T}\pm E)\le2r$
because $XY^{\star T}=\Pi XY^{\star T}\Pi$ and $E=\Pi E\Pi$ and
$\rank(\Pi)\le\rank(X)+\rank(X^{\star})\le2r$. We conclude that (\ref{eq:gradbnd})
is true, because $Y^{\star}$ is a maximizer for the right-hand side
of (\ref{eq:gradbnd}).
\end{proof}

\section{\label{sec:sublin}Local Sublinear Convergence of Gradient Descent}

In order to explain why PrecGD is able to maintain linear convergence
in the overparameterized regime $r>r^{\star}$, we must first understand
why gradient descent slows down to sublinear convergence. In this
paper, we focus on a property known as \emph{gradient dominance} \citep{polyak1963gradient,nesterov2006cubic}
or the \emph{Polyak-Łojasiewicz} inequality~\citep{lojasiewicz1963propriete},
which is a simple, well-known sufficient condition for linear convergence.
Here, we use the degree-2 definition from \citet[Definition~3]{nesterov2006cubic}.
\begin{definition}
A function $f$ is said to satisfy \emph{gradient dominance} (in the
Euclidean norm) if it attains a global minimum $f^{\star}=f(X^{\star})$
at some point $X^{\star}$ and we have
\begin{equation}
f(X)-f^{\star}\le R\quad\implies\quad\tau\cdot[f(X)-f^{\star}]\le\frac{1}{2}\|\nabla f(X)\|_{F}^{2}\label{eq:pl1}
\end{equation}
for a radius constant $R>0$ and dominance constant $\tau>0$.
\end{definition}
If the function $f$ is additionally $\ell$-gradient Lipschitz, as
in
\begin{align*}
f(X+\alpha V) & \le f(X)+\alpha\inner{\nabla f(X)}V+\frac{\ell}{2}\alpha^{2}\|V\|_{F}^{2},
\end{align*}
then it follows that the amount of progress made by an iteration of
gradient descent $X_{+}=X-\alpha\nabla f(X)$ is proportional to the
gradient norm squared:
\begin{align*}
f(X_{+}) & \le f(X)-\alpha\inner{\nabla f(X)}{\nabla f(X)}+\frac{\ell}{2}\alpha^{2}\|\nabla f(X)\|_{F}^{2}\\
 & =f(X)-\alpha\left(1-\frac{\ell}{2}\alpha\right)\|\nabla f(X)\|_{F}^{2}\\
 & \le f(X)-\frac{\alpha}{2}\|\nabla f(X)\|_{F}^{2}\qquad\text{ with }\alpha\le\frac{1}{\ell}.
\end{align*}
The purpose of gradient dominance (\ref{eq:pl1}), therefore, is to
ensure that the gradient norm remains large enough for good progress
to be made. Substituting (\ref{eq:pl1}) yields
\begin{align}
f(X_{+})-f^{\star} & \le(1-\tau\alpha)\cdot(f(X)-f^{\star})\qquad\text{ with }\alpha\le\frac{1}{\ell}.\label{eq:gdconv-1}
\end{align}
Starting from an initial point $X_{0}$ within the radius $f(X_{0})-f^{\star}\le R$,
it follows that gradient descent $X_{k+1}=X_{k}-\frac{1}{\ell}\nabla f(X_{k})$
converges to an $\epsilon$-suboptimal point $X_{k}$ that satisfies
$f(X_{k})-f^{\star}\le\epsilon$ in at most $k=O((\tau/\ell)\log(R/\epsilon))$
iterations.

The nonconvex objective $f(X)\eqdef\phi(XX^{T})$ associated with
a well-conditioned convex objective $\phi$ is easily shown to satisfy
gradient dominance (\ref{eq:pl1}) in the exactly parameterized regime
$r=r^{\star}$, for example by manipulating existing results on local
strong convexity~\citep{sun2016guaranteed}~\citep[Lemma~4]{chi2019nonconvex}.
In the overparameterized case $r>r^{\star}$, however, local strong
convexity is lost, and gradient dominance can fail to hold. 

The goal of this section is to elucidate this failure mechanism, in order to motivate the ``fix'' encompassed
by PrecGD. We begin by considering a specific instance
of the following nonconvex objective $f_{0}$, corresponding to a
perfectly conditioned quadratic objective $\phi_{0}$:
\begin{equation}
f_{0}(X)\eqdef\phi_{0}(XX^{T})=f_{0}^{\star}+\frac{1}{2}\|XX^{T}-M^{\star}\|_{F}^{2}.\label{eq:f0def}
\end{equation}
The associated gradient norm has a variational characterization
\begin{align}
\|\nabla f_{0}(X)\|_{F}=\max_{\|Y\|_{F}=1}\inner{\nabla f_{0}(X)}Y & =\max_{\|Y\|_{F}=1}\inner{XX^{T}-M^{\star}}{XY^{T}+YX^{T}},\label{eq:vari}
\end{align}
which we can interpret as a \emph{projection} from the error vector
$XX^{T}-M^{\star}$ onto the linear subspace $\{XY^{T}+YX^{T}:Y\in\R^{n\times r}\}$,
as in
\begin{align}
\|\nabla f_{0}(X)\|_{F} & =\|XX^{T}-M^{\star}\|_{F}\|XY^{\star T}+Y^{\star}X^{T}\|_{F}\cos\theta.\label{eq:gd-proj}
\end{align}
Here, the incidence angle $\theta$ is defined 
\begin{equation}
\cos\theta=\max_{Y\in\R^{n\times r}}\frac{\inner{XX^{T}-M^{\star}}{XY^{T}+YX^{T}}}{\|XX^{T}-M^{\star}\|_{F}\|XY^{T}+YX^{T}\|_{F}},\label{eq:cosdef}
\end{equation}
and $Y^{\star}$ is a corresponding maximizer for (\ref{eq:cosdef})
scaled to satisfy $\|Y^{\star}\|_{F}=1$. Substituting the suboptimality $f_{0}(X)-f_{0}^{\star}$
in place of the error norm $\|XX^{T}-M^{\star}\|_{F}$ via \lemref{riperr}
yields a critical identity:
\begin{equation}
\frac{1}{2}\|\nabla f_{0}(X)\|_{F}^{2}=\|XY^{\star T}+Y^{\star}X^{T}\|_{F}^{2}\cdot\cos^{2}\theta\cdot[f_{0}(X)-f_{0}^{\star}].\label{eq:gd-cos}
\end{equation}
The loss of gradient dominance  implies that at
least one of the two terms $\|XY^{\star T}+Y^{\star}X^{T}\|_{F}$
and $\cos\theta$ in (\ref{eq:gd-cos}) must decay to zero as gradient
descent makes progress towards the solution. 

The term $\cos\theta$ in (\ref{eq:gd-cos}) becomes small if the
error $XX^{T}-M^{\star}$ becomes \emph{poorly aligned} to the linear
subspace $\{XY^{T}+YX^{T}:Y\in\R^{n\times r}\}$. In fact, this failure
mechanism cannot occur within a sufficiently small neighborhood of
the ground truth, due to the following key lemma. Its proof is technical,
and is deferred to \appref{pf_align}. 
\begin{lemma}[Basis alignment]
\label{lem:align}For $M^{\star}\in\R^{n\times n},$ $M^{\star}\succeq0$,
suppose that $X\in\R^{n\times r}$ satisfies $\|XX^{T}-M^{\star}\|_{F}\le\rho\lambda_{r^{\star}}(M^{\star})$
with $r^{\star}=\rank(M^{\star})$ and $\rho\le1/\sqrt{2}$. Then
the incidence angle $\theta$ defined in (\ref{eq:cosdef}) satisfies
\begin{equation}
\sin\theta=\frac{\|(I-XX^{\dagger})M^{\star}(I-XX^{\dagger})\|_{F}}{\|XX^{T}-M^{\star}\|_{F}}\le\frac{1}{\sqrt{2}}\frac{\rho}{\sqrt{1-\rho^{2}}}\label{eq:ratio-1-1}
\end{equation}
where $\dagger$ denotes the pseudoinverse.
\end{lemma}
The term $\|XY^{\star T}+Y^{\star}X^{T}\|_{F}$ in (\ref{eq:gd-cos})
becomes small if the error vector $XX^{T}-M^{\star}$ concentrates
itself within the \emph{ill-conditioned directions} of $\{XY^{T}+YX^{T}:Y\in\R^{n\times r}\}$.
In particular, if $XX^{T}-M^{\star}$ lies entirely with the subspace
$\{u_{r}y^{T}+yu_{r}^{T}:y\in\R^{n}\}$ associated with the $r$-th
eigenpair $(\lambda_{r},u_{r})$ of the matrix $XX^{T}$, and if the
corresponding eigenvalue $\lambda_{r}=\lambda_{\min}(X^{T}X)$ decays
towards zero, then the term $\|XY^{\star T}+Y^{\star}X^{T}\|_{F}$
must also decay towards zero. The following lemma provides a lower-bound
on $\|XY^{\star T}+Y^{\star}X^{T}\|_{F}$ by accounting for this mechanism.
\begin{lemma}[Basis scaling]
\label{lem:scal}For any $H\in\R^{n\times n}$ and $X\in\R^{n\times r}$,
there exists a choice of $Y^{\star}=\arg\max_{Y}\inner H{XY^{T}+YX^{T}}$
such that 
\[
\|XY^{\star T}+Y^{\star}X^{T}\|_{F}^{2}\ge2\cdot\lambda_{k}(XX^{T})\cdot\|Y^{\star}\|_{F}^{2}\qquad\text{ where }k=\rank(X).
\]
\end{lemma}
\begin{proof}
Define $\JJ:\R^{n\times r}\to\S^{n}$ such that $\JJ(Y)=XY^{T}+YX^{T}$
for all $Y$. We observe that $Y^{\star}=\JJ^{\dagger}(H)$ where
$\dagger$ denotes the pseudoinverse. Without loss of generality,
let $X=[\Sigma;0]$ where $\Sigma=\mathrm{diag}(\sigma_{1},\dots,\sigma_{r})$
and $\sigma_{1}\ge\cdots\ge\sigma_{r}\ge0$. Then, the minimum norm
solution is written
\begin{align*}
\JJ^{\dagger}(H) & =\arg\min_{Y=[Y_{1};Y_{2}]}\left\Vert \begin{bmatrix}\Sigma Y_{1}^{T}+Y_{1}\Sigma & \Sigma Y_{2}^{T}\\
Y_{2}\Sigma & 0
\end{bmatrix}-\begin{bmatrix}H_{11} & H_{12}\\
H_{12}^{T} & H_{22}
\end{bmatrix}\right\Vert ^{2}=\begin{bmatrix}\frac{1}{2}H_{11}\\
H_{12}^{T}
\end{bmatrix}\Sigma^{\dagger},
\end{align*}
where $\Sigma^{\dagger}=\mathrm{diag}(\sigma_{1}^{-1},\dots,\sigma_{k}^{-1},0,\dots,0)$.
From this we see that the pseudoinverse $\JJ^{\dagger}$ has operator
norm
\[
\|\JJ^{\dagger}\|_{\mathrm{op}}=\max_{\|H\|_{F}=1}\|\JJ^{\dagger}(H)\|=(\sqrt{2}\sigma_{k})^{-1},
\]
with maximizer $H_{11}^{\star}=0$ and $H_{22}^{\star}=0$ and $H_{12}^{\star}=\frac{1}{\sqrt{2}}e_{k}h^{T}$,
where $h$ is any unit vector with $\|h\|=1$ and $e_{k}$ is the
$k$-th column of the identity matrix. The desired claim then follows
from the fact that $Y^{\star}\in\range(\JJ^{T})$ and therefore $Y^{\star}=\JJ^{\dagger}\JJ(Y^{\star})$
and $\|Y^{\star}\|_{F}^{2}\le\|\JJ^{\dagger}\|_{\mathrm{op}}^{2}\cdot\|\JJ(Y^{\star})\|_{F}^{2}.$
\end{proof}

Suppose that $\cos^{2}\theta\ge1/2$ holds due to \lemref{align}
within the neighborhood $f(X)-f^{\star}\le R$ for some radius $R>0$.
Substituting \lemref{scal} into (\ref{eq:gd-cos}) yields a \emph{local}
gradient dominance condition
\begin{equation}
\frac{1}{2}\|\nabla f_{0}(X)\|_{F}^{2}\ge\lambda_{\min}(X^{T}X)\cdot[f_{0}(X)-f_{0}^{\star}].\label{eq:simplelb}
\end{equation}
In the overparameterized case $r>r^{\star}$, however, \eqref{eq:simplelb}
does not prove gradient dominance, because $\lambda_{\min}(X^{T}X)$
becomes arbitrarily small as it converges towards $\lambda_{r}(M^{\star})=0$.
Indeed, the inequality (\ref{eq:simplelb}) suggests a \emph{sublinear}
convergence rate, given that 
\begin{equation}
f_{0}(X_{+})-f_{0}^{\star}\le\left(1-\alpha\lambda_{\min}(X^{T}X)\right)\left(f_{0}(X)-f_{0}^{\star}\right),
\end{equation}
has a linear convergence rate $1-\alpha\lambda_{r}(XX^{T})$ that
itself converges to $1$.

\section{\label{sec:local}Local Linear Convergence of Preconditioned Gradient
Descent}

In the  literature, \emph{right preconditioning}
is a technique frequently used to improve the condition number of
a matrix (i.e. its \emph{conditioning}) without affecting its column
span (i.e. its \emph{alignment}); see e.g.~\citet[Section~9.3.4]{saad2003iterative}
or~\citet[Chapter~10]{greenbaum1997iterative}. In this section,
we define a local norm and dual local norm based on\emph{ }right preconditioning
with a positive definite preconditioner 
\begin{align*}
\|U\|_{X,\eta} & \eqdef\|UP_{X,\eta}^{1/2}\|_{F}, & \|V\|_{X,\eta}^{*} & \eqdef\|VP_{X,\eta}^{-1/2}\|_{F}, & P_{X,\eta} & \eqdef X^{T}X+\eta I.
\end{align*}
If we can demonstrate gradient dominance under the dual local norm
\[
f(X)-f^{\star}\le R\quad\implies\quad\tau_{X,\eta}\cdot[f(X)-f^{\star}]\le\frac{1}{2}(\|\nabla f(X)\|_{X,\eta}^{*})^{2}
\]
for some radius constant $R>0$ and dominance constant $\tau_{P}>0$,
and if the function $f$ remains gradient Lipschitz under the local
norm
\begin{align*}
f(X+\alpha V) & \le f(X)+\alpha\inner{\nabla f(X)}V+\frac{\ell_{X,\eta}}{2}\alpha^{2}\|V\|_{X,\eta}^{2},
\end{align*}
for some new lipschitz constant $\ell_{X,\eta}$,
then it follows from the same reasoning as before that the \emph{right
preconditioned} gradient descent iterations $X_{+}=X-\alpha\nabla f(X)P_{X,\eta}^{-1}$
achieves linear convergence
\begin{align*}
f(X_{+})-f^{\star} & \le\left(1-\alpha\cdot\tau_{X,\eta}\right)\cdot(f(X)-f^{\star})\qquad\text{ with step size }\alpha\le\ell_{X,\eta}^{-1}.
\end{align*}
Starting from an initial point $X_{0}$ within the radius $f(X_{0})-f^{\star}\le R$,
it follows that preconditioned gradient descent 
converges to an $\epsilon$-suboptimal point $X_{k}$ that satisfies
$f(X_{k})-f^{\star}\le\epsilon$ in at most $k=O((\tau_{X,\eta}/\ell_{X,\eta})\log(R/\epsilon))$
iterations.

In order to motivate our choice of preconditioner $P_{X,\eta}$, we
return to the perfectly conditioned function $f_{0}(X)=f_{0}^{\star}+\frac{1}{2}\|XX^{T}-M^{\star}\|_{F}^{2}$
considered in the previous section. Repeating the derivation of (\ref{eq:gd-proj})
results in the following
\begin{align}
\|\nabla f_{0}(X)\|_{X,\eta}^{*}= & \max_{\|Y\|_{X,\eta}=1}\inner{\nabla f_{0}(X)}Y=\max_{\|Y\|_{P}=1}\inner{XX^{T}-M^{\star}}{XY^{T}+YX^{T}}\nonumber \\
= & \|XX^{T}-M^{\star}\|_{F}\|XY^{\star T}+Y^{\star}X^{T}\|_{F}\cos\theta,\label{eq:mgd-proj}
\end{align}
in which the incidence angle $\theta$ coincides with the one previously
defined in (\ref{eq:cosdef}), but the corresponding maximizer $Y^{\star}$
is rescaled so that $\|Y^{\star}\|_{X,\eta}=1$. Suppose that $\cos^{2}\theta\ge1/2$
holds due to \lemref{align} within the neighborhood $f(X)-f^{\star}\le R$
for some radius $R>0$. Evoking \lemref{scal} with $Y\gets YP_{X,\eta}^{+1/2}$
and $X\gets XP_{X,\eta}^{-1/2}$ to lower-bound $\|XY^{\star T}+Y^{\star}X^{T}\|_{F}$
yields:
\begin{equation}
\frac{1}{2}(\|\nabla f_{0}(X)\|_{X,\eta}^{*})^{2}\ge\lambda_{\min}(P_{X,\eta}^{-1/2}X^{T}XP_{X,\eta}^{-1/2})\cdot[f_{0}(X)-f_{0}^{\star}].\label{eq:simplelb-1}
\end{equation}
While right preconditioning does not affect the term $\cos\theta$,
which captures the \emph{alignment} between the column span of $X$
and the ground truth $M^{\star}$, it can substantially improve the
\emph{conditioning} of the subspace $\{XY^{T}+YX^{T}:Y\in\R^{n\times r}\}$. 

In particular, choosing $\eta=0$ sets $P_{X,0}=X^{T}X$
and $\lambda_{\min}(P_{X,\eta}^{-1/2}X^{T}XP_{X,\eta}^{-1/2})=1$.
While $f_{0}$ fails to satisfy gradient dominance under the Euclidean
norm, this derivation shows that gradient dominance does indeed hold
after a change of norm. The following is a specialization of \lemref{pl}
that we prove later in this section.
\begin{corollary}[Gradient dominance with $\eta=0$]
Let $\phi$ be $(\mu,r)$-restricted strongly convex and $L$-gradient
Lipschitz, and let $M^{\star}\succeq0$ satisfy $\nabla\phi(M^{\star})=0$
and $r^{\star}=\rank(M^{\star})\le r$. Then, $f(X)\eqdef\phi(XX^{T})$
satisfies gradient dominance
\begin{equation}
f(X)-f^{\star}\le\frac{\mu\cdot\lambda_{r^{\star}}^{2}(M^{\star})}{2(1+L/\mu)}\quad\implies\quad\frac{\mu^{2}}{2L}\cdot[f(X)-f^{\star}]\le\frac{1}{2}(\|\nabla f(X)\|_{X,0}^{*})^{2}.\label{eq:gdom1-1}
\end{equation}
\end{corollary}
In fact, the resulting iterations $X_{+}=X-\alpha\nabla f(X)(X^{T}X)^{-1}$
coincide with the ScaledGD of \citet{tong2020accelerating}. One might
speculate that gradient dominance in this case would readily imply
linear convergence, given that 
\[
f(X_{+})-f^{\star}\le(1-\alpha\tau_{X,0})(f(X)-f^{\star})\qquad\text{with }\alpha\le\max\{1,\ell_{X,0}^{-1}\}.
\]
However, the Lipschitz parameter $\ell_{X,\eta}$ may diverge to infinity
as $\eta\to0$, and this causes the admissible step-size $\alpha\le\max\{1,\ell_{X,\eta}^{-1}\}$
to shrink to zero. Conversely, if we insist on using a fixed step-size
$\alpha>0$, then the objective function may on occasion \emph{increase}
after an iteration, as in $f(X_{+})>f(X)$. Indeed, this possible
increment explains the apparently sporadic behavior exhibited by ScaledGD.

Measured under the Euclidean norm, the function $f$ is gradient Lipschitz
but not gradient dominant. Measured under a right-preconditioned $P$-norm
with $P=X^{T}X$, the function $f$ is gradient dominant but not gradient
Lipschitz. Viewing the Euclidean norm as simply a right-preconditioned
norm with $P=I$, a natural idea is to \emph{interpolate} between
these two norms, by choosing the preconditioner $P_{X,\eta}=X^{T}X+\eta I$.
It is not difficult to show that keeping $\eta$ sufficiently \emph{large}
with respect to the error norm $\|XX^{T}-M^{\star}\|_{F}$ is enough
to ensure that $f$ continues to satisfy gradient Lipschitzness under
the local norm. The proof of \lemref{decr} and \lemref{bndgrad}
below follows from straightforward linear algebra, and are deferred
to \appref{pf_decr} and \appref{pf_bndgrad} respectively.
\begin{lemma}[Gradient Lipschitz]
\label{lem:decr}Let $\phi$ be $L$-gradient Lipschitz. Let $M^{\star}=\arg\min\phi$
satisfy $M^{\star}\succeq0$. Then $f(X)\eqdef\phi(XX^{T})$ satisfies
\begin{gather*}
f(X+V)\le f(X)+\inner{\nabla f(X)}V+\frac{\ell_{X,\eta}}{2}\|V\|_{X,\eta}^{2}\\
\text{where }\ell_{X,\eta}=L\cdot\left[4+\frac{2\|XX^{T}-M^{\star}\|_{F}+4\|V\|_{X,\eta}}{\lambda_{\min}(X^{T}X)+\eta}+\left(\frac{\|V\|_{X,\eta}}{\lambda_{\min}(X^{T}X)+\eta}\right)^{2}\right].
\end{gather*}
\end{lemma}
\begin{lemma}[Bounded gradient]
\label{lem:bndgrad}Under the same conditions as \lemref{decr},
the search direction $V=\nabla f(X)(X^{T}X+\eta I)^{-1}$ satisfies
$\|V\|_{X,\eta}=\|\nabla f(X)\|_{X,\eta}^{*}\le2L\|XX^{T}-M^{\star}\|_{F}$. 
\end{lemma}
Substituting $X_{+}=X-\alpha\nabla f(X)(X^{T}X+\eta I)^{-1}$ into
\lemref{decr} yields the usual form of the Lipschitz gradient decrement\begin{subequations}\label{eq:local1}
\begin{equation}
f(X_{+})\le f(X)-\alpha\cdot(\|\nabla f(X)\|_{X,\eta}^{*})^{2}+\alpha^{2}\cdot\frac{\ell_{X,\eta}}{2}(\|\nabla f(X)\|_{X,\eta}^{*})^{2}\label{eq:local1a}
\end{equation}
in which the \emph{local} Lipschitz term $\ell_{X,\eta}$ is bounded
by \lemref{bndgrad} as
\begin{equation}
\ell_{X,\eta}\le4L+(2L+8L^{2})\cdot\frac{\|XX^{T}-M^{\star}\|_{F}}{\lambda_{\min}(X^{T}X)+\eta}+4L^{3}\cdot\left(\frac{\|XX^{T}-M^{\star}\|_{F}}{\lambda_{\min}(X^{T}X)+\eta}\right)^{2}.\label{eq:local1b}
\end{equation}
\end{subequations}By keeping $\eta$ sufficiently \emph{large} with
respect to the error norm $\|XX^{T}-M^{\star}\|_{F}$, it follows
that $\ell_{X,\eta}$ can be replaced by a \emph{global} Lipschitz
constant $\ell\ge\ell_{X,\eta}$ that is independent of $X$. 

Our main result in this paper is that keeping $\eta$ sufficiently
\emph{small} with respect to the error norm $\|XX^{T}-M^{\star}\|_{F}$
is enough to ensure that $f$ satisfies gradient dominance, even in
the overparameterized regime where $r>r^{\star}$.
\begin{lemma}[Gradient dominance]
\label{lem:pl}Let $\phi$ be $L$-gradient Lipschitz and $(\mu,2r)$-restricted
strongly convex. Let $M^{\star}=\arg\min\phi$ satisfy $M^{\star}\succeq0$
and $r^{\star}=\rank(M^{\star})\le r$. Then, $f(X)\eqdef\phi(XX^{T})$
satisfies 
\begin{gather*}
f(X)-f^{\star}\le\frac{\mu}{2(1+L/\mu)}\cdot\lambda_{r^{\star}}^{2}(M^{\star})\qquad\implies\\
\frac{\mu}{\sqrt{2}}\left(1+\eta\cdot\frac{c_{0}+c_{1}\cdot\sqrt{r-r^{\star}}}{\|XX^{T}-M^{\star}\|_{F}}\right)^{-1/2}\le\frac{\|\nabla f(X)\|_{X,\eta}^{*}}{\|XX^{T}-M^{\star}\|_{F}}
\end{gather*}
where $c_{0}=1+\sqrt{2}$ and $c_{1}=(L+\mu)/\sqrt{\mu L}$. 
\end{lemma}
Substituting $f(X)-f^{\star}\le\frac{L}{2}\|XX^{T}-M^{\star}\|_{F}^{2}$
from \lemref{riperr} into \lemref{pl} recovers the usual form of
gradient dominance \begin{subequations}\label{eq:local2}
\begin{equation}
\tau_{X,\eta}\cdot[f(X)-f^{\star}]\le\frac{1}{2}(\|\nabla f(X)\|_{X,\eta}^{*})^{2}\label{eq:local2a}
\end{equation}
in which the \emph{local} dominance term $\tau_{X,\eta}$ reads 
\begin{equation}
\tau_{X,\eta}=\frac{\mu^{2}}{2L}\left(1+\eta\cdot\frac{c_{0}+c_{1}\cdot\sqrt{r-r^{\star}}}{\|XX^{T}-M^{\star}\|_{F}}\right)^{-1}>0.\label{eq:local2b}
\end{equation}
\end{subequations}By keeping $\eta$ sufficiently \emph{small} with
respect to the error norm $\|XX^{T}-M^{\star}\|_{F}$, it follows
that $\tau_{X,\eta}$ can be replaced by a \emph{global} dominance
constant $\tau\le\tau_{X,\eta}$ that is independent of $X$. Finally,
substituting the global Lipschitz constant $\ell\ge\ell_{X,\eta}$
and the global dominance constant $\tau\le\tau_{X,\eta}$ into (\ref{eq:local1})
and (\ref{eq:local2}) yields a proof of linear convergence in \thmref{local}.

\begin{proof}[Proof of \thmref{local}]It follows from (\ref{eq:local1b})
that
\[
\eta\ge C_{\lb}\cdot\|XX^{T}-M^{\star}\|_{F}\quad\implies\quad\ell_{X,\eta}\le4L+\frac{2L+8L^{2}}{C_{\lb}}+\frac{4L^{3}}{C_{\lb}^{2}}\eqdef\ell.
\]
Substituting $\ell\ge\ell_{X,\eta}$ into (\ref{eq:local1a}) yields
a guaranteed gradient decrement
\begin{equation}
f(X_{+})-f(X)\le-\frac{\alpha}{2}(\|\nabla f(X)\|_{X,\eta}^{*})^{2}\le0\text{ for }\alpha\le\min\{1,\ell^{-1}\},\label{eq:local3}
\end{equation}
for a fixed step-size $\alpha>0$. It follows from (\ref{eq:local2b})
that
\[
\eta\le C_{\ub}\cdot\|XX^{T}-M^{\star}\|_{F}\quad\implies\quad\tau_{X,\eta}\ge\frac{\mu^{2}}{2L}\left(1+\frac{c_{0}+c_{1}\cdot\sqrt{r-r^{\star}}}{C_{\ub}^{-1}}\right)^{-1}\eqdef\tau.
\]
Substituting $\tau\le\tau_{X,\eta}$ and gradient dominance (\ref{eq:local2a})
into the decrement in (\ref{eq:local3}) yields linear convergence
\[
f(X_{+})-f^{\star}\le(1-\alpha\cdot\tau)\cdot(f(X)-f^{\star})\text{ for }\alpha\le\min\{1,\ell^{-1}\},
\]
which is exactly the claim in \thmref{local}. \end{proof}

\begin{proof}[\corref{param}] Within the neighborhood
stated in \lemref{pl} where $f$ is gradient dominant, it follows
immediately from \lemref{bndgrad} and \lemref{pl} that the choice
of $\eta=\|\nabla f(X)\|_{X,0}^{*}$ satisfies 
\[
\frac{\mu}{\sqrt{2}}\cdot\|XX^{T}-M^{\star}\|_{F}\le\|\nabla f(X)\|_{X,0}^{*}\le2L\cdot\|XX^{T}-M^{\star}\|_{F},
\]
which is exactly the claim in \corref{param}. \end{proof}

We now turn our attention to the proof of \lemref{pl}. Previously,
in motivating our proof for gradient dominance under the Euclidean
norm, we derived a bound like
\begin{align}
\|\nabla f_{0}(X)\|_{F}^{*} & =\max_{\|Y\|_{F}=1}\inner{XX^{T}-M^{\star}}{XY^{T}+YX^{T}}\nonumber \\
 & =\|XX^{T}-M^{\star}\|_{F}\|XY^{\star T}+Y^{\star}X^{T}\|_{F}\cos\theta\label{eq:proofmot1}
\end{align}
where $Y^{\star}$ is a maximizer such that $\|Y^{\star}\|_{F}=1$.
We found that $\cos\theta$ is always large, because the error $XX^{T}-M^{\star}$
is guaranteed to align well with the linear subspace $\{XY^{T}+YX^{T}:Y\in\R^{n\times r}\}$,
but that the term $\|XY^{\star T}+Y^{\star}X^{T}\|_{F}$ can decay
to zero if the error concentrates within the degenerate directions
of the subspace. 

In our initial experiments with PrecGD, we observed that small values
of $\eta$ caused the error to preferrably align towards the \emph{well-conditioned}
directions of the subspace. Suppose that $X$ contains $k$ large,
well-conditioned singular values, and $r-k$ near-zero singular values.
Let $X_{k}$ denote the rank-$k$ approximation of $X$, constructed
by setting the $r-k$ near-zero singular values of $X$ as exactly
zero:
\[
X_{k}=\sum_{i=1}^{k}\sigma_{i}u_{i}v_{i}^{T}=\arg\min_{\tilde{X}\in\R^{n\times r}}\left\{ \|\tilde{X}-X\|:\rank(\tilde{X})\le k\right\} .
\]
Then, our observation is that small values of $\eta$ tend to concentrate
the error $XX^{T}-M^{\star}$ within the well-conditioned subspace
$\{X_{k}Y^{T}+YX_{k}^{T}:Y\in\R^{n\times r}\}$.

In order to sharpen the bound (\ref{eq:proofmot1})\textbf{ }to reflect
the possibility that $\{XY^{T}+YX^{T}:Y\in\R^{n\times r}\}$ may contain
degenerate directions that do not significantly align with the error
vector $XX^{T}-M^{\star}$, we suggest the following refinement 
\begin{align*}
\|\nabla f_{0}(X)\|_{F} & \ge\max_{\|Y\|_{F}=1}\left\{ \inner{XX^{T}-M^{\star}}{XY^{T}+YX^{T}}:Yv_{i}=0\text{ for }i>k\right\} \\
 & =\max_{\|Y\|_{F}=1}\inner{XX^{T}-M^{\star}}{X_{k}Y^{T}+YX_{k}^{T}}\\
 & =\|XX^{T}-M^{\star}\|_{F}\|X_{k}Y_{k}^{\star T}+Y_{k}^{\star}X_{k}^{T}\|_{F}\cos\theta_{k},
\end{align*}
where each $\cos\theta_{k}$ measures the alignment between the error
$XX^{T}-M^{\star}$ and the well-conditioned subspace $\{X_{k}Y^{T}+YX_{k}^{T}:Y\in\R^{n\times r}\}$.
While $\cos\theta_{k}$ must be necessarily be worse than $\cos\theta$,
given that the well-conditioned subspace is a subset of the whole
subspace, our hope is that eliminating the degenerate directions will
allow the term $\|X_{k}Y_{k}^{\star T}+Y_{k}^{\star}X_{k}^{T}\|_{F}$
to be significantly improved from $\|XY^{\star T}+Y^{\star}X^{T}\|_{F}$. 
\begin{lemma}[Alignment lower-bound]
\label{lem:gradbnd2}Let $X=\sum_{i=1}^{r}\sigma_{i}u_{i}v_{i}^{T}$
with $\|u_{i}\|=\|v_{i}\|=1$ and $\sigma_{1}\ge\cdots\ge\sigma_{r}$
denote its singular value decomposition. Under the same conditions
as \lemref{pl}, we have 
\begin{equation}
\frac{\|\nabla f(X)\|_{X,\eta}^{*}}{\|XX^{T}-M^{\star}\|_{F}}\ge\max_{k\in\{1,2,\dots,r\}}\frac{\mu+L}{\sqrt{2}}\cdot\frac{\cos\theta_{k}-\delta}{\sqrt{1+\eta/\lambda_{k}(XX^{T})}}\label{eq:maxcos}
\end{equation}
where $\delta=\frac{L-\mu}{L+\mu}$ and each $\theta_{k}$ is defined
\begin{equation}
\cos\theta_{k}=\max_{Y\in\R^{n\times r}}\frac{\inner{XX^{T}-M^{\star}}{X_{k}Y^{T}+YX_{k}^{T}}}{\|XX^{T}-M^{\star}\|_{F}\|X_{k}Y^{T}+YX_{k}^{T}\|_{F}},\qquad X_{k}=\sum_{i=1}^{k}\sigma_{i}u_{i}v_{i}^{T}.\label{eq:cosk}
\end{equation}
\end{lemma}
\begin{proof}
Let $E=XX^{T}-M^{\star}$ and $\JJ(Y)=XY^{T}+YX^{T}$ and $\JJ_{k}=X_{k}Y^{T}+YX_{k}^{T}$.
Repeating the proof of \lemref{prelimgrad} yields the following corollary
\[
\|\nabla f(X)\|_{X,\eta}^{*}\ge\nu\cdot\left\{ \max_{\|Y\|_{X,\eta}=1}\inner E{\JJ(Y)}-\delta\|E\|_{F}\|\JJ(Y)\|_{F}\right\} 
\]
where $\nu=\frac{1}{2}(\mu+L)$ and $\delta=\frac{L-\mu}{L+\mu}$.
For any $k\in\{1,2,\dots,r\}$, we can restrict this problem so that
\begin{align*}
\|\nabla f(X)\|_{X,\eta}^{*} & \ge\nu\cdot\left\{ \max_{\|Y\|_{X,\eta}=1}\inner E{\JJ(Y)}-\delta\|E\|_{F}\|\JJ(Y)\|_{F}:Yv_{i}=0\text{ for }i>k\right\} \\
 & =\nu\cdot\left\{ \max_{\|Y\|_{X,\eta}=1}\inner E{\JJ_{k}(Y)}-\delta\|E\|_{F}\|\JJ_{k}(Y)\|_{F}\right\} \\
 & \ge\nu\cdot\|E\|_{F}\|\JJ_{k}(Y_{k}^{\star})\|_{F}(\cos\theta_{k}-\delta)
\end{align*}
where $Y_{k}^{\star}=\JJ_{k}^{\dagger}(E)$ denotes the solution of
(\ref{eq:cosk}) rescaled so that $\|Y\|_{X,\eta}=1$. Let $P=X^{T}X+\eta I$
and observe that $\|Y_{k}^{\star}\|_{X,\eta}^{2}=\|Y_{k}^{\star}P^{1/2}\|_{F}^{2}$.
It follows from \lemref{scal} with $X\gets X_{k}P^{-1/2}$ and $Y\gets Y_{k}^{\star}P^{1/2}$
that
\begin{align*}
\|X_{k}Y_{k}^{\star T}+Y_{k}^{\star}X_{k}^{T}\|_{F}^{2} & \ge2\cdot\lambda_{\min}(P^{-1/2}X_{k}^{T}X_{k}P^{-1/2})\cdot\|Y_{k}^{\star}\|_{X,\eta}^{2}.
\end{align*}
In turn, we have $P=\sum(\sigma_{i}^{2}+\eta)v_{i}v_{i}^{T}$ and
therefore
\[
\lambda_{\min}(P^{-1/2}X_{k}^{T}X_{k}P^{-1/2})=\min_{i\le k}\left\{ \frac{\sigma_{i}^{2}}{\eta+\sigma_{i}^{2}}\right\} =\frac{\sigma_{k}^{2}}{\eta+\sigma_{k}^{2}}=\frac{1}{1+\eta/\sigma_{k}^{2}}.
\]
Substituting these together yields 
\begin{align*}
\|\nabla f(X)\|_{X,\eta}^{*}\ge & \nu\cdot\|E\|_{F}\cdot\|\JJ_{k}(Y_{k}^{\star})\|_{F}\cdot(\cos\theta_{k}-\delta)\\
\ge & \frac{\mu+L}{2}\cdot\|E\|_{F}\cdot\frac{\sqrt{2}}{\sqrt{1+\eta/\sigma_{k}^{2}}}\cdot(\cos\theta_{k}-\delta).
\end{align*}
\end{proof}

From \lemref{gradbnd2}, we see that gradient dominance holds if the
subspace $\{X_{k}Y^{T}+YX_{k}^{T}:Y\in\R^{n\times r}\}$ induced by
the rank-$k$ approximation of $X$ is \emph{well-conditioned}, and
if the error vector $XX^{T}-M^{\star}$ is\emph{ well-aligned} with
it. Specifically, this is to require both $\lambda_{k}(XX^{T})$ and
$\cos\theta_{k}$ to remain sufficiently large for the \emph{same}
value of $k$. Within a neighborhood of the ground truth, it follows
from Weyl's inequalty that $\lambda_{k}(XX^{T})$ will remain sufficiently
large for $k=r^{\star}$; see \lemref{dom_base} below. In the overparameterized
regime $r>r^{\star}$, however, it is not necessarily true that $\cos\theta_{k}\to1$
for $k=r^{\star}$. Instead, we use an induction argument: if $\cos\theta_{k}$
is too small to prove gradient dominance, then the smallness of $\cos\theta_{k}$
provides a lower-bound on $\lambda_{k+1}(XX^{T})$ via \lemref{sintheta}
below. Inductively repeating this argument for $k=r^{\star},r^{\star}+1,\dots$
arrives at a lower-bound on $\lambda_{r}(XX^{T})$. At this point,
\lemref{align} guarantees that $\cos\theta_{r}$ is large, and therefore,
we conclude that gradient dominance must hold.
\begin{lemma}[Base case]
\label{lem:dom_base}Under the same conditions as \lemref{pl}, let
$f(X)-f(X^{\star})\leq\frac{\mu}{2(1+L/\mu)}\cdot\lambda_{r^{\star}}^{2}(M^{\star})$.
Then, 
\[
\lambda_{r^{\star}}(XX^{T})\ge(\sqrt{1+L/\mu}-1)\cdot\|XX^{T}-M^{\star}\|_{F}.
\]
\end{lemma}
\begin{proof}
By our choice of neighborhood, we have
\[
\|XX^{T}-M^{\star}\|_{F}^{2}\leq\frac{2}{\mu}\cdot[f(X)-f(X^{\star})]\leq\frac{1}{1+L/\mu}\cdot\lambda_{r^{\star}}^{2}(M^{\star}).
\]
The desired claim follows from Weyl's inequality:
\begin{align*}
\lambda_{r^{\star}}(XX^{T}) & =\lambda_{r^{\star}}(M^{\star}+XX^{T}-M^{\star})\\
 & \ge\lambda_{r^{\star}}(M^{\star})-\|XX^{T}-M^{\star}\|_{F}\\
 & \ge(\sqrt{1+L/\mu}-1)\cdot\|XX^{T}-M^{\star}\|_{F}.
\end{align*}
\end{proof}

\begin{lemma}[Induction step]
\label{lem:sintheta}Under the same conditions as \lemref{pl}, let
$f(X)-f(X^{\star})\leq\frac{\mu}{2(1+L/\mu)}\cdot\lambda_{r^{\star}}^{2}(M^{\star})$.
Then, $\cos\theta_{k}$ defined in (\ref{eq:cosk}) gives the following
lower-bound on $\lambda_{k+1}(XX^{T})$:
\[
\frac{\lambda_{k+1}^{2}(XX^{T})\cdot(r-k)}{\|XX^{T}-M^{\star}\|_{F}^{2}}-\frac{\mu L}{(L+\mu)^{2}}\ge\left(\frac{L}{L+\mu}\right)^{2}-\cos^{2}\theta_{k}.
\]
\end{lemma}
\begin{proof}
For $r^\star \le k\le r$, we have from $M^\star\succeq 0$
\begin{align*}
 & \|(I-X_{k}X_{k}^{\dagger})(XX^{T}-M^{\star})(I-X_{k}X_{k}^{\dagger})\|_{F}^{2}\\
= & \|(I-X_{k}X_{k}^{\dagger})XX^{T}(I-X_{k}X_{k}^{\dagger})\|_{F}^{2} +\|(I-X_{k}X_{k}^{\dagger})M^{\star}(I-X_{k}X_{k}^{\dagger})\|_{F}^{2} \\
& \quad -2\inner{(I-X_{k}X_{k}^{\dagger})XX^{T}(I-X_{k}X_{k}^{\dagger})}{(I-X_{k}X_{k}^{\dagger})M^{\star}(I-X_{k}X_{k}^{\dagger})} \\
\le & \lambda_{k+1}^{2}(XX^{T})(r-k)
+\|(I-X_{r^\star}X_{r^\star}^{\dagger})M^{\star}(I-X_{r^\star}X_{r^\star}^{\dagger})\|_{F}^{2} + 0
\end{align*}
and therefore
\begin{align*}
\sin^{2}\theta_{k} & \le\frac{\lambda_{k+1}^{2}(XX^{T})(r-k)}{\|XX^{T}-M^{\star}\|_{F}^{2}}+\frac{\|(I-X_{r^\star}X_{r^\star}^{\dagger})M^{\star}(I-X_{r^\star}X_{r^\star}^{\dagger})\|_{F}^{2}}{\|X_{r^\star}X_{r^\star}^{T}-M^{\star}\|_{F}^{2}}
\end{align*}
By the choice of the neighborhood, we have
\[
\|XX^{T}-M^{\star}\|_{F}^{2}\leq\frac{2}{\mu}\cdot[f(X)-f(X^{\star})]\leq\frac{1}{1+L/\mu}\cdot\lambda_{r^{\star}}^{2}(M^{\star}).
\]
Substituting $\rho=\frac{1}{\sqrt{1+L/\mu}}\le\frac{1}{\sqrt{2}}$
into \lemref{basis_k} proves that
\[
\frac{\|(I-X_{r^\star}X_{r^\star}^{\dagger})M^{\star}(I-X_{r^\star}X_{r^\star}^{\dagger})\|_{F}^{2}}{\|XX^{T}-M^{\star}\|_{F}^{2}}\leq\frac{1}{2}\frac{\rho}{1-\rho^{2}}=\frac{\mu}{2L}.
\]
Splitting 
\begin{align*}
1-\frac{\mu}{2L} & =\left(\frac{L}{L+\mu}\right)^{2}+\frac{\mu}{2L}\frac{(3L^{2}-\mu^{2})}{(L+\mu)^{2}}\\
 & \ge\left(\frac{L}{L+\mu}\right)^{2}+\frac{\mu L}{(L+\mu)^{2}}
\end{align*}
 and bounding $\cos^{2}\theta_{k}\ge1-\sin^{2}\theta_{k}$ yields
the desired bound. 
\end{proof}

Rigorously repeating this induction results in a proof of \lemref{pl}.

\begin{proof}[\lemref{pl}] \lemref{gradbnd2} proves
gradient dominance if we can show that both $\cos\theta_{k}$ and
$\lambda_{k}(XX^{T})$ remain large for the same value of $k$. By
\lemref{dom_base} we have
\begin{equation}
\frac{\lambda_{r^{\star}}(XX^{T})}{\|XX^{T}-M^{\star}\|_{F}}\ge\sqrt{1+L/\mu}-1\ge\sqrt{2}-1=\frac{1}{1+\sqrt{2}}.\label{eq:induct3}
\end{equation}
If $\cos\theta_{r^{\star}}\ge\frac{L}{L+\mu}$, then substituting
(\ref{eq:induct3}) into \lemref{gradbnd2} with $k=r^{\star}$ yields
gradient dominance:
\begin{align}
\frac{\|\nabla f(X)\|_{X,\eta}^{*}}{\|XX^{T}-M^{\star}\|} & \ge\frac{(\mu+L)}{\sqrt{2}}\cdot\left(\frac{L}{L+\mu}-\frac{L-\mu}{L+\mu}\right)\left(1+\frac{\eta}{\lambda_{r^{\star}}(XX^{T})}\right)^{-1/2}\nonumber \\
 & \ge\frac{\mu}{\sqrt{2}}\left(1+\eta\cdot\frac{1+\sqrt{2}}{\|XX^{T}-M^{\star}\|_{F}}\right)^{-1/2}.\label{eq:inductbnd1}
\end{align}
Otherwise, if $\cos\theta_{r^{\star}}<\frac{L}{L+\mu}$, then we proceed
with an induction argument. Beginning at the base case $k=r^{\star}$,
we evoke \lemref{sintheta} and use $\cos\theta_{k}<\frac{L}{L+\mu}$
to lower-bound $\lambda_{k+1}(XX^{T})$ by a constant:
\begin{gather}
\frac{\lambda_{k+1}^{2}(XX^{T})\cdot(r-k)}{\|XX^{T}-M^{\star}\|_{F}^{2}}-\frac{\mu L}{(L+\mu)^{2}}\ge\left(\frac{L}{L+\mu}\right)^{2}-\cos^{2}\theta_{k}>0,\nonumber \\
\implies\quad\frac{\lambda_{k+1}(XX^{T})}{\|XX^{T}-M^{\star}\|_{F}}>\frac{1}{\sqrt{r-r^{\star}}}\cdot\frac{\sqrt{\mu L}}{L+\mu}.\label{eq:induct2}
\end{gather}
If $\cos\theta_{k+1}\ge\frac{L}{L+\mu}$, then substituting (\ref{eq:induct2})
into \lemref{gradbnd2} yields gradient dominance:
\begin{align}
\frac{\|\nabla f(X)\|_{X,\eta}^{*}}{\|XX^{T}-M^{\star}\|} & \ge\frac{(\mu+L)}{\sqrt{2}}\left(\frac{L}{L+\mu}-\frac{L-\mu}{L+\mu}\right)\left(1+\frac{\eta}{\lambda_{k+1}(XX^{T})}\right)^{-1/2}\nonumber \\
 & \ge\frac{\mu}{\sqrt{2}}\left(1+\eta\cdot\frac{\sqrt{r-r^{\star}}\cdot\frac{L+\mu}{\sqrt{\mu L}}}{\|XX^{T}-M^{\star}\|_{F}}\right)^{-1/2}.\label{eq:inductbnd2}
\end{align}
Otherwise, if $\cos\theta_{k+1}<\frac{L}{L+\mu}$, then we repeat
the same argument in (\ref{eq:induct2}) with $k\gets k+1$, until
we arrive at $k=r$. At this point, \lemref{sintheta} guarantees
$\cos\theta_{r}\ge\frac{L}{L+\mu}$, since
\[
0\ge\underbrace{\frac{\lambda_{k+1}^{2}(XX^{T})\cdot(r-k)}{\|XX^{T}-M^{\star}\|_{F}^{2}}}_{=0\text{ because }k=r}-\frac{\mu L}{(L+\mu)^{2}}\ge\left(\frac{L}{L+\mu}\right)^{2}-\cos^{2}\theta_{k},
\]
so the induction terminates with (\ref{eq:inductbnd2}). Finally,
lower-bounding the two bounds (\ref{eq:inductbnd1}) and (\ref{eq:inductbnd2})
via $\min\{a^{-1},b^{-1}\}\ge(a+b)^{-1}$ yields the desired \lemref{pl}.
\end{proof}

\section{\label{sec:global}Global Convergence}

In this section, we study the global convergence of perturbed PrecGD or \ref{PPrecGD} from an arbitrary initial point to an approximate second order stationary point. To establish the global convergence of \ref{PPrecGD}, we study a slightly more general variant of gradient descent, which we call \textit{perturbed metric gradient descent}.

\subsection{Perturbed Metric Gradient Descent}
Let $P:\R^{d}\to\S_{++}^{d}$ denote an arbitrary metric function,
which we use to define the following two local norms
\[
\|v\|_{x}\eqdef\sqrt{v^{T}P(x)v},\qquad\|v\|_{x}^{*}\eqdef\sqrt{v^{T}P(x)^{-1}v}
\]
Let $f:\R^{d}\to\R$ denote an arbitrary $\ell_1$-gradient Lipschitz
and $\ell_{2}$-Hessian Lipschitz function. We consider solving the general
minimization problem $f^{\star}=\min_{x}f(x)$ via \textit{perturbed metric gradient descent}, defined as 
\begin{equation}
x_{k+1}=x_{k}-\alpha P(x_{k})^{-1}\nabla f(x_{k})+\alpha\zeta_{k},\tag{PMGD}\label{eq:pmgd}
\end{equation}
in which the random perturbation $\zeta_{k}$ is chosen as
\[
\begin{cases}
\zeta_{k}\sim\mathbb{B}(\beta) & \text{if }\;\|\nabla f(x)\|_{x}^{*}\le\epsilon, \text{ and it has been at least } \mathcal{T} \text{ iters since last perturbation}\\
\zeta_{k}=0 & \text{otherwise.}
\end{cases}
\]

Indeed, \ref{PPrecGD} is a special case of \ref{eq:pmgd} after choosing $x = \mathrm{vec}(X)$ and $P(x) = (X^TX+\eta_0 I_n)\otimes I_r$. Our main result in this section is to show that \ref{eq:pmgd} converges to $\epsilon$-second order stationary point measured in the metric norm
\[
\|\nabla f(x)\|_{x}^{*}\le\epsilon,\quad\nabla^{2}f(x)\succeq-\sqrt{L_d\epsilon}\cdot P(x),
\]
for some constant $L_d$ to be defined later, in at most $\tilde{O}(\epsilon^{-2})$ iterations under two assumptions:
\begin{itemize}
	\item The metric
	$P$ should be \emph{well-conditioned}:
	\begin{equation}
	p_{\lb}I\preceq P(x)\preceq p_{\ub}I\qquad\text{for all }x\in\R^{d}.\label{eq:well-cond}
	\end{equation}
	for some $p_\ub\geq p_\lb>0$.
	\item The metric $P$ should be \emph{Lipschitz continuous}:
	\begin{equation}
	\|P(x)-P(y)\|\le L_{P}\cdot\|x-y\|\qquad\text{for all }x,y\in\R^{d},\label{eq:lipP}
	\end{equation}
	for some $L_P>0$.
\end{itemize}

\paragraph{\textbf{Comparison to Perturbed Gradient Descent.}} \citet{jin2021nonconvex} showed that the episodic injection of isotropic noise to gradient descent enables it to escape strict saddle points efficiently. Indeed, this algorithm, called perturbed gradient descent or PGD for short, can be regarded as an special case of PMGD, with a crucial simplification that the metric function $P$ is the identity mapping throughout the iterations of the algorithm. As will be explained later, such simplification enables PGD to behave almost like the power method within the vicinity of a strict saddle point, thereby steering the iterations away from it at an exponential rate along the negative curvature of the function. Extending this result to \ref{eq:pmgd} with a more general metric function that changes along the solution trajectory requires a more intricate analysis, which will be provided next. Our main theorem shows that \ref{eq:pmgd} can also escape strict saddle points, so long as the metric function $P(x)$ remains well-conditioned and Lipschitz continuous.

\begin{theorem}[Global Convergence of PMGD]\label{thm:pmgd}Let $f$ be $\ell_{1}$-gradient and $\ell_{2}$-Hessian
	Lipschitz, and let $P$ satisfy $p_{\lb}I\preceq P(x)\preceq p_{\ub}I$ and $\|P(x)-P(y)\|\le L_{P}\cdot\|x-y\|$.
	Then, with an overwhelming probability and for any $\epsilon =\tilde{O}(1/(L_d p_\ub))$, \ref{eq:pmgd} with perturbation radius $\beta = \tilde{\mathcal{O}}(\epsilon/L_d)$ and time interval $\mathcal{T} = \tilde{\mathcal{O}}({\ell_1}/{(p_\lb\sqrt{ L_{d}\epsilon})})$ converges to a point $x$ that satisfies 
	\[
	\|\nabla f(x)\|_{x}^{*}\le\epsilon,\quad\text{and} \quad \nabla^{2}f(x)\succeq-\sqrt{L_d\epsilon}\cdot P(x),
	\]
	in at most $\tilde{O}(\mathcal{C}(f(x)-f^*)/\epsilon^{2})$ iterations, where $f^*$ is the optimal objective value, and $L_d = {5\max\{\ell_2,L_P \ell_1\sqrt{p_\ub}\}}/{p_\lb^{2.5}}$, $\mathcal{C} = {\ell_1}/{p_\lb^2}$.
\end{theorem}

Before providing the proof for Theorem \ref{thm:pmgd}, we first show how it can be invoked to prove Theorem~\ref{cor_pprecgd}. To apply Theorem \ref{thm:pmgd}, we need to show that: $(i)$ $f(X) = \phi(XX^T)$ is gradient and Hessian Lipschitz; and $(ii)$ $P = (XX^\top+\eta I)$ is well-conditioned. However, the function $f(X) = \phi(XX^T)$ may neither be gradient  nor Hessian Lipschitz, even if these properties hold for $\phi$. To see this, consider $\phi(M) = \|M-M^\star\|_F^2$. Evidently, $\phi(M)$ is $2$-gradient Lipschitz with constant Hessian. However, $f(X) = \phi(XX^\top) = \|XX^\top-M^\star\|_F^2$ is neither gradient- nor Hessian-Lipschitz since it is a quartic function of $X$. To alleviate this hurdle, we show that, under a mild condition on the coercivity of $\phi$, the iterations of \ref{PPrecGD} reside in a bounded region, within which $f(X)$ is both gradient and Hessian Lipschitz. 

\begin{lemma}\label{lem_Lip}
Suppose that $\phi$ is coercive. Let $\Gamma_F$ and $\Gamma_2$ be defined as
\begin{align}
	\Gamma_F &= \max\left\{\|X\|_F: \phi(XX^T)\leq \phi(X_0X_0^T)+2\sqrt{\|X_0\|_F^2+\eta}\cdot \alpha r\epsilon\right\}\label{eq_GammaF}\\
	\Gamma_2 &= \max\left\{\|X\|_2: \phi(XX^T)\leq \phi(X_0X_0^T)+2\sqrt{\|X_0\|_F^2+\eta}\cdot \alpha r\epsilon\right\}.\label{eq_Gamma2}
\end{align}
Then, the following statements hold:
\begin{itemize}
\item Every iteration of \ref{PPrecGD} satisfies $\|X_t\|_F\leq \Gamma_F$ and $\|X_t\|\leq \Gamma_2$.
\item The function $f(X)$ is $9\Gamma_F^2L_1$-gradient Lipschitz within the ball $\{M: \|M\|_F\leq \Gamma_F\}$.
\item The function $f(X)$ is $\left((4\Gamma_F+2)L_1+4\Gamma_F^2 L_2\right)$-Hessian Lipschitz within the ball $\{M: \|M\|_F\leq \Gamma_F\}$.
\item For $\P_{X,\eta} = (X^{T}X+\eta I_{n})\otimes I_{r}$, we have $\eta I \preceq \P_{X,\eta}\preceq (\Gamma_2^2+\eta)I$ and $\|\P_{X,\eta}-\P_{Y,\eta}\|\leq 2\Gamma_2\norm {X-Y}$ within the ball $\{M: \|M\|\leq \Gamma_2\}$.
\end{itemize}
\end{lemma}
Equipped with Lemma~\ref{lem_Lip} and Theorem~\ref{thm:pmgd}, we are ready to present the global convergence result for \ref{PPrecGD}.


\vspace{2mm}
\begin{proof} [Theorem \ref{cor_pprecgd}]
	Due to our definition of $\Gamma_F$ and $\Gamma_2$, every iteration of~\ref{PPrecGD} belongs to the set $\mathcal{D} = \{X: \|X\|_F\leq \Gamma_F, \|X\|\leq \Gamma_2\}$. On the other hand, Lemma~\ref{lem_Lip} implies that $f(X)$ is $9L_1\Gamma_F^2$-gradient Lipschitz and $((4\Gamma_F+2)L_1+4\Gamma_F^2 L_2)$-Hessian Lipschitz within $\mathcal{D}$. Moreover, $\eta I \preceq \P_{X,\eta}\preceq (\Gamma_2^2+\eta)I$ and $\|\P_{X,\eta}-\P_{Y,\eta}\|\leq 2\Gamma_2\norm {X-Y}$ within $\mathcal{D}$. Therefore, invoking Theorem~\ref{thm:pmgd} with parameters $\ell_1 = 9L_1\Gamma_F^2$, $\ell_2 = ((4\Gamma_F+2)L_1+4\Gamma_F^2 L_2)$, $p_\lb = \eta$, $p_\ub = \Gamma_2^2+\eta$, and $L_P = 2\Gamma_2$ completes the proof. 
 \end{proof}

\subsection{Proof of Theorem~\ref{thm:pmgd}}
To prove Theorem~\ref{thm:pmgd}, we follow the main idea of \citet{jin2021nonconvex} and split the iterations into two parts:
\begin{itemize}
	\item {\bf Large gradient in local norm:} Suppose that $\norm{\nabla f(x)}_x^*> \epsilon$ for some $\epsilon>0$. Then, we show in Lemma~\ref{lemma:grad_decrement} that a single iteration of \ref{PPrecGD} without perturbation reduces the objective function by $\Omega(\epsilon^2)$.
	\item {\bf Large negative curvature in local norm:} Suppose that $\|\nabla f(x)\|_x^*\geq \epsilon$ and $x$ is not an $\epsilon$-second order stationary point in local norm, i.e., $\nabla^2 f(x)\not\succeq -\sqrt{L_d\epsilon}P(x)$. We show in Lemma~\ref{lemma:escape} that perturbing $x$ with an isotropic noise followed by $\tilde{\mathcal{O}}(\epsilon^{-1/2})$ iterations of \ref{PPrecGD} reduces the the objective function by $\tilde{\Omega}(\epsilon^{3/2})$. 
\end{itemize}
Combining the above two scenarios, we show that PMGD decreases the objective value by $\tilde{\Omega}(\epsilon^2)$ per iteration (on average). Therefore, it takes at most $\tilde{\mathcal{O}}((f(x_0)-f^*)\epsilon^{-2})$ iterations to reach a $\epsilon$-second order point in local norm.

\begin{lemma}[Large gradient $\implies$ large decrement]\label{lemma:grad_decrement}
	Let $f$ is $\ell_{1}$-gradient Lipschitz, and let $p_{\lb}I\preceq P(x)\preceq p_\ub I$ for every $x$. Suppose that $x$ satisfies
	$\|\nabla f(x)\|_{x}^{*}>\epsilon$, and define
	\[
	x_{+}=x-\alpha P(x)^{-1}\nabla f(x)
	\]
	with step-size $\alpha=p_{\lb}/(2\ell_1)$. Then, we have
	\[
	f(x_{+})-f(x)<-\frac{p_{\lb}}{4\ell_1}\epsilon^{2}.
	\]
\end{lemma}
\begin{proof}
	Due to the gradient Lipschitz continuity of $f(x)$, we have:
	\begin{align}
	f(x_{+}) & \le f(x)+\alpha\inner{\nabla f(x)}{-P(x)^{-1}\nabla f(x)}+\frac{\ell_1}{2p_{\lb}}\alpha^{2}\|P(x)^{-1}\nabla f(x)\|_{x}^{2}\\
	& =f(x)-\alpha(\|\nabla f(x)\|_{x}^{*})^{2}\left(1-\frac{\ell_1}{2p_{\lb}}\alpha\right),\\
	& \le f(x)-\frac{\alpha}{2}(\|\nabla f(x)\|_{x}^{*})^{2}\\
	&\leq f(x)-\frac{p_\lb}{4\ell_1}\epsilon^2,\label{eq_grad_dec}
	\end{align}
	where in the last inequality we used
	the optimal step-size $\alpha=p_{\lb}/(2\ell_1)$ and the assumption $\|\nabla f(x)\|_x^*>\epsilon$.
\end{proof}
\begin{lemma}[Escape from saddle point]\label{lemma:escape}
	\label{lem:escape}Let $f$ be $\ell_1$-gradient and $\ell_{2}$-Hessian
	Lipschitz. Moreover, let $p_{\lb}I\preceq P(x)\preceq p_\ub I$ and $\left\|P(x)-P(y)\right\|\leq L_P\|x-y\|$ for every $x$ and $y$.
	Suppose that $\bar{x}$ satisfies $\|\nabla f(\bar{x})\|_{\bar{x}}^{*}\le\epsilon$ and $\nabla^{2}f(\bar{x})\not\succeq -\sqrt{L_d\epsilon}\cdot P(\bar{x})$. Then, the PMGD defined as
	\[
	\qquad x_{k+1}=x_{k}-\alpha P(x_{k})^{-1}\nabla f(x_{k}),\qquad \text{starting at}\qquad
	x_{0}=\bar{x}+\alpha\cdot\xi,
	\]
	with step-size $\alpha=p_{\lb}/(2\ell_1)$ and initial perturbation $\xi\sim \B(\beta)$
	with $\beta=\epsilon/(400 L_d\iota^{3})$ achieves the following decrement with
	probability of at least $1-\delta$ 
	\[
	f(x_{t})-f(\bar{x})\le-\underbrace{\frac{1}{50\iota^{3}}\sqrt{\frac{\epsilon^{3}}{L_d}}}_{:=\mathcal{F}}\qquad\text{ after } \mathcal{T}=\frac{\ell_1}{p_\lb\sqrt{ L_{d}\epsilon}}\iota\text{ iterations,}
	\]
	where $L_d = {5\max\{\ell_2,L_P \ell_1\sqrt{p_\ub}\}}/{p_\lb^{2.5}}$ and $\iota=c\cdot\log({p_{\ub}}d\ell_1(f(x_0)-f^*)/(p_\lb \ell_2\epsilon\delta))$ for some absolute constant $c$. 
\end{lemma}

Before presenting the sketch of the proof for Lemma~\ref{lemma:escape}, we complete the proof of Theorem~\ref{thm:pmgd} based on Lemmas~\ref{lemma:grad_decrement} and~\ref{lemma:escape}.

\begin{proof} [Theorem~\ref{thm:pmgd}.] 
Let us define $T = 2\left({\mathcal{T}}/{\mathcal{F}}+ 4\ell_1/(p_{\lb}\epsilon^2)\right)\cdot{(f(x_0)-f^*)}$. By contradiction, suppose that $x_t$ is \textit{not} a $\epsilon$-second order stationary point in local norm for any $t\leq T$. This implies that we either have $\|\nabla f(x_t)\|_{x_t}^{*}>\epsilon$, or $\|\nabla f({x_t})\|_{{x_t}}^{*}\le\epsilon$ and $\nabla^{2}f({x_t})\not\succeq -\sqrt{L_d\epsilon}\cdot P({x}_t)$ for every $t\leq T$. Define $T_1$ as the number of iterations that satisfy $\|\nabla f(x_t)\|_{x_t}^{*}>\epsilon$. Similarly, define $T_2$ as the number of iterations that satisfy $\|\nabla f({x_t})\|_{{x_t}}^{*}\le\epsilon$ and $\nabla^{2}f({x_t})\not\succeq -\sqrt{L_d\epsilon}\cdot P({x}_t)$. Evidently, we have $T_1+T_2 = T$. We divide our analysis into two parts:
\begin{itemize}
	\item Due to the definition of $T_2$, and in light of Lemma~\ref{lemma:escape}, we perturb the metric gradient at least $T_2/\mathcal{T}$ times. After each perturbation followed by $\mathcal{T}$ iterations, the PMGD reduces the objective function by at least $\mathcal{F}$ with probability $1-\delta$. Since the objective value cannot be less than $f^*$, we have
	$$
	f(x_0)-(\mathcal{F}/\mathcal{T})T_2\geq f^*\implies T_2\leq (\mathcal{T}/\mathcal{F})\cdot ({f(x_0)-f^*}),
	$$
	which holds with probability of at least
	$$
	1-(T_2/\mathcal{T})\delta\geq 1-\frac{f(x_0)-f^*}{\mathcal{F}}\cdot\delta = 1-\delta',
	$$
	where $\delta$ is chosen to be small enough so that $\delta' = ({f(x_0)-f^*})\delta/{\mathcal{F}} < 1$.
	\item Excluding $T_2$ iterations that are within $\mathcal{T}$ steps after adding the perturbation, we are left with $T_1$ iterations with $\|\nabla f(x_t)\|_{x_t}^{*}>\epsilon$. According to Lemma~\ref{lemma:grad_decrement}, the PMGD reduces the objective by $p_{\lb}\epsilon^2/(4\ell_1)$ at every iteration $x_t$ that satisfies $\|\nabla f(x_t)\|_{x_t}^{*}>\epsilon$. Therefore, we have 
	$$
	f(x_0)-(p_{\lb}\epsilon^2/(4\ell_1))T_1\geq f^*\implies T_1\leq (4\ell_1/(p_{\lb}\epsilon^2))\cdot ({f(x_0)-f^*})
	$$
\end{itemize}
Recalling the definition of $T$, the above two cases imply that $T_1+T_2< T$ with probability of at least $1-\delta'$, which is a contradiction. Therefore, $\|\nabla f({x_t})\|_{{x_t}}^{*}\le\epsilon$ and $\nabla^{2}f({x_t})\succeq -\sqrt{L_d\epsilon}\cdot P({x}_t)$ after at most $2\left({\mathcal{T}}/{\mathcal{F}}+ 4\ell_1/(p_{\lb}\epsilon^2)\right)\cdot{(f(x_0)-f^*)}$ iterations. Finally, note that 
\begin{align*}
	T& = \mathcal{O}\left(\left(\frac{\ell_1}{p_\lb^2\epsilon^2}+\frac{\ell_1}{p_{\lb}\epsilon^2}\right)\cdot (f(x_0)-f^*)\cdot \iota^4\right)\\
	& = \tilde{\mathcal{O}}\left(\mathcal{C}\cdot\frac{(f(x_0)-f^*)}{\epsilon^2}\right).
\end{align*}
This completes the proof. 
\end{proof}

Finally, we explain the proof of Lemma~\ref{lemma:escape}. To streamline the presentation, we only provide a sketch of the proof and defer the detailed arguments to the appendix.
\paragraph{Sketch of the proof for Lemma~\ref{lemma:escape}.}
The proof of this lemma follows that of~\citep[Lemma 5.3]{jin2021nonconvex} with a few key differences to account for the metric function $P(x)$, which will be explained below. 

Consider two sequences $\{x_t\}_{t=0}^T$ and $\{y_t\}_{t=0}^T$ generated by PMGD and initialized at $x_0=\bar{x}+\alpha\zeta_1$ and $y_0=\bar{x}+\alpha\zeta_1$, for some $\zeta_1,\zeta_2\in\mathbb{B}(\beta)$. By contradiction, suppose that Lemma~\ref{lem:escape} does not hold for either sequences $\{x_t\}_{t=0}^{\mathcal{T}}$ and $\{y_t\}_{t=0}^{\mathcal{T}}$, i.e., $f(x_{t})-f(\bar{x})\geq -\mathcal{F}$ and $f(y_{t})-f(\bar{x})\geq -\mathcal{F}$ for $t\leq \mathcal{T}$. That is, PMGD did not make sufficient decrement along $\{x_t\}_{t=0}^{\mathcal{T}}$ and $\{y_t\}_{t=0}^{\mathcal{T}}$. As a critical step in our proof, we show in the appendix (see Lemma~\ref{lem_bound}) that such small decrement in the objective function implies that both sequences $\{x_t\}_{t=0}^{\mathcal{T}}$ and $\{y_t\}_{t=0}^{\mathcal{T}}$ must remain close to $\bar{x}$. The closeness of $\{x_t\}_{t=0}^T$ and $\{y_t\}_{t=0}^T$ to $\bar{x}$ implies that their differences can be modeled as a quadratic function with a small deviation term:
\begin{align}\label{eq_dist}
&P({\bar{x}})^{1/2}\left(x_{t+1}-y_{t+1}\right) \nonumber\\
=& P({\bar{x}})^{1/2}\left(x_{t}-y_{t}\right) - \alpha P({\bar{x}})^{1/2}(P(x)^{-1}\nabla f(x_t)-P(y)^{-1}\nabla f(y_t)) \nonumber\\
=& \left(I - \alpha P({\bar{x}})^{-1/2}\nabla^2f(\bar{x})P(\bar{x})^{-1/2}\right)P({\bar{x}})^{1/2}(x_{t}-y_{t})+\alpha\xi(\bar{x}, x_t,y_t)
\end{align}
\begin{sloppypar}
\noindent where the deviation term $\xi(\bar{x},x_t,y_t)$ is defined as 
$$
\xi(\bar{x},  x_t,y_t) \!=\! P({\bar{x}})^{1/2}\left(P({\bar{x}})^{-1}\nabla^2f(\bar{x})(x_{t}\!-\!y_{t})\!-\!\left(P({x_t})^{-1}\nabla f(x_t)\!-\!P({y_t})^{-1}\nabla f(y_t)\right)\right)
$$
\citet{jin2021nonconvex} showed that, when $P({x_t}) = P({y_t}) = P({\bar{x}}) = I$, the deviation term $\xi(\bar{x},  x_t,y_t)$ remains small. However, such argument cannot be readily applied to the PMGD, since the metric function $P(x)$ can change drastically throughout the iterations. The crucial step in our proof is to show that the rate of change in $P(x)$ slows down drastically around stationary points.
\end{sloppypar}
\begin{lemma}[Informal]\label{lem_dist_informal}
	Let $f$ be $\ell_1$-gradient and $\ell_2$-Hessian Lipschitz. Let $P(x)$ be well-conditioned and Lipschitz continuous. Suppose that $u$ satisfies $\|\nabla f(u)\|_u^*\leq \epsilon$. Then, we have
	\begin{align*}
	\|\zeta(u,x,y)\|\leq C_1 \max\{\|x-u\|,\|y-u\|\}\left\|P(u)^{1/2}(x-y)\right\|+C_2\epsilon \left\|P(u)^{1/2}(x-y)\right\|,
	\end{align*}
	for constants $C_1$ and $C_2$ that only depend on $\ell_1, \ell_2, L_P, p_{\lb}, p_{\ub}$.
\end{lemma}
The formal version of Lemma~\ref{lem_dist_informal} can be found in appendix (see Lemma~\ref{lem_distance}).
Recall that due to our assumption, the term $\max\{\|x_t-\bar{x}\|,\|y_t-\bar{x}\|\}$ must remain small for every $t\leq\mathcal{T}$. Therefore, applying Lemma~\ref{lem_dist_informal} with $u=\bar{x}$, $x=x_t$, and $y=y_t$ implies that $\|\zeta(\bar{x}, x_t,y_t)\| = o(\|x_t-y_t\|)$. Therefore, (\ref{eq_dist}) can be further approximated as
\begin{align}\label{eq_eig}
P(\bar{x})^{1/2}\left(x_{t}-y_{t}\right) \approx \left(I - \alpha P({\bar{x}})^{-1/2}\nabla^2f(\bar{x})P(\bar{x})^{-1/2}\right)^tP({\bar{x}})^{1/2}(x_{0}-y_{0})
\end{align}
\begin{sloppypar}
	\noindent Indeed, the above approximation enables us to argue that $P^{1/2}_{\bar{x}}\left(x_{t}-y_{t}\right)$ evolves according to a power iteration.
	In particular, suppose that $x_0$ and $y_0$ are picked such that $P({\bar{x}})^{1/2}(x_0-y_0) = cv$, where $v$ is the eigenvector corresponding the smallest eigenvalue $\lambda_{\min}\left(P({\bar{x}})^{-1/2}\nabla^2f(\bar{x})P({\bar{x}})^{-1/2}\right) = -\gamma \leq -\sqrt{L_d\epsilon}<0$. With this assumption, (\ref{eq_eig}) can be approximated as the following power iteration:
\end{sloppypar}
\begin{align}\label{eq_power}
	P(\bar{x})^{1/2}\left(x_{t}-y_{t}\right) \approx c (1+\alpha\gamma)^t v
\end{align}
Suppose that $\{x_{t}\}_{t=0}^{\mathcal{T}}$ does not escape the strict saddle point, i.e., $f(x_t)-f(x_0)\geq -\mathcal{F}$. This implies that $\{x_{t}\}_{t=0}^{\mathcal{T}}$ remains close to $\bar{x}$. On the other hand, (\ref{eq_power}) implies that the sequence $\{y_{t}\}_{t=0}^{\mathcal{T}}$ must diverge from $\bar{x}$ \textit{exponentially fast}, which is in direct contradiction with our initial assumption that  $\{y_{t}\}_{t=0}^{\mathcal{T}}$ and $\bar{x}$ remain close. This in turn implies that $f(x_{t})-f(\bar{x})< -\mathcal{F}$ or $f(y_{t})-f(\bar{x})< -\mathcal{F}$. In other words, at least one of the sequences $\{x_{t}\}_{t=0}^{\mathcal{T}}$ and $\{y_{t}\}_{t=0}^{\mathcal{T}}$ must escape the strict saddle point. Considering $\{y_{t}\}_{t=0}^{\mathcal{T}}$ as a perturbed variant of $\{x_{t}\}_{t=0}^{\mathcal{T}}$, the above argument implies that it takes an \textit{exponentially small} perturbation in the direction of $v$ for PMGD to escape the strict saddle point $\bar{x}$. The sketch of the proof is completed by noting that such perturbation in the direction $v$ is guaranteed to occur (with high probability) with a random isotropic perturbation of $x_0$.

\section{Numerical Experiments}

In this section, we provide numerical experiments to illustrate our
theoretical findings. All experiments are implemented
using MATLAB 2020b, and performed with a 2.6 Ghz 6-Core Intel Core
i7 CPU with 16 GB of RAM.

We begin with numerical simulations for our local convergence result,
\thmref{local}, where we proved that PrecGD with an appropriately
chosen regularizer $\eta$ converges linearly towards the optimal
solution $M^{\star}$, at a rate that is independent of both ill-conditioning
and overparameterization. In contrast, gradient descent is slowed
down significantly by both ill-conditioning and overparameterization.

We plot the convergence of GD and PrecGD for three choices of $\phi(\cdot)$, which correspond to the
problems of low-rank matrix recovery~\citep{candes2011tight}, 1-bit
matrix sensing~\citep{davenport20141}, and phase retrieval~\citep{candes2013phaselift}
respectively: 
\begin{itemize}
\item \textbf{Low-rank matrix recovery with $\ell_{2}$ loss}. Our goal is to find a low-rank matrix $M^{\star}\succeq 0$ that satisfies
$\mathcal{A}(M^{\star})=b$, where $\mathcal{A}:\mathbb{R}^{n\times n}\to\mathbb{R}^{m}$
is a linear operator and $b\in\mathbb{R}^{m}$ is given. To find $M^{\star}$,
we minimize the objective 
\begin{equation}
\phi(M)=\|\mathcal{A}(M)-b\|^{2}\label{eq:phim}
\end{equation}
subject to the constraint that $M$ is low-rank. The Burer-Montiero
formulation of this problem then becomes: minimize $f(X)=\|\mathcal{A}(XX^{T})-b\|^{2}$
with $X\in\mathbb{R}^{n\times r}$.
\item \textbf{1-bit matrix sensing}. The goal of 1-bit matrix recovery is
to recovery a low-rank matrix $M^{\star}\succeq0$
through 1-bit measurements of each entry $M_{ij}$. Each measurement on the $M_{ij}$ is equal to $1$ with probability $\sigma(M_{ij})$
and $0$ with probability $1-\sigma(M_{ij})$, where $\sigma(\cdot)$
is the sigmoid function. After a number of measurements have been
taken, let $\alpha_{ij}$ denote the percentage of measurements on
the $(i,j)$-entry that is equal to 1. To recover $M^{\star}$, we
want to find the maximum likelihood estimator for $M^{\star}$ by
minimizing 
\begin{equation}
\phi(M)=\sum_{i=1}^{n}\sum_{j=1}^{n}(\log(1+e^{M_{ij}})-\alpha_{ij}M_{ij}).\label{eq:1bit}
\end{equation}
subject to the constraint $\mathrm{rank}(M)\leq r$.
\item \textbf{Phase retrieval}. The goal of phase retrieval is to recover
a vector $z\in\mathbb{R}^{d}$ from $m$ measurements of the form
$y_{i}=|\langle a_{i},z\rangle|^{2}$, where $a_{i},i=1,\dots,m$
are measurement vectors in $\mathbb{C}^{d}$. Equivalently, we can
view this problem as recovering a rank-1 matrix $zz^{*}$ from $m$
linear measurements of the form $y_{i}=\langle a_{i}a_{i}^{*},zz^{*}\rangle$.
To find $zz^{*}$ we minimize the $\ell_{2}$ loss 
\begin{equation}
\phi(M)=\sum_{i=1}^{m}\left(\langle a_{i}a_{i}^{*},M\rangle-y_{i}\right)^{2}\label{eq:phaser}
\end{equation}
subject to the constraint that $M$ is rank-1. 
\end{itemize}
One can readily check that both low-rank matrix
recovery (\ref{eq:phim}) and 1-bit matrix sensing (\ref{eq:1bit-1})
satisfy $(\mu,r)$-restricted strong convexity (see \citealt{li2019non}
for details), so our theoretical results predict that PrecGD will converge linearly.
While phase retrieval (\ref{eq:phaser})
\textit{does not} satisfy restricted strong convexity, we
will nevertheless see that PrecGD continues to converge linearly for phase retrieval.
This indicates that PrecGD will continue to work well for more general
optimization problems that present a low-rank structure. We leave
the theoretical justifications of these numerical results as future
work.

\subsection{\label{sec:gauss}Low-rank matrix recovery with $\ell_{2}$ loss}

In this problem we assume that there is a $n\times n$, rank $r^{\star}$
matrix $M^{\star}\succeq0$, which we call the ground truth, that
we cannot observe directly. However, we have access to linear measurements
of $M^{\star}$ in the form $b=A(M^{\star})$. Here the linear operator
$\mathcal{A}:\mathbb{R}^{n\times n}\to\mathbb{R}^{m}$ is defined
as $\mathcal{A}(M^{\star})=[\langle A_{1},M^{\star}\rangle,\dots\langle A_{m},M^{\star}\rangle]$,
where each $A_{i}$ is a fixed matrix of size $n\times n$. The goal
is to recovery $M^{\star}$ from $b$, potentially with $m\ll n^{2}$
measurements by exploiting the low-rank structure of $M^{\star}$.
This problem has numerous applications in areas such as colloborative
filtering \citep{rennie2005fast}, quantum state tomography \citep{gross2010quantum},
power state estimation \citep{zhang2019spurious}.

To recover $M^{\star}$, we minimize the objective $\phi(M)$ in (\ref{eq:phim})
by solving the unconstrained problem
\[
\min_{X\in\mathbb{R}^{n\times r}}f(X)=\phi(XX^{T})=\|\mathcal{A}(XX^{T})-b\|^{2}
\]
using both GD and PrecGD. 
To gauge the effects of ill-conditioning, in our experiments we consider
two choices of $M^{\star}$: one well-conditioned and one ill-conditioned.
In the well-conditioned case, the ground truth is a rank-2
($r^{*}=2$) positive semidefinite matrix of size $100\times100$, where
both of the non-zero eigenvalues are $1$. To generate $M^{\star}$,
we compute $M^{\star}=Q^{T}\Lambda Q$, where $\Lambda=\mathrm{diag}(1,1,0,\dots,0)$
and $Q$ is a random orthogonal matrix of size $n\times n$ (sampled
uniformly from the orthogonal group). In the ill-conditioned case,
we set $M^{\star}=Q^{T}\Lambda Q$, where $\Lambda=\mathrm{diag}(1,1/5,0,\dots,0)$.

For each $M^{\star}$ we perform two set of experiments: the exactly-parameterized
case with $r=2$ and the overparameterized case where $r=4$.
The step-size is set to $2\times10^{-6}$ for both GD and PrecGD in the first case and to $1\times10^{-5}$ in the latter case.
For PrecGD, the regularization parameter is set to $\eta=\|\nabla f(X)\|_{X,0}^{*}$.
Both methods are initialized near the ground truth. In particular, we compute $M^{\star}=ZZ^{T}$
with $Z\in\mathbb{R}^{n\times r}$ and choose the initial point as $X_{0}=Z+10^{-2}w$,
where $w$ is a $n\times r$ random matrix with standard Gaussian
entries. In practice, the closeness of a initial point to the ground truth can be guaranteed via spectral initialization (\citealt{chi2019nonconvex, tu2016low}). Finally, to ensure that $\phi(M)=\|A(M)-b\|^{2}$ satisfies restricted
strong convexity, we set the linear operator $\mathcal{A}:\mathbb{R}^{n\times n}\to\mathbb{R}^{m}$
to be $\mathcal{A}(M^{\star})=[\langle A_{1},M^{\star}\rangle,\dots\langle A_{m},M^{\star}\rangle]$,
where $m=3nr$ and each $A_{i}$ is a standard Gaussian matrix with
i.i.d. entries (\citealt{recht2010guaranteed}).

We plot the error $\|XX^{T}-M^{\star}\|_{F}$ versus the number of
iterations for both GD and PrecGD. The results of our experiments
for a well-conditioned $M^{\star}$ are shown on the first row of
Figure~\ref{fig:gauss}. We see here that if $M^{\star}$ is well-conditioned
and $r=r^{*}$, then GD converges at a linear rate and reaches machine
precision quickly. The performance of PrecGD is almost identical.
However once the search rank $r$ exceeds the true rank $r^{*}$,
then GD slows down significantly, as we can see from the figure on the
right. In contrast, PrecGD continue to converge at a linear rate,
obtaining machine accuracy within a few hundred iterations.

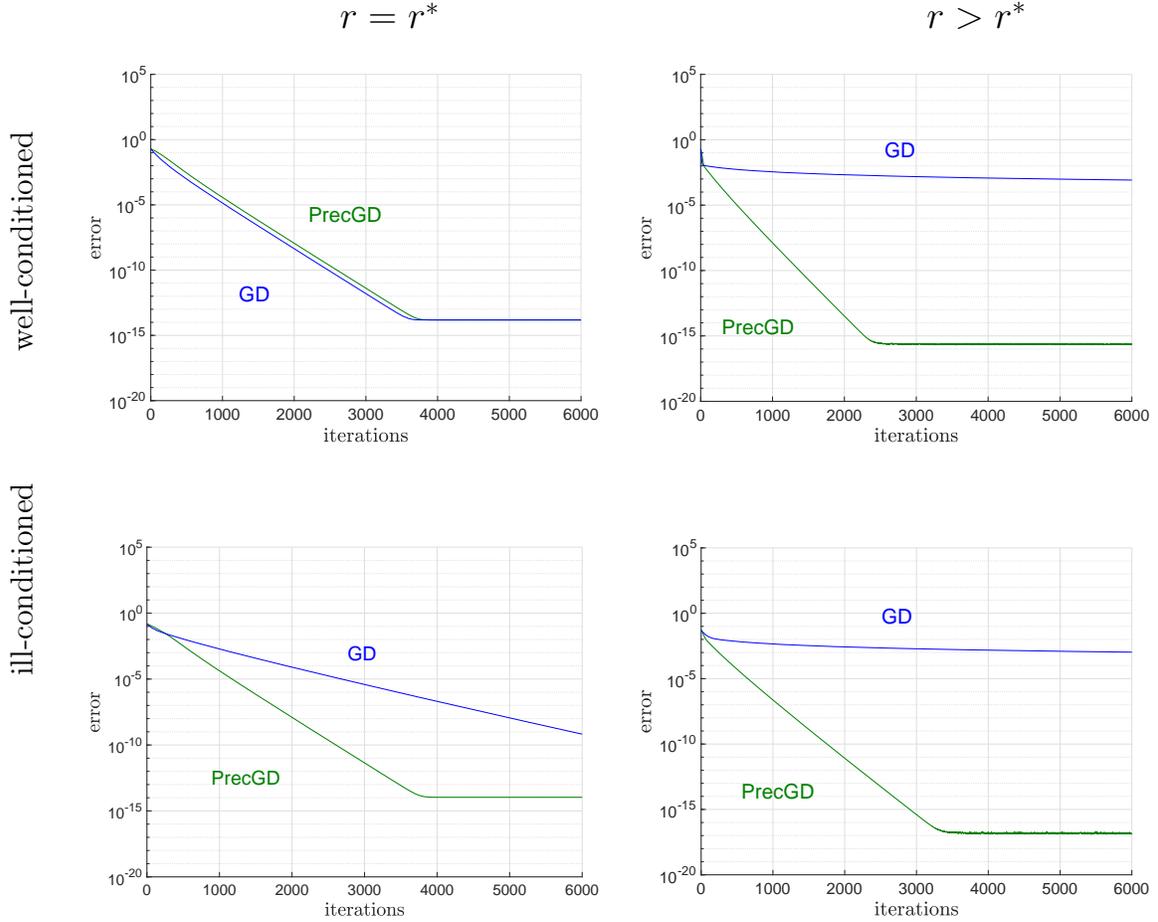
\begin{figure}[th]
\centering
\begin{tikzpicture}
\newcommand{\coord}{4.3}
\newcommand{\wid}{0.45}
\newcommand{\rl}{0.25}
\node (image) at (\rl,\coord) {
\includegraphics[width = \wid\textwidth]{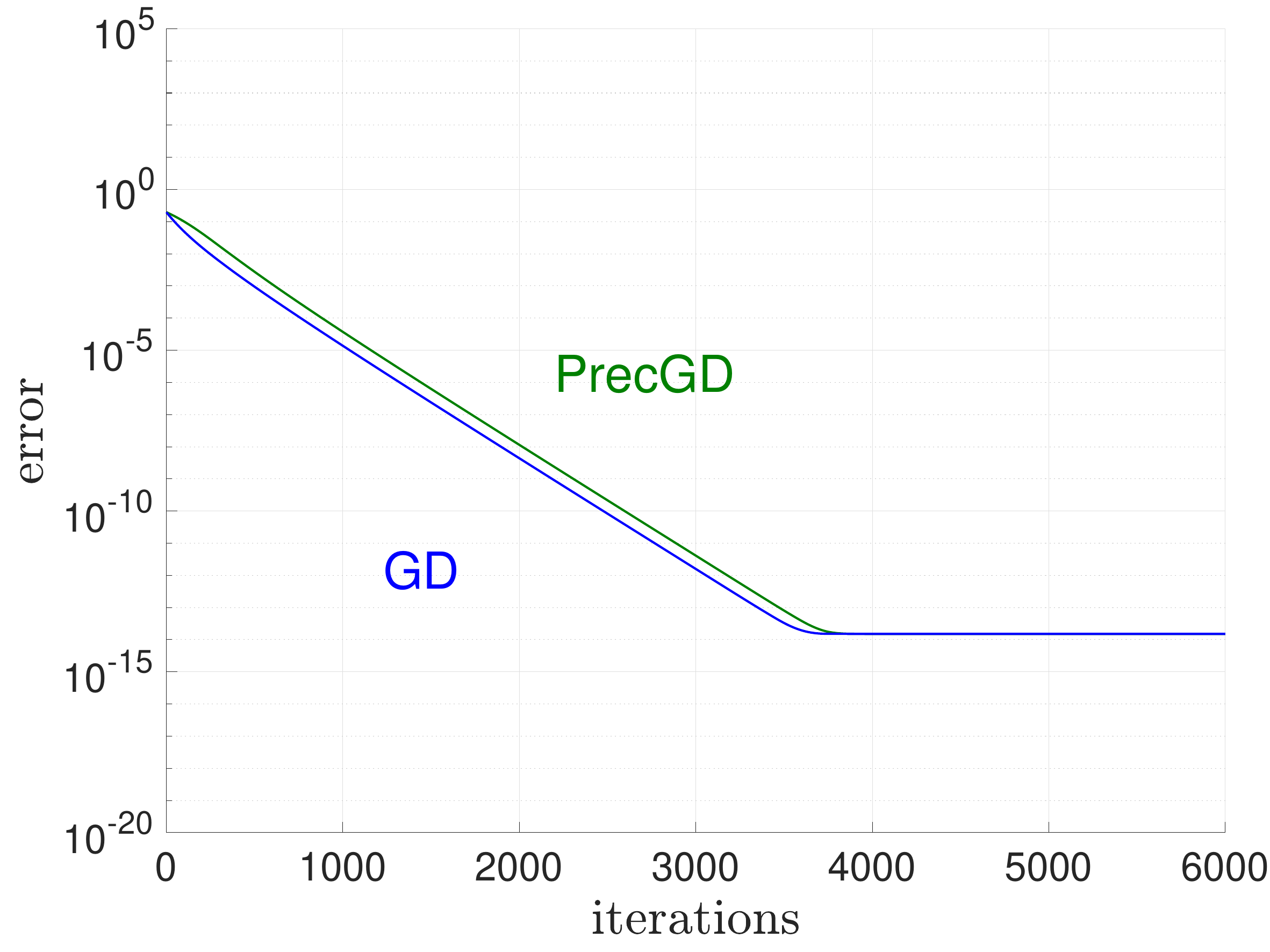}
};
\node (image) at (2\rl+\wid\textwidth,\coord) {
\includegraphics[width = \wid\textwidth]{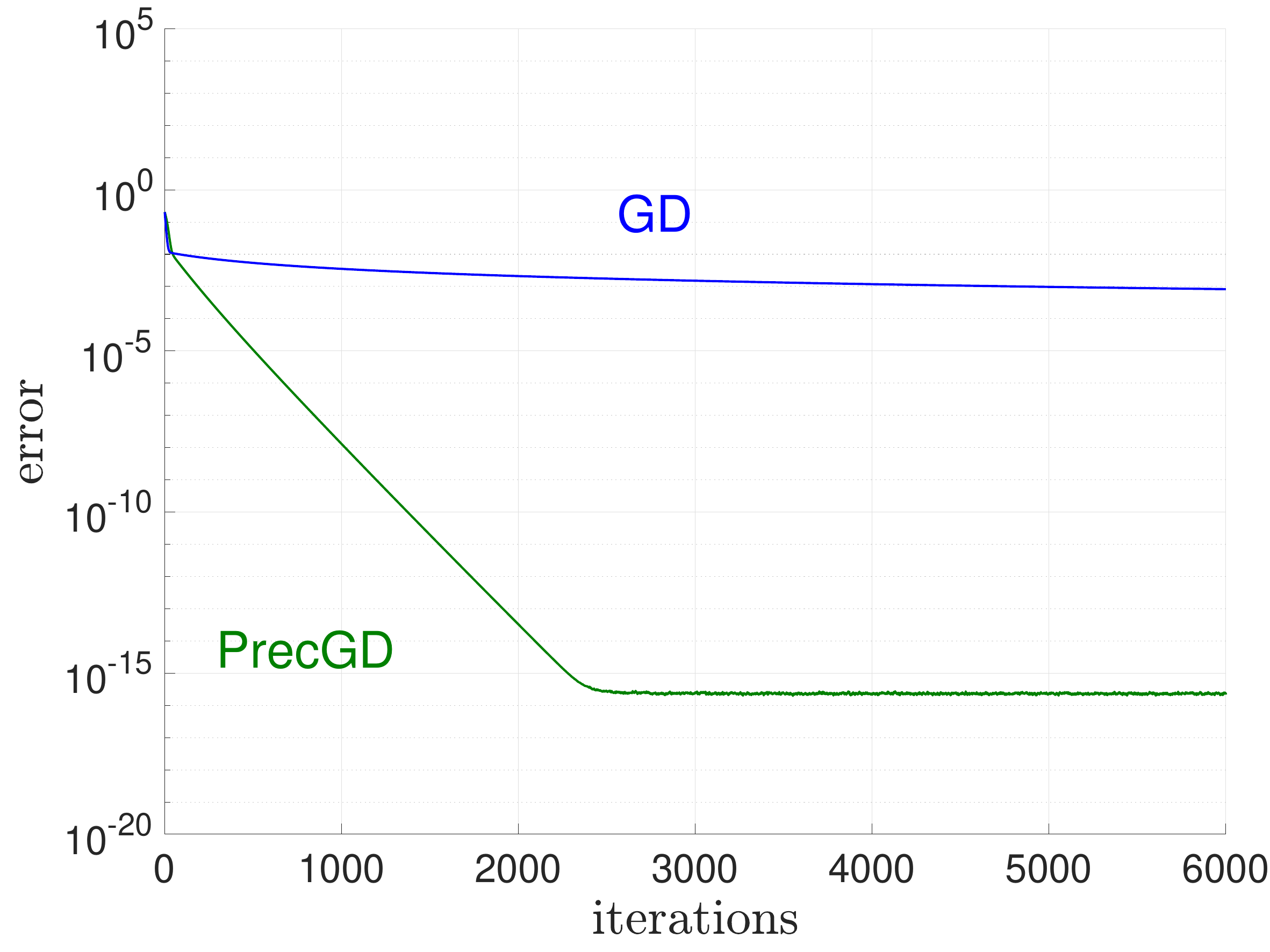}
};
\node (image) at (\rl,-2) {
\includegraphics[width = \wid\textwidth]{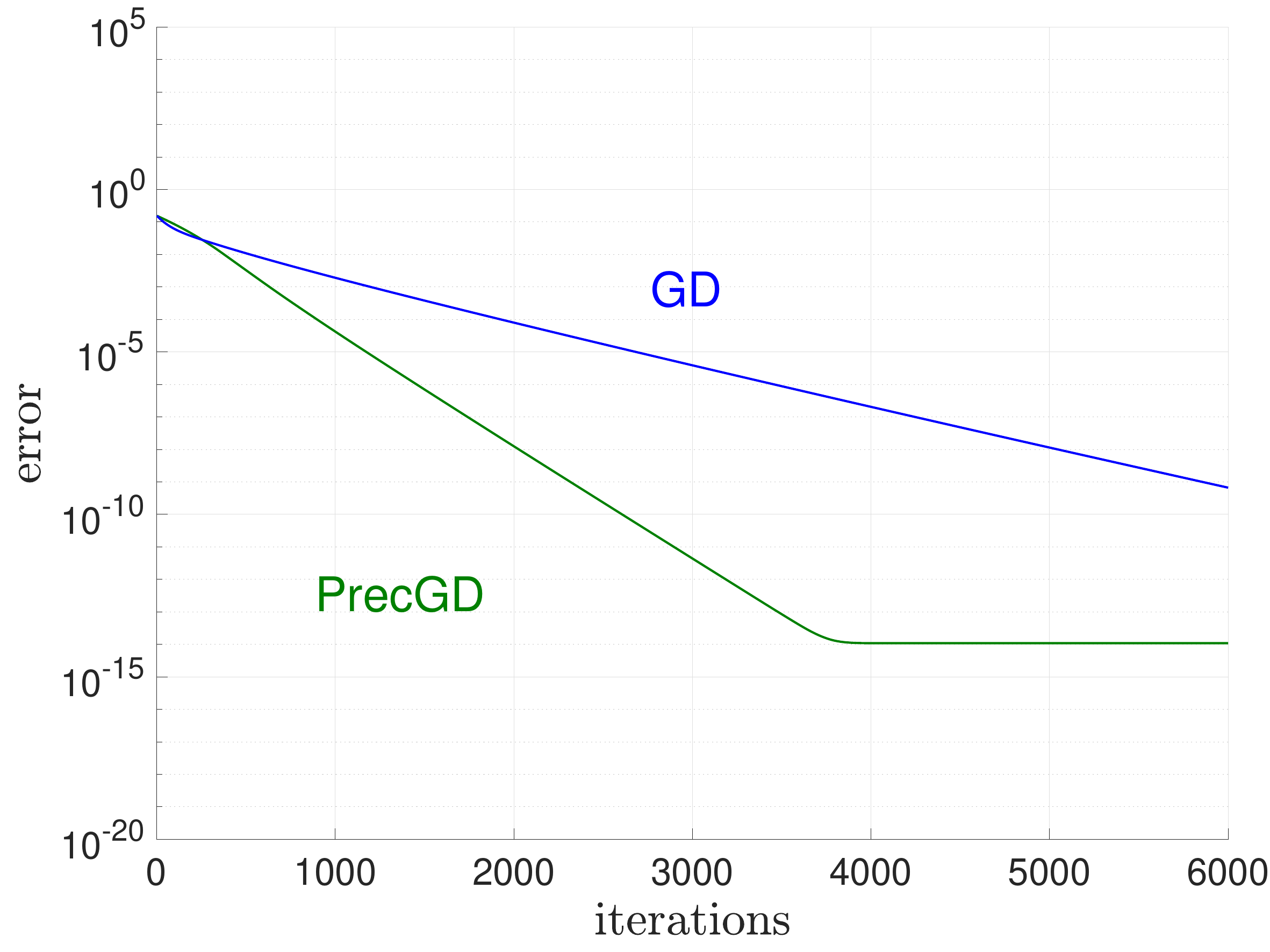}
};
\node (image) at (2\rl+\wid\textwidth,-2) {
\includegraphics[width = \wid\textwidth]{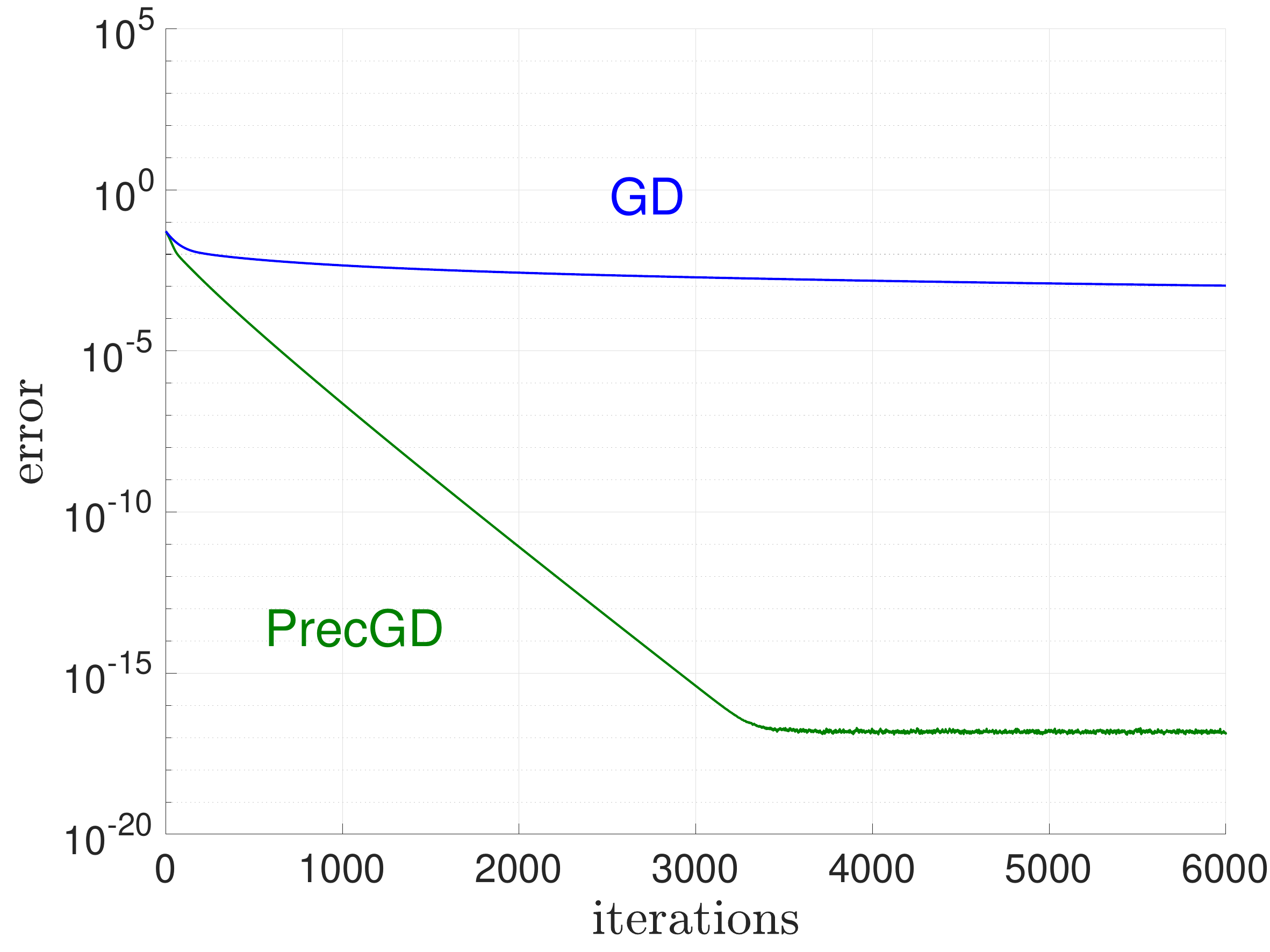}
};

\node[] at (0.9,7.5) {\Large $r=r^*$
};
\node[] at (8.7,7.5) {\Large $r>r^*$
};
\node[rotate=90] at (-4,0) {\large ill-conditioned
};
\node[rotate=90] at (-4,4.5) {\large well-conditioned
};
\end{tikzpicture}
\caption{\footnotesize
Low-rank matrix recovery with $\ell_{2}$ loss. \textbf{First row}: Well-conditioned ($\kappa=1$), rank-2
ground truth of size $100\times100$. The left panel shows the performance of GD and PrecGD for $r=r^* = 2$. Both
algorithms converge linearly to machine error.
The right panel shows the performance of GD and PrecGD for $r = 4$. The overparameterized GD converges sublinearly,
while PrecGD maintains the same converge rate. \textbf{Second row}: Ill-conditioned
($\kappa=5$), rank-2 ground truth of size $100\times100$. The left panel shows the performance of GD and PrecGD for $r=r^* = 2$. GD stagnates
due to ill-conditioning while PrecGD converges linearly. The right panel shows the performance of GD and PrecGD for $r = 4$. The overparameterized GD continues to stagnate, while PrecGD maintains the
same linear convergence rate.}
\label{fig:gauss}
\end{figure}

The results of our experiments for an ill-conditioned $M^{\star}$
are shown on the second row of Figure \ref{fig:gauss}. We can see
that ill-conditioning causes GD to slow down significantly, even in
the exactly-parameterized case. In the overparameterized case, GD
becomes even slower. On the other hand, PrecGD is fully agnostic
to ill-conditioning and overparameterization. In fact, as Theorem
1 shows, the convergence rate of PrecGD is unaffected by both ill-conditioning
and overparameterization.

In Figure \ref{fig:glob} , we also plot a comparison of PrecGD, ScaledGD \cite{tong2020accelerating} and GD with all three methods initialized at a random initial point and using the same step-size. Here, we see that both GD and PrecGD was able to converge towards the global solution, while ScaledGD behaves sporadically and diverges.
\begin{figure}
    \centering
    \includegraphics[scale=0.12]{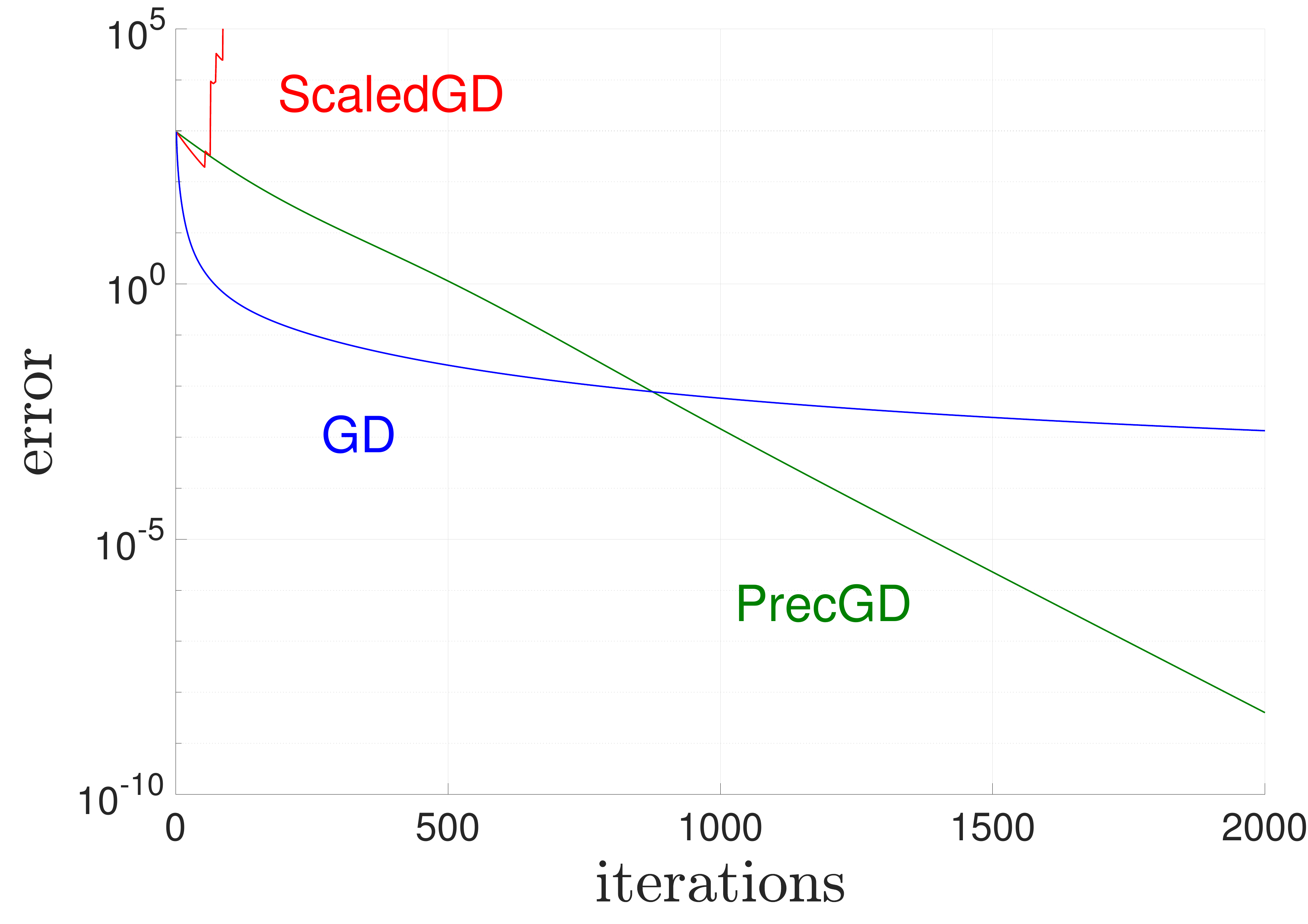}
    \caption{\textbf{PrecGD, ScaledGD and GD with random initialization.} Comparison of PrecGD against regular gradient descent (GD),
and the ScaledGD algorithm.
All three methods uses the same global Gaussian random initialization.  The same step-size $\alpha=2\times10^{-3}$
was used for all three algorithms. With $n=4$, $r=4$
and $r^{*}=2$, overparameterization causes gradient descent to slow
down to a sublinear rate. ScaledGD behaves sporadically and diverges. Only
PrecGD converges linearly to the global minimum.}
\label{fig:glob}
\end{figure}

\subsection{1-bit Matrix Sensing}

\begin{figure}
\centering
\begin{tikzpicture}
\newcommand{\coord}{4.3}
\newcommand{\wid}{0.45}
\newcommand{\rl}{0.25}
\node (image) at (\rl,\coord) {
\includegraphics[width = \wid\textwidth]{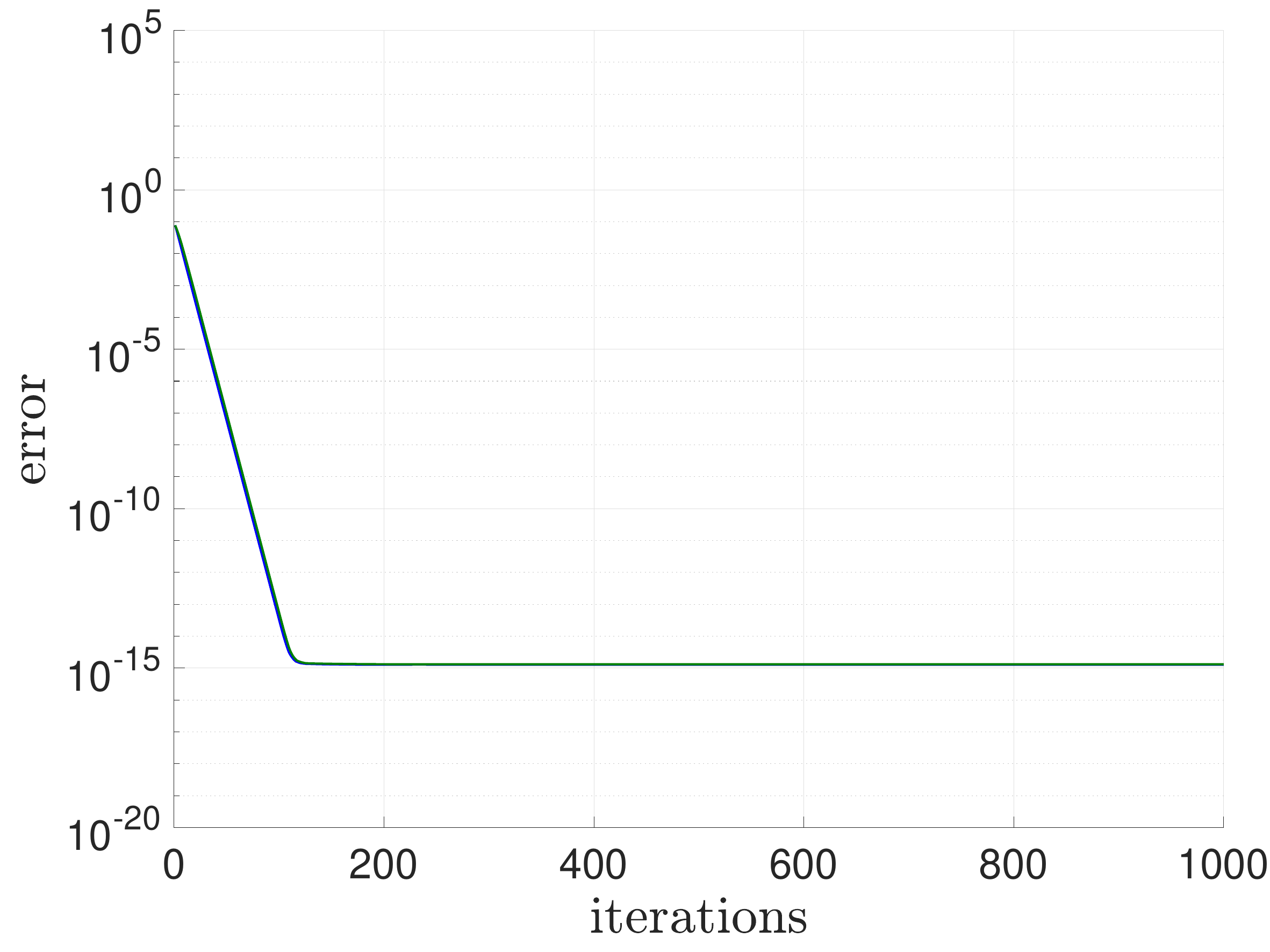}
};
\node (image) at (2\rl+\wid\textwidth,\coord) {
\includegraphics[width = \wid\textwidth]{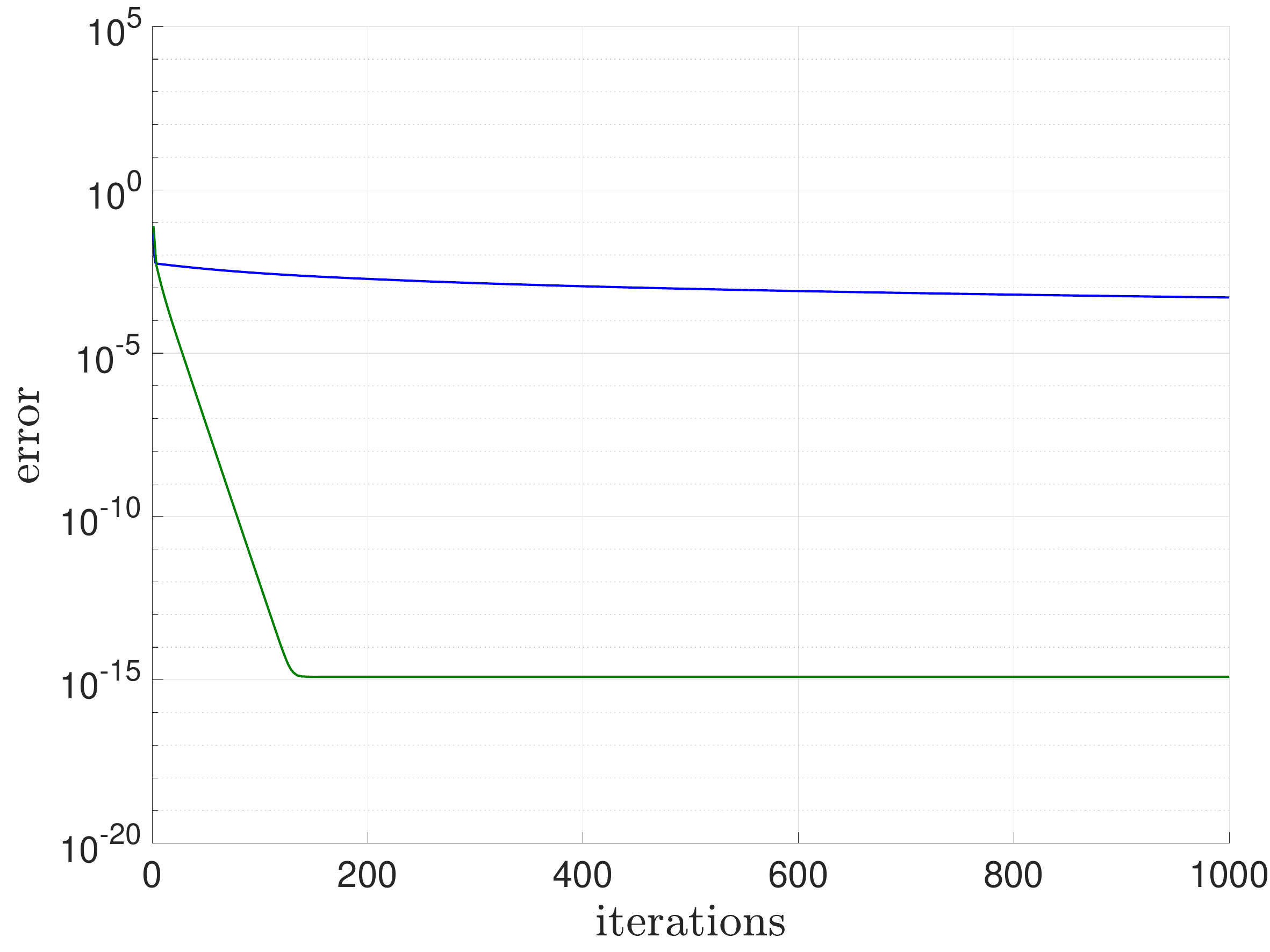}
};
\node (image) at (\rl,-2) {
\includegraphics[width = \wid\textwidth]{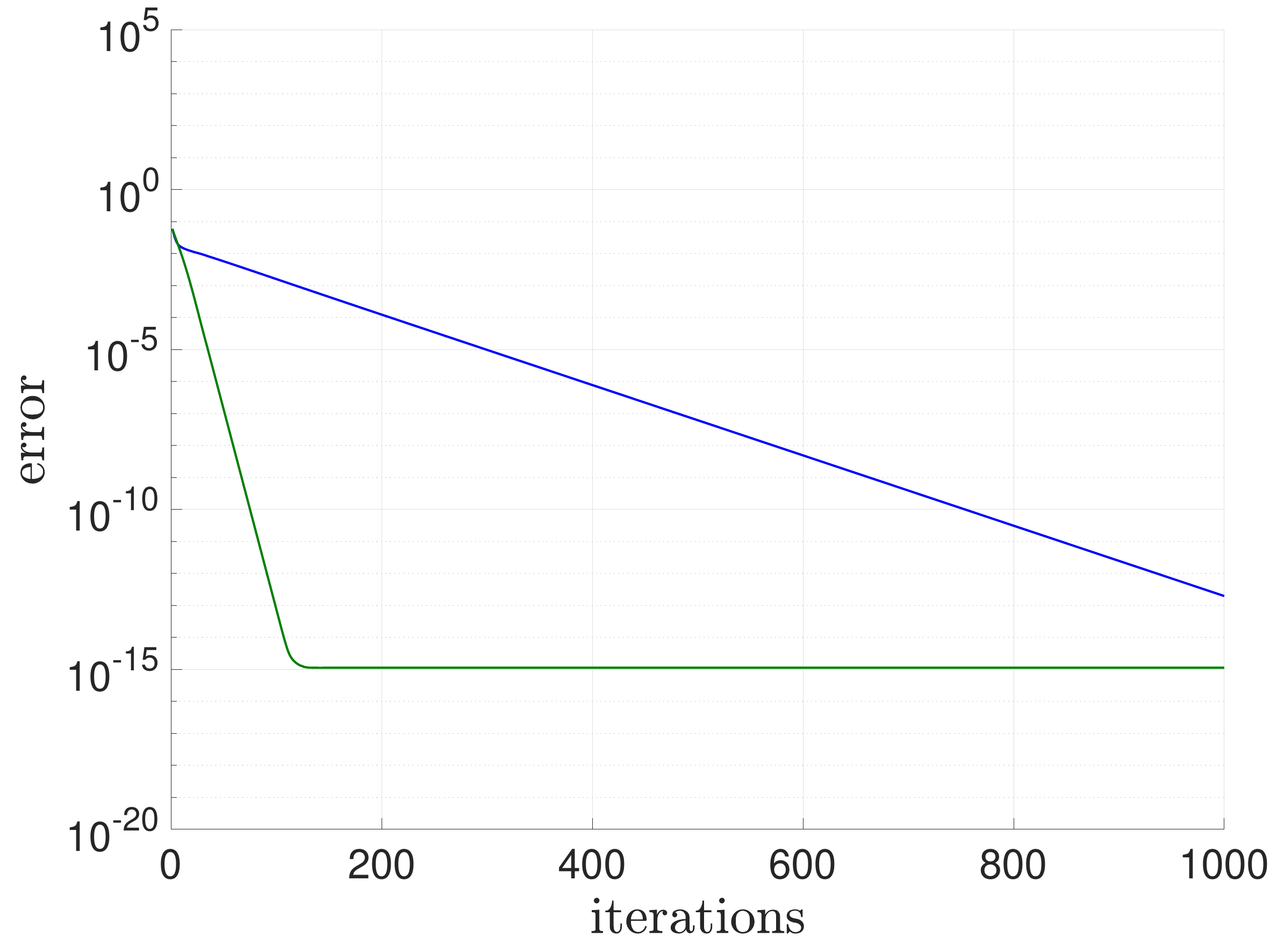}
};
\node (image) at (2\rl+\wid\textwidth,-2) {
\includegraphics[width = \wid\textwidth]{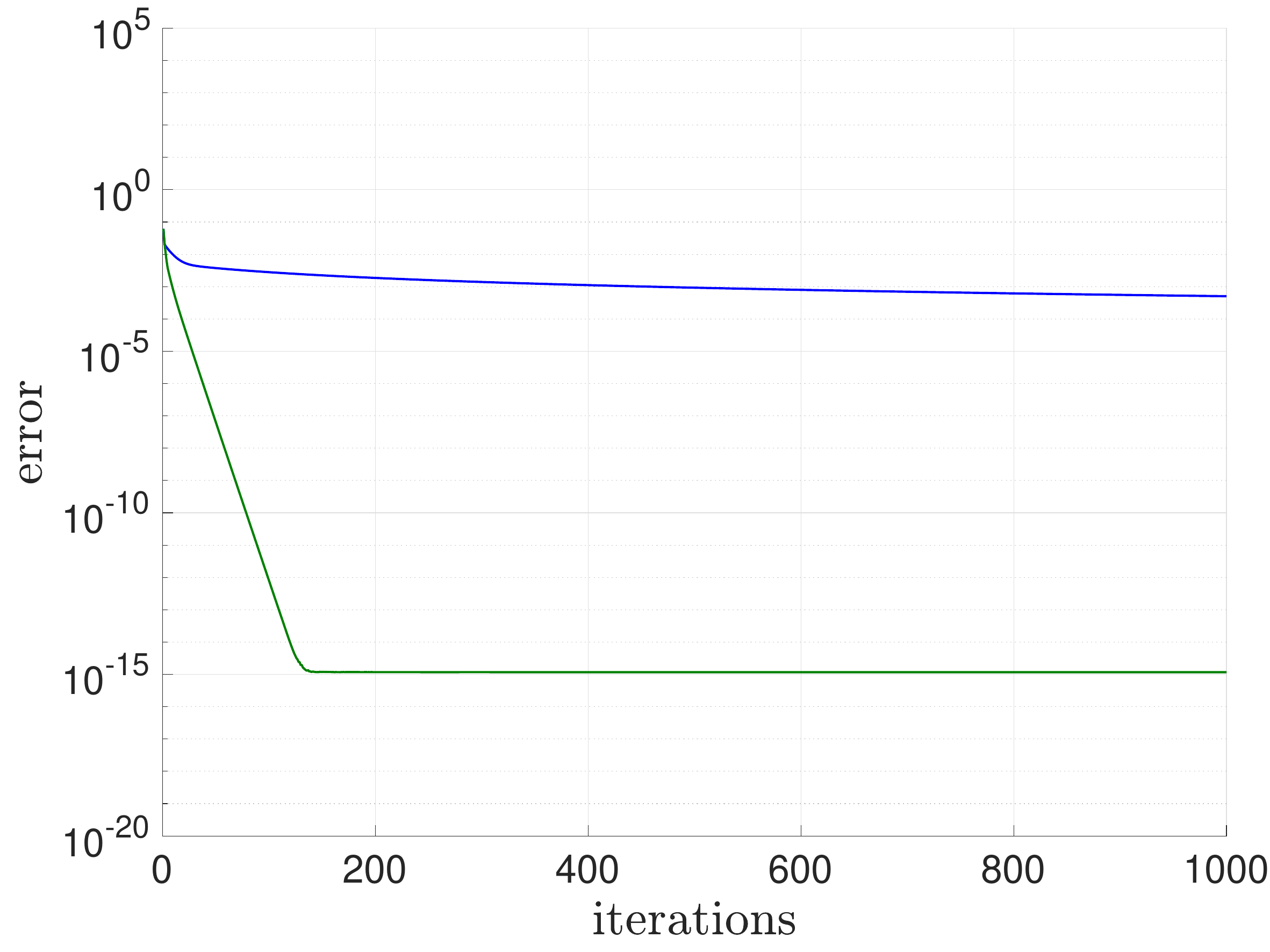}
};

\node[] at (0.9,7.5) {\Large $r=r^*$
};
\node[] at (8.7,7.5) {\Large $r>r^*$
};
\node[rotate=90] at (-4,0) {\large ill-conditioned
};
\node[rotate=90] at (-4,4.5) {\large well-conditioned
};
\end{tikzpicture}
\caption{\footnotesize
1-bit matrix sensing. \textbf{First row}: Well-conditioned ($\kappa=1$), rank-2
ground truth of size $100\times100$. The left panel shows the performance of GD and PrecGD for $r=r^* = 2$. Both
algorithms converge linearly to machine error.
The right panel shows the performance of GD and PrecGD for $r = 4$. The overparameterized GD converges sublinearly,
while PrecGD maintains the same converge rate. \textbf{Second row}: Ill-conditioned
($\kappa=10$), rank-2 ground truth of size $100\times100$. The left panel shows the performance of GD and PrecGD for $r=r^* = 2$. GD stagnates
due to ill-conditioning while PrecGD converges linearly. The right panel shows the performance of GD and PrecGD for $r = 4$. The overparameterized GD continues to stagnate, while PrecGD maintains the
same linear convergence rate.}
\label{fig:1bit}
\end{figure}
Similar to low-rank matrix recovery, in 1-bit matrix sensing we also
assume that there is a low rank matrix $M^{\star}\succeq0$,
which we call the ground truth, that we cannot observe directly, but
have access to a total number of $m$ 1-bit measurements of $M^{\star}$.
Each measurement of $M_{ij}$ 
is 1 with probability $\sigma(M_{ij})$ and 0 with probability $1-\sigma(M_{ij})$, where $\sigma(\cdot)$ is the sigmoid function.
This problem is a variant of the classical matrix completion problem
and appears in applications where only quantized observations are
available; see \citep{singer2008remark,gross2010quantum}
for instance.

Let $\alpha_{ij}$
denote the percentage of measurements on the $(i,j)$-entry that is
equal to 1. Then the MLE estimator can formulated as the minimizer
of 
\begin{equation}
\phi(M)=\sum_{i=1}^{n}\sum_{j=1}^{n}(\log(1+e^{M_{ij}})-\alpha_{ij}M_{ij}).\label{eq:1bit-1}
\end{equation}
It is easy to check that $\nabla^{2}\phi(M)$ is positive definite with bounded eigenvalues (see \citealt{li2019non}), so $\phi(M)$ satisfies the restricted
strong convexity, which is required by Theorem \ref{thm:local}.

To find the minimizer, we solve the problem $\min_{X\in\mathbb{R}^{n\times r}}\phi(XX^{T})$
using GD and PrecGD. For presentation, we assume that the
number of measurements $m$ is large enough so that $\alpha_{ij}=\sigma(M_{ij})$.
In this case the optimal solution of (\ref{eq:1bit-1}) is $M^{\star}$
and the error $\|XX^{T}-M^{\star}\|$ will go to zero when GD or PrecGD
converges.

In our experiments, we use exactly the same choices of well- and ill-conditioned $M^{\star}$
as in Section \ref{sec:gauss}.
The rest of the experimental set up is also the same. We perform two
set of experiments: (1) the exactly-parameterized case with $r=2$
and (2) the overparameterized case where $r=4$. Moreover, we use the same
initialization scheme and same regularization parameter $\eta=\|\nabla f(X)\|_{X,0}^{*}$
for PrecGD. The step-size is chosen to be $0.5$ in all four plots. 

 Our experiments are shown in Figure
\ref{fig:1bit}. We observe almost identical results as those of low-rank
matrix recovery in Figure \ref{fig:gauss}. In short, for 1-bit matrix sensing, both ill-conditioning
and overparameterization causes gradient descent to slow down significantly,
while PrecGD maintains a linear convergence rate independent of both.

\subsection{Phase Retrieval}

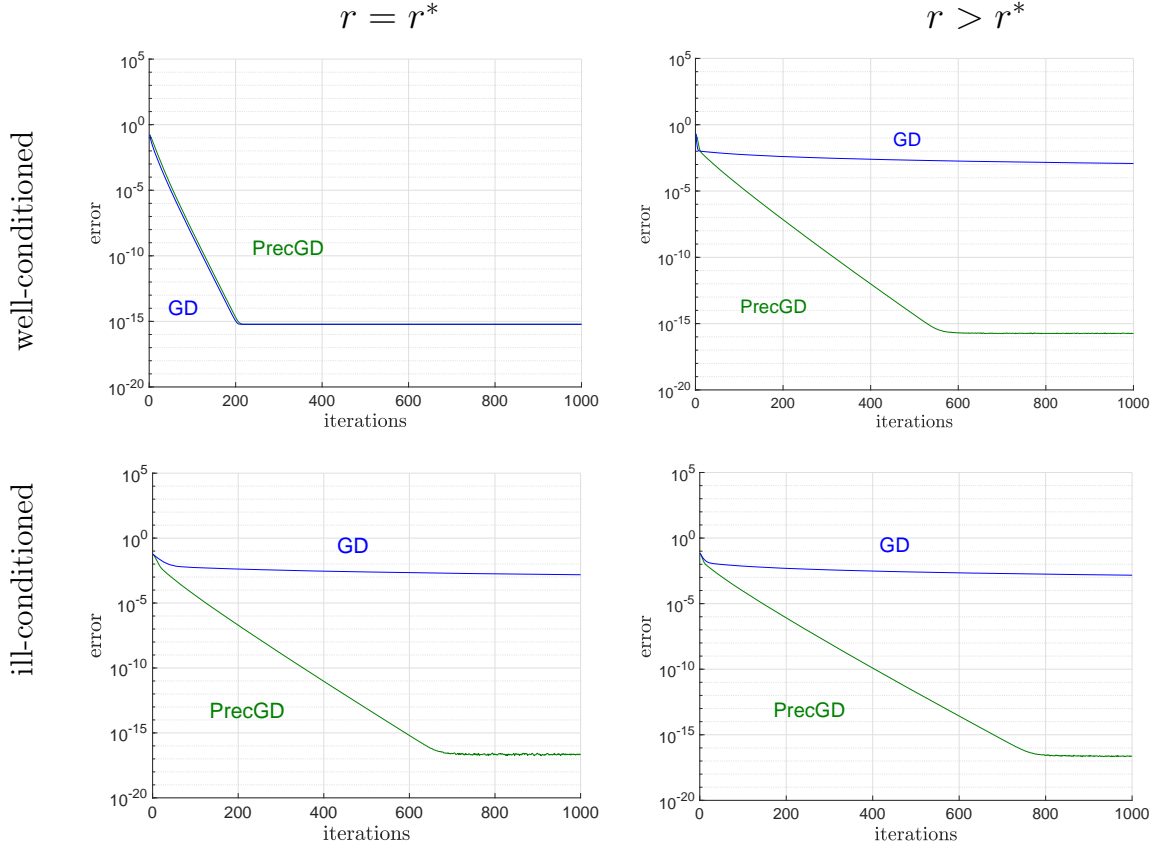
\begin{figure}[t]
\centering
\begin{tikzpicture}
\newcommand{\coord}{4.5}
\newcommand{\wid}{0.45}
\newcommand{\rl}{0.25}
\node (image) at (\rl,\coord) {
\includegraphics[width = \wid\textwidth]{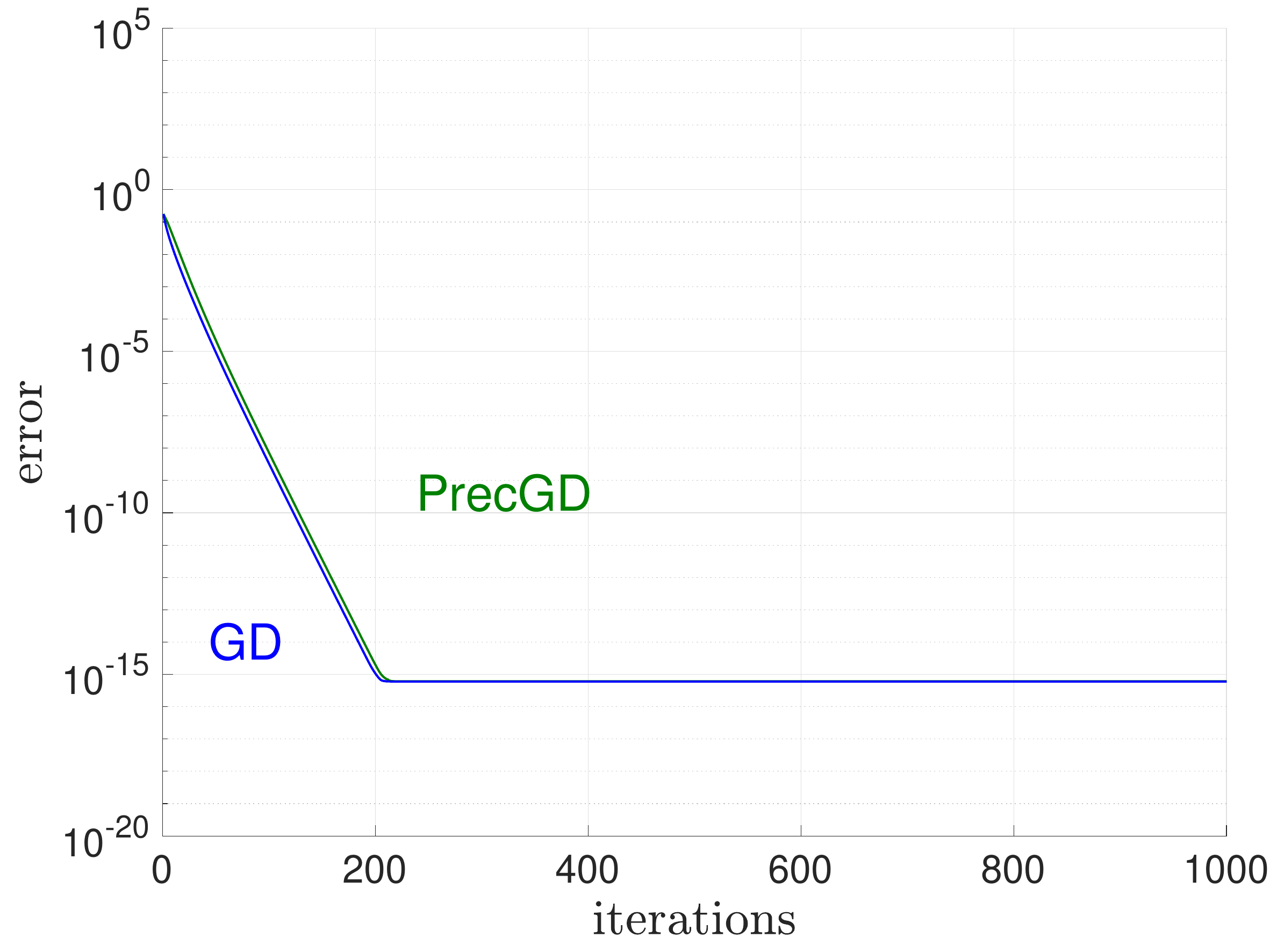}
};
\node (image) at (2\rl+\wid\textwidth,\coord) {
\includegraphics[width = \wid\textwidth]{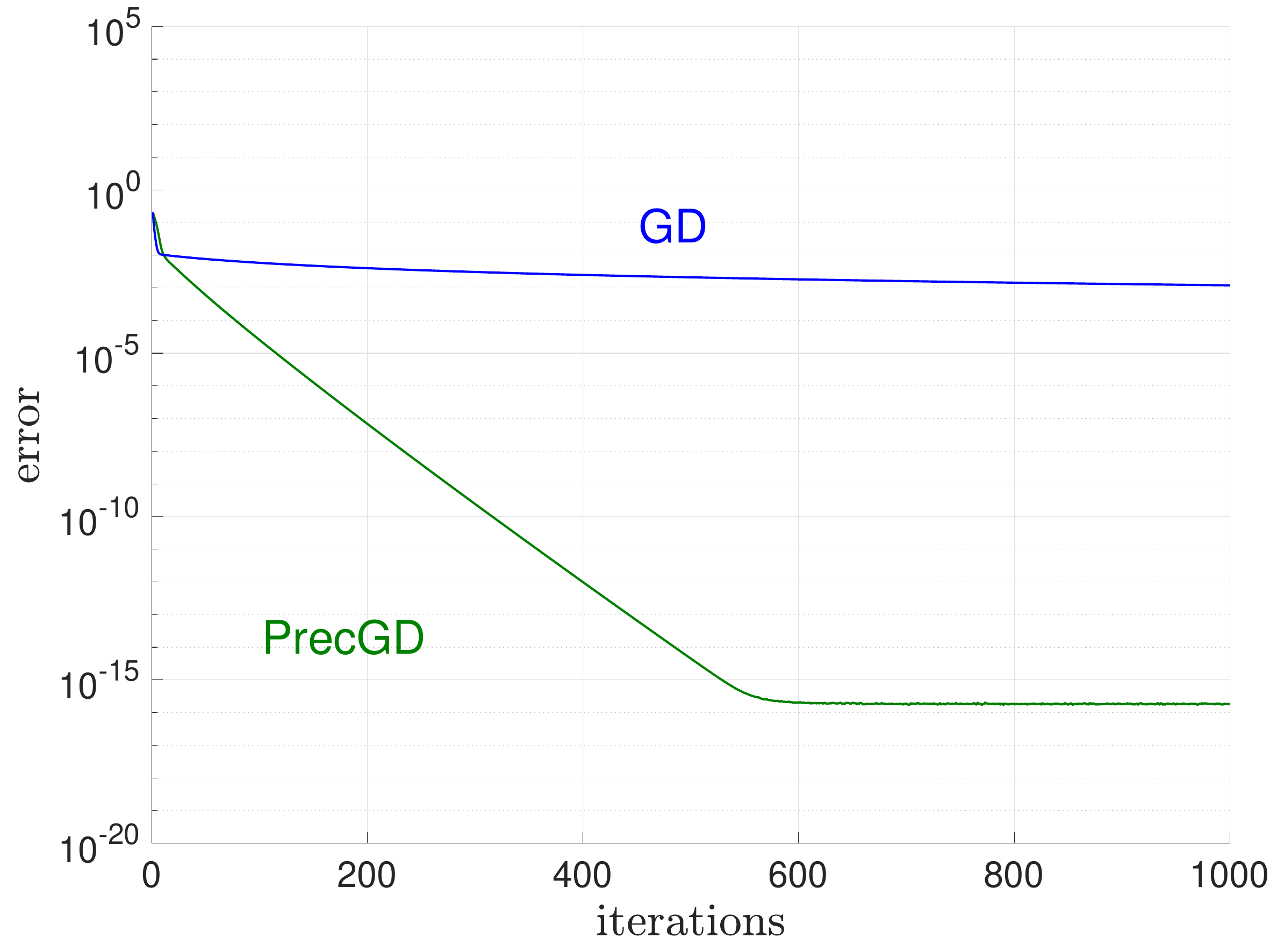}
};
\node (image) at (\rl,-1) {
\includegraphics[width = \wid\textwidth]{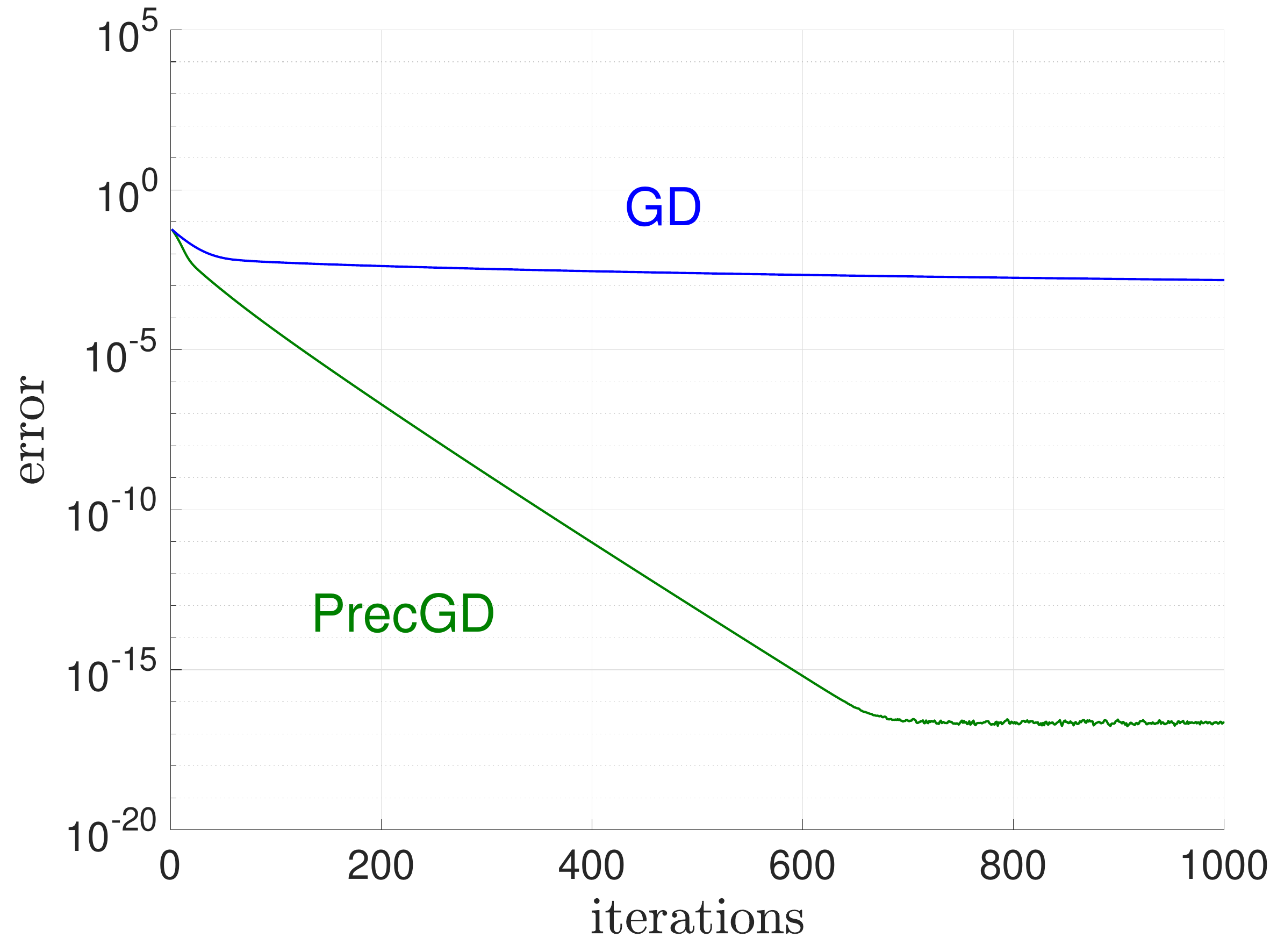}
};
\node (image) at (2\rl+\wid\textwidth,-1) {
\includegraphics[width = \wid\textwidth]{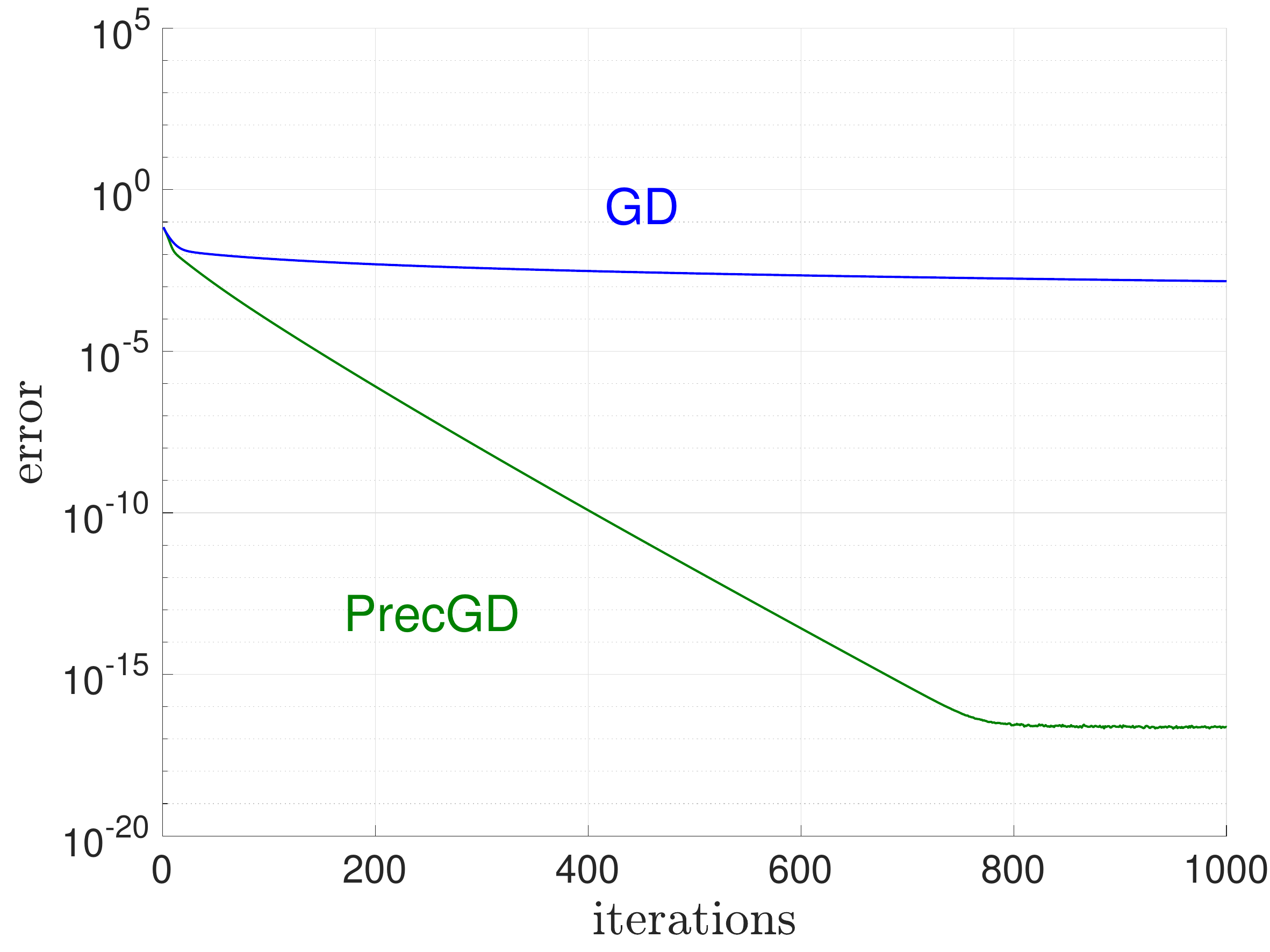}
};

\node[] at (0.9,7.5) {\Large $r=r^*$
};
\node[] at (8.7,7.5) {\Large $r>r^*$
};
\node[rotate=90] at (-4,0) {\large ill-conditioned
};
\node[rotate=90] at (-4,4.5) {\large well-conditioned
};
\end{tikzpicture} \caption{\footnotesize
Phase retrieval. \textbf{First row}: Well-conditioned ($\kappa=1$), rank-2
ground truth of size $100\times100$. The left panel shows the performance of GD and PrecGD for $r=r^* = 2$. Both
algorithms converge linearly to machine error.
The right panel shows the performance of GD and PrecGD for $r = 4$. The overparameterized GD converges sublinearly,
while PrecGD maintains the same converge rate. \textbf{Second row}: Ill-conditioned
($\kappa=5$), rank-2 ground truth of size $100\times100$. The left panel shows the performance of GD and PrecGD for $r=r^* = 2$. GD stagnates
due to ill-conditioning while PrecGD converges linearly. The right panel shows the performance of GD and PrecGD for $r = 4$. The overparameterized GD continues to stagnate, while PrecGD maintains the
same linear convergence rate.}
\label{fig:phase}
\end{figure}
For our final set of experiments we consider the problem of recovering
a low matrix $M^{\star}\succeq0$ from
\textit{quadratic} measurements of the form $y_{i}=a_{i}^{T}M^{\star}a_{i}$
where $a_{i}\in\mathbb{R}^{n}$ are the measurement vectors. In general,
the measurement vectors $a_{i}$ can be complex, but for illustration
purposes we focus on the case where the measurements are real. Suppose
that we have a total of $m$ measurements, then our objective is 
\begin{equation}
\min_{X\in\mathbb{R}^{n\times r}}f(X)=\sum_{i=1}^{m}(\|a_{i}^{T}X\|_{F}^{2}-y_{i})^{2}.
\end{equation}
In the special case where $M^{\star}$ is rank-1, this problem is
known as phase retrieval, which arises in a wide range of problems
including crystallography \citep{harrison1993phase,millane1990phase},
diffraction and array imaging \citep{bunk2007diffractive}, quantum
mechanics \citep{corbett2006pauli} and so on.

To gauge the effects of ill-conditioning in $M^{\star}$, we
focus on the case where $M^{\star}$ is rank-2 instead. As before,
we consider two choices of $M^{\star}$, one well-conditioned and
one ill-conditioned, generated exactly the same way as the previous
two problems. The measurement vectors $a_{i}$ are chosen to be random
vectors with standard Gaussian entries.

We perform two set of experiments: (1) the exactly-parameterized case
with $r=2$ and (2) the overparameterized case where $r=4$. In the
case $r=2$, the step-size is set to $4\times10^{-4}$ and in the
case $r=4$, the step-size is set to $10^{-4}$. As before, both methods
are initialized near the ground truth: we compute $M^{\star}=ZZ^{T}$
with $Z\in\mathbb{R}^{n\times r}$ and set the initial point $X_{0}=Z+10^{-2}w$,
where $w$ is a $n\times r$ random matrix with standard Gaussian
entries.

Our experiments are shown in Figure \ref{fig:1bit}. Even though the
objective for phase retrieval no longer satisfies restricted strong
convexity, we still observe the same results as before. Both ill-conditioning
and overparameterization causes gradient descent to slow down significantly,
while PrecGD maintains a linear convergence rate independent of both.

\subsection{Certification of optimality}
A key advantage of overparameterization is that it allows us to \textit{certify} the optimality of a point $X$ computed using local search methods. As we proved in Proposition \ref{prop:euclid_cert}, the suboptimality of a point $X$ can be bounded as 
\begin{equation}
\label{eq:subo}
f(X)-f(X^{\star})\le C_{g}\cdot\epsilon_{g}+C_{H}\cdot\epsilon_{H}+C_{\lambda}\cdot\epsilon_{\lambda}.
\end{equation}
Here we recall that $
\inner{\nabla f(X)}V\le\epsilon_{g}\cdot\|V\|_{F}, \inner{\nabla^{2}f(X)[V]}V\ge-\epsilon_{H}\cdot\|V\|_{F}^{2}
$ for all $V$, and $\lambda_{\min}(X^TX)\leq \epsilon_{\lambda}.$
To evaluate the effectiveness of this optimality certificate, we consider three problems as before: matrix sensing with $\ell_2$ loss, $1$-bit matrix sensing, and phase retrieval. The experimental setup is the same as before.
For each problem, we plot the function value $f(X)-f(X^\star)$ as the number of iterations increases, where $X^\star$ is the global minimizer of $f(\cdot)$. Additionally, we also compute the suboptimality as given by (\ref{eq:subo}). The constants in (\ref{eq:subo}) can be computed efficiently in linear time. For $\epsilon_H$ in particular, we apply the shifted power-iteration as described in Section \ref{sec:Cert}.

The results are shown in Figures \ref{fig:cert1} and \ref{fig:cert2}, for matrix sensing, phase retrieval, and 1-bit matrix sensing, respectively. We see that in each case, the upper bound in (\ref{eq:subo}) indeed bounds the suboptimality $f(X)-f(X^\star).$ Moreover, this upper bound also converges linearly, albeit at a different rate. This slower rate is due to the fact that $\epsilon_g$, the norm of the gradient, typically scales as $\sqrt{\epsilon_H}$ (\citealt{jin2017escape, nesterov2006cubic}), hence it converges to $0$ slower (by a square root). As a result, we see in all three plots that the upper bound converges slower roughly by a factor of a square root. In practice, this mean that if we want to certify $n$ digits of accuracy within optimality, we would need our iterate to be accurate up to roughly $2n$ digits.

\begin{figure}
    \centering
    \includegraphics[width=0.48\textwidth]{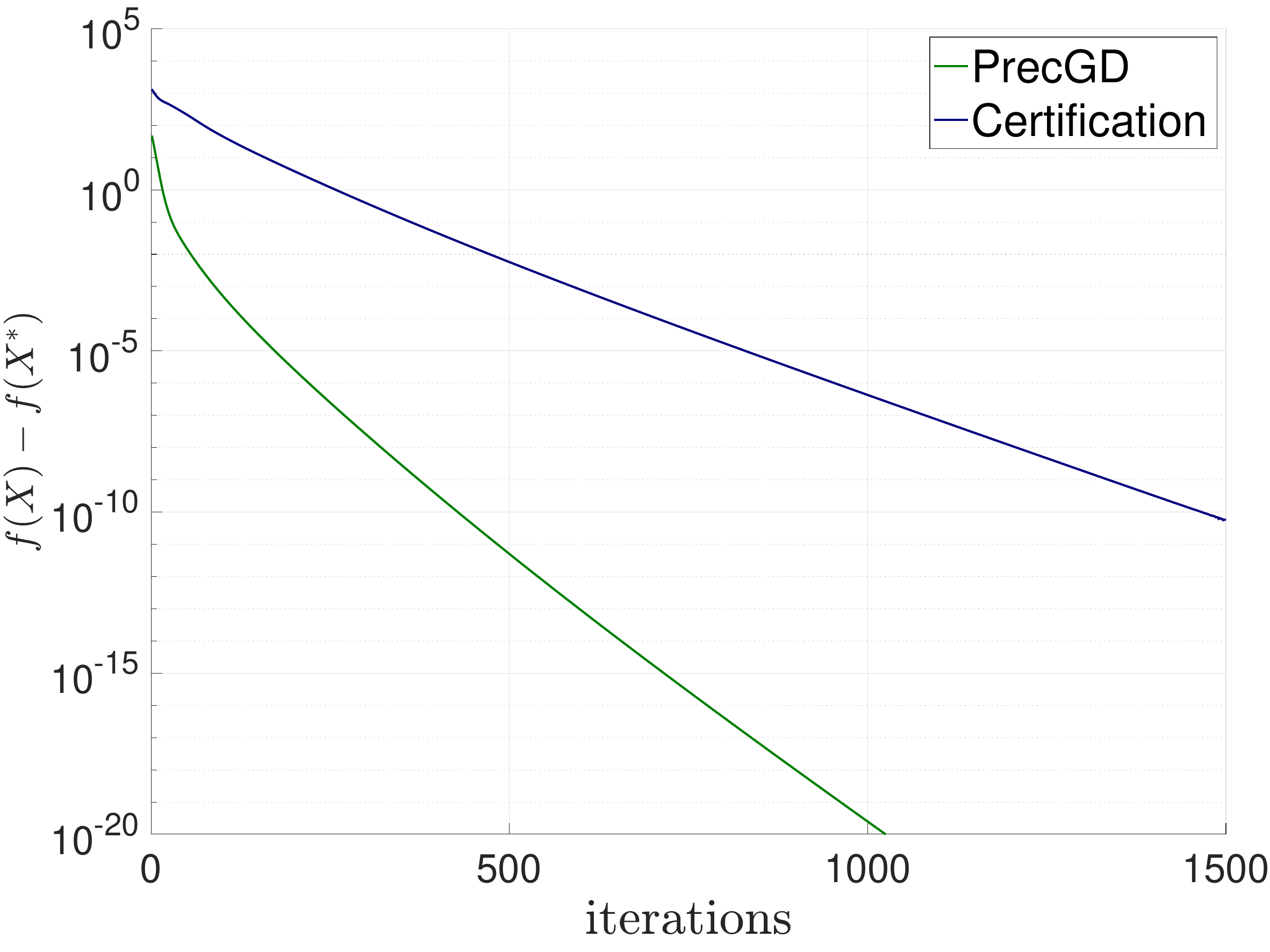}
    \includegraphics[width=0.48\textwidth]{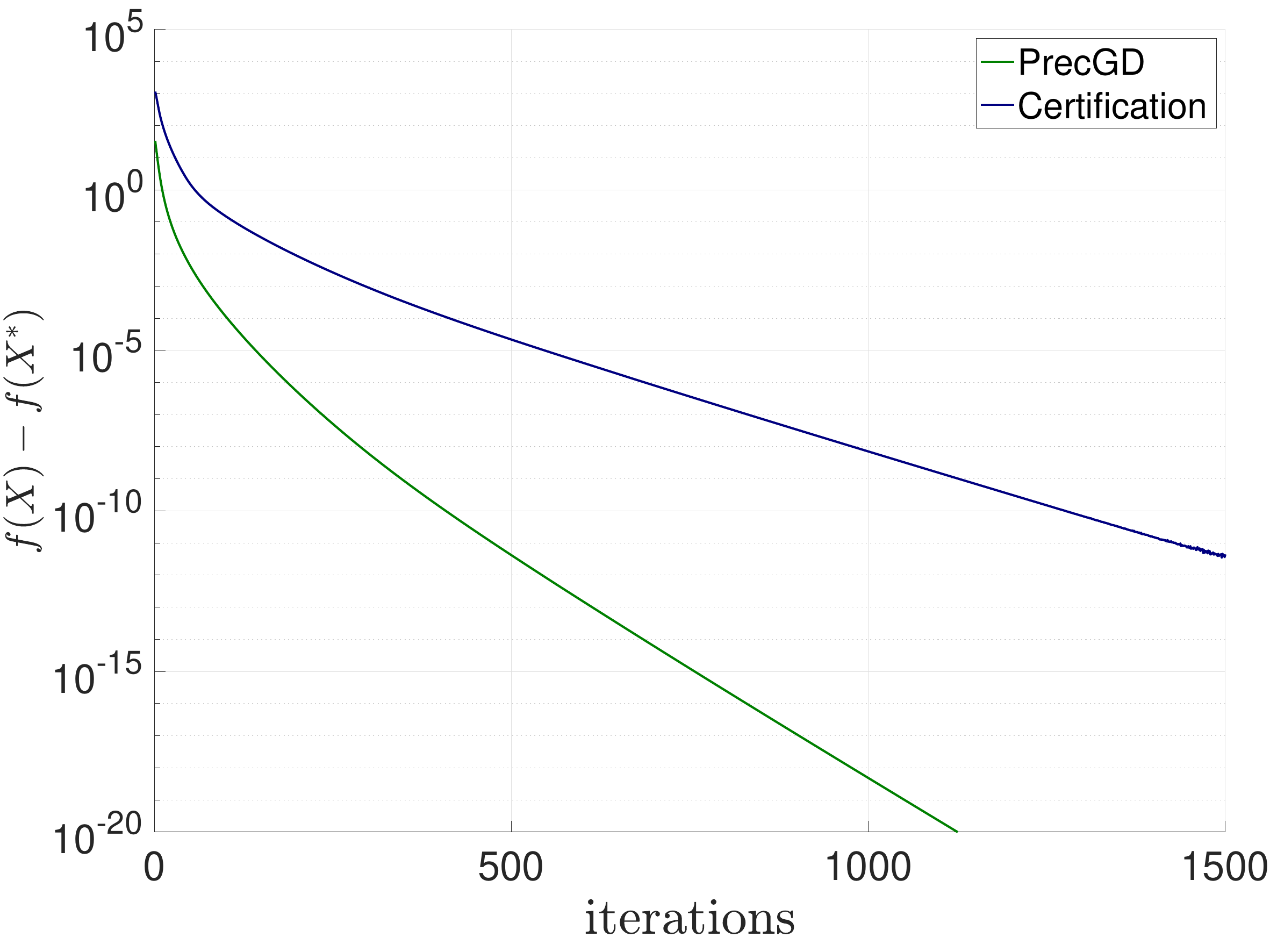}
    \caption{\footnotesize
    Certificate of global optimality. \textbf{Left}: Matrix sensing with a well-conditioned ($\kappa=1$), rank-2 ground truth of size $100\times 100$. The search rank is set to $r=4$ and the algorithm is initialized within a neighborhood of radius $10^{-2}$ around the ground truth. The stepsize is set to $5\times 10^{-5}$. \textbf{Right}: Phase retrieval with a well-conditioned ($\kappa=1$), rank-2 ground truth of size $100\times 100$. The search rank is set to $r=4$ and the algorithm is initialized within a neighborhood of radius $10^{-2}$ around the  ground truth. The step-size is set to $3\times 10^{-5}$.}
    \label{fig:cert1}
\end{figure}

\begin{figure}
    \centering
    \includegraphics[width=0.48\textwidth]{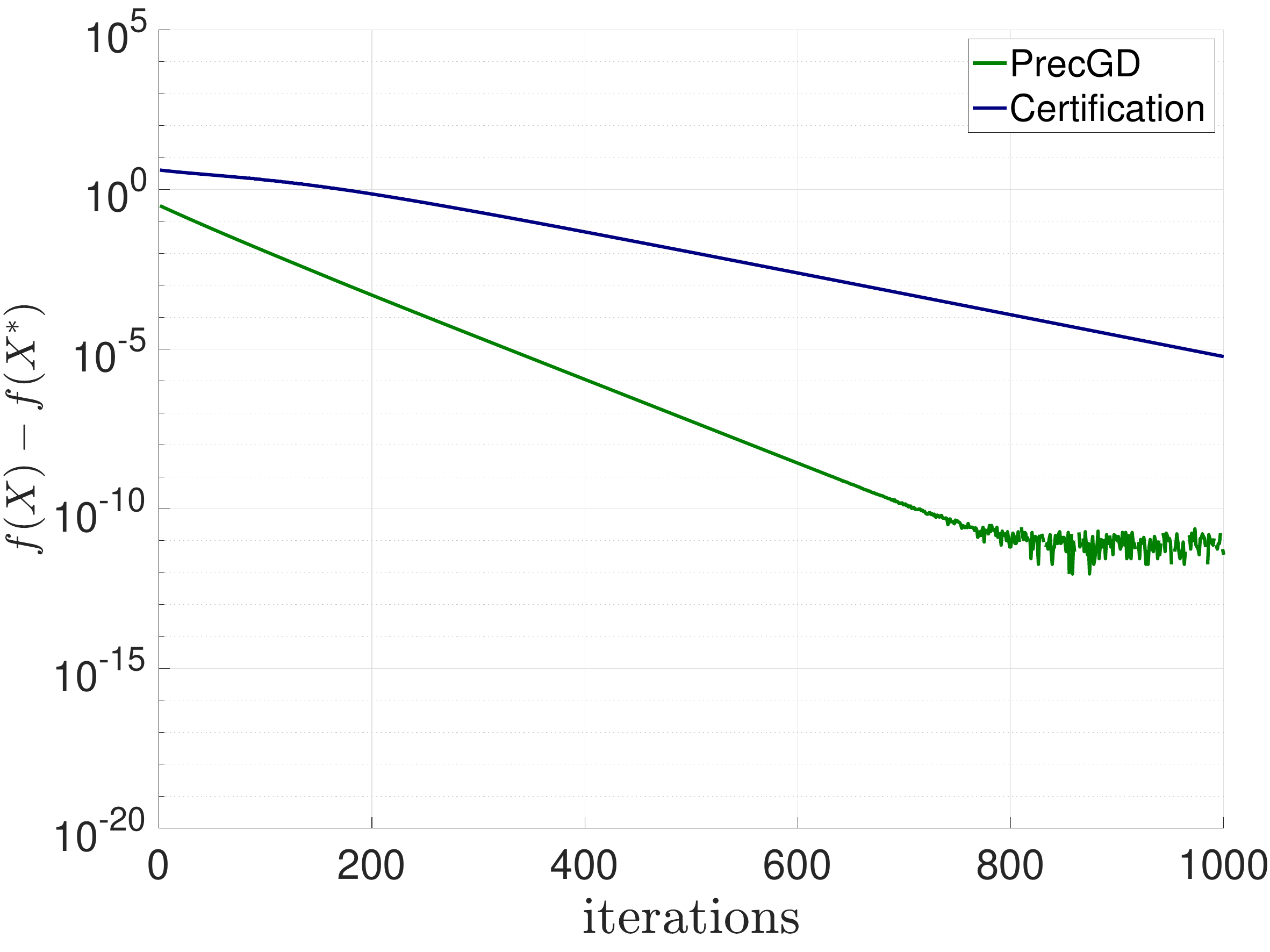}
    \caption{\footnotesize
    Certificate of global optimality for 1-bit matrix sensing with a well-conditioned ($\kappa=1$), rank-2 ground truth of size $100\times 100$. The search rank is set to $r=4$, and the algorithm is initialized within a neighborhood of radius $10^{-2}$ around the ground truth. The step-size is set to $3\times 10^{-2}$.}
    \label{fig:cert2}
\end{figure}

\section{Conclusions}
In this work, we consider the problem of minimizing a smooth convex function $\phi$ over a positive semidefinite matrix $M$. The Burer-Monteiro approach eliminates the large $n\times n$ positive semidefinite matrix by reformulating the problem as minimizing the nonconvex function
$f(X)=\phi(XX^{T})$ over an $n\times r$ factor matrix $X$. Here, we overparameterize the search rank $r>r^\star$ with respect to the true rank $r^\star$ of the solution $X^\star$, because the rank deficiency of a second-order stationary point $X$ allows us to \textit{certify} that $X$ is globally optimal. 

Unfortunately, gradient descent becomes extremely slow once the problem is overparameterized.
 Instead, we propose a method known as PrecGD, which enjoys a similar per-iteration cost as gradient descent, but speeds up the convergence rate of gradient descent exponentially in
overparameterized case. In particular, we prove that within a neighborhood around the ground truth, PrecGD converges linearly towards the ground truth, at a rate independent of both ill-conditioning in the ground truth and overparameterization. We also prove that, similar to gradient descent, a perturbed version of PrecGD converges globally, from any initial point. 

Our numerical experiments find that preconditioned gradient descent works equally well in restoring the linear convergence of gradient descent in the overparmeterized regime for choices of $\phi$ that do not satisfy restricted strong convexity. We leave the justification of these results for future work. 

\acks{Financial support for this work was provided by NSF CAREER Award ECCS-2047462,
NSF Award DMS-2152776, ONR Award N00014-22-1-2127.}

\appendix

\section{\label{app:pf_align}Proof of Basis Alignment (\lemref{align})}

For $X\in\R^{n\times r}$ and $Z\in\R^{n\times r^{\star}}$, suppose
that $X$ satisfies
\begin{equation}
\rho\eqdef\frac{\|XX^{T}-ZZ^{T}\|_{F}}{\lambda_{\min}(Z^{T}Z)}<\frac{1}{\sqrt{2}}.\label{eq:rhodef}
\end{equation}
In this section, we prove that the incidence angle $\theta$ defined
as
\[
\cos\theta=\max_{Y\in\R^{n\times r}}\frac{\inner{XX^{T}-ZZ^{T}}{XY^{T}+YX^{T}}}{\|XX^{T}-ZZ^{T}\|_{F}\|XY^{T}+YX^{T}\|_{F}},
\]
satisfies 
\begin{equation}
\sin\theta=\frac{\|(I-XX^{\dagger})ZZ^{T}(I-XX^{\dagger})\|_{F}}{\|XX^{T}-ZZ^{T}\|_{F}}\le\frac{1}{\sqrt{2}}\frac{\rho}{\sqrt{1-\rho^{2}}}\label{eq:bnd}
\end{equation}
where $\dagger$ denotes the pseudoinverse. 

First, note that an $X$ that satisfies (\ref{eq:rhodef}) must have
$\rank(X)\ge r^{\star}$. This follows from Weyl's inequality
\begin{align*}
\lambda_{r^{\star}}(XX^{T}) & \ge\lambda_{r^{\star}}(ZZ^{T})-\|XX^{T}-ZZ^{T}\|_{F}\ge\left(1-\frac{1}{\sqrt{2}}\right)\cdot\lambda_{r^{\star}}(ZZ^{T}).
\end{align*}
Next, due to the rotational invariance of this problem, we can assume
without loss of generality\footnote{We compute the singular value decomposition $X=USV^{T}$ with $U\in\R^{n\times n}$
and $V\in\R^{r\times r}$, and then set $X\gets U^{T}XV$ and $Z\gets U^{T}Z$.} that $X,Z$ are of the form 
\begin{equation}
X=\begin{bmatrix}X_{1} & 0\\
0 & X_{2}
\end{bmatrix},\quad Z=\begin{bmatrix}Z_{1}\\
Z_{2}
\end{bmatrix},\quad\sigma_{\min}(X_{1})\ge\sigma_{\max}(X_{2})\label{invari}
\end{equation}
where $X_{1}\in\mathbb{R}^{r\times k},\;Z_{1}\in\R^{k\times r^{\star}}$.
For $k\ge r^{\star}$, the fact that $\rank(X)\ge r^{\star}$ additionally
implies that $\sigma_{\min}(X_{1})>0.$ 

The equality in (\ref{eq:bnd}) immediately follows by setting $k=\rank(X)$
and solving the projection problem
\begin{align*}
\|E\|_{F}\sin\theta & =\min_{Y}\|(XY^{T}+YX^{T})-(XX-ZZ^{T})\|_{F}\\
 & =\min_{Y_{1},Y_{2}}\left\Vert \begin{bmatrix}X_{1}Y_{1}^{T}+Y_{1}X_{1}^{T} & X_{1}Y_{2}^{T}\\
Y_{2}X_{1}^{T} & 0
\end{bmatrix}-\begin{bmatrix}X_{1}X_{1}^{T}-Z_{1}Z_{1}^{T} & Z_{2}Z_{1}^{T}\\
Z_{1}Z_{2}^{T} & -Z_{2}Z_{2}^{T}
\end{bmatrix}\right\Vert _{F}\\
 & =\|Z_{2}Z_{2}^{T}\|_{F}=\|(I-XX^{\dagger})ZZ^{T}(I-XX^{\dagger})\|_{F}.
\end{align*}
Before we prove the inequality in (\ref{eq:bnd}), we state and prove
a technical lemma that will be used in the proof.
\begin{lemma}
\label{lem:ab}Suppose that $X,Z$ are in the form in (\ref{invari}),
and $k\geq r^{\star}$. If $\rho$ defined in (\ref{eq:rhodef}) satisfies
$\rho<1/\sqrt{2}$, we have $\lambda_{\min}(Z_{1}^{T}Z_{1})\ge\lambda_{\max}(Z_{2}^{T}Z_{2}).$ 
\end{lemma}
\begin{proof}
Denote $\gamma_{1}=\lambda_{\min}(Z_{1}^{T}Z_{1})$ and $\gamma_{2}=\lambda_{\max}(Z_{2}^{T}Z_{2})$.
By contradiction, we will prove that $\gamma_{1}<\gamma_{2}$ implies
$\rho\geq1/\sqrt{2}$. This claim is invariant to scaling of $X$
and $Z$, so we assume without loss of generality that $\lambda_{\min}(Z^{T}Z)=1$.
Under (\ref{invari}), our radius hypothesis $\rho\ge\|XX^{T}-ZZ^{T}\|_{F}$
reads
\begin{align*}
\rho^{2} & \ge\|X_{1}X_{1}^{T}-Z_{1}Z_{1}^{T}\|_{F}^{2}+2\|Z_{1}Z_{2}^{T}\|_{F}^{2}+\|X_{2}X_{2}^{T}-Z_{2}Z_{2}^{T}\|_{F}^{2}.\\
 & \ge\|X_{1}X_{1}^{T}-Z_{1}Z_{1}^{T}\|_{F}^{2}+\|X_{2}X_{2}^{T}-Z_{2}Z_{2}^{T}\|_{F}^{2}+2\lambda_{\min}(Z_{1}^{T}Z_{1})\lambda_{\max}(Z_{2}^{T}Z_{2}).
\end{align*}
Below, we will prove that $X_{1},X_{2}$ that satisfy $\sigma_{\min}(X_{1})\ge\sigma_{\max}(X_{2})$
also satisfies
\begin{align}
 & \|X_{1}X_{1}^{T}-Z_{1}Z_{1}^{T}\|_{F}^{2}+\|X_{2}X_{2}^{T}-Z_{2}Z_{2}^{T}\|_{F}^{2}\nonumber \\
\ge & \min_{d_{1},d_{2}\in\R_{+}}\{[d_{1}-\gamma_{1}]^{2}+[d_{2}-\gamma_{2}]^{2}:d_{1}\ge d_{2}\}\label{eq:optim_ab}
\end{align}
If $\gamma_{1}<\gamma_{2},$ then $d_{1}=d_{2}$ holds at optimality,
so the minimum value is $\frac{1}{2}(\gamma_{1}-\gamma_{2})^{2}$.
Substituting $\gamma_{1}=\lambda_{\min}(Z_{1}^{T}Z_{1})$ and $\gamma_{2}=\lambda_{\max}(Z_{2}^{T}Z_{2})$
then proves that 
\[
\rho^{2}\geq\frac{(\gamma_{1}-\gamma_{2})^{2}}{2}+2\gamma_{1}\gamma_{2}=\frac{1}{2}(\gamma_{1}+\gamma_{2})^{2}.
\]
But we also have 
\begin{align*}
\gamma_{1}+\gamma_{2}=\lambda_{\min}(Z_{1}^{T}Z_{1})+\lambda_{\max}(Z_{2}^{T}Z_{2})\geq\lambda_{\min}(Z_{1}^{T}Z_{1}+Z_{2}^{T}Z_{2})=\lambda_{\min}(Z^{T}Z)=1
\end{align*}
and this implies $\rho^{2}\geq1/2$, a contradiction. 

We now prove (\ref{eq:optim_ab}). Consider the following optimization
problem
\[
\min_{X_{1},X_{2}}\left\{ \|X_{1}X_{1}^{T}-Z_{1}Z_{1}^{T}\|_{F}^{2}+\|X_{2}X_{2}^{T}-Z_{2}Z_{2}^{T}\|_{F}^{2}:\lambda_{\min}(X_{1}X_{1}^{T})\ge\lambda_{\max}(X_{2}X_{2}^{T})\right\} .
\]
We relax $X_{1}X_{1}^{T}$ into $S_{1}\succeq0$ and $X_{2}X_{2}^{T}$
into $S_{2}\succeq0$ to yield a lower-bound
\begin{align*}
\ge & \min_{S_{1}\succeq0,S_{2}\succeq0}\{\|S_{1}-Z_{1}Z_{1}^{T}\|_{F}^{2}+\|S_{2}-Z_{2}Z_{2}^{T}\|_{F}^{2}:\lambda_{\min}(S_{1})\ge\lambda_{\max}(S_{2})\}.
\end{align*}
The problem is invariant to a change of basis, so we change into the
eigenbases of $Z_{1}Z_{1}^{T}$ and $Z_{2}Z_{2}^{T}$ to yield 
\begin{align*}
= & \min_{s_{1}\ge0,s_{2}\ge0}\{\|s_{1}-\lambda(Z_{1}Z_{1}^{T})\|^{2}+\|s_{2}-\lambda(Z_{2}Z_{2}^{T})\|^{2}:\min(s_{1})\ge\max(s_{2})\}
\end{align*}
where $\lambda(Z_{1}Z_{1}^{T})\ge0$ and $\lambda(Z_{2}Z_{2}^{T})\ge0$
denote the vector of eigenvalues. We lower-bound this problem by dropping
all the terms in the sum of squares except the one associated with
$\gamma_{1}=\lambda_{\min}(Z_{1}^{T}Z_{1})$ and $\gamma_{2}=\lambda_{\max}(Z_{2}Z_{2}^{T})$
to obtain 
\begin{align}
 & \geq\min_{d_{1},d_{2}\in\R_{+}}\{[d_{1}-\gamma_{1}]^{2}+[d_{2}-\gamma_{2}]^{2}:d_{1}\ge d_{2}\}
\end{align}
which is exactly (\ref{eq:optim_ab}) as desired. 
\end{proof}

Now we are ready to prove the inequality in (\ref{eq:bnd}). Instead, we will prove a stronger result for $k=r^\star$. 
\begin{lemma}\label{lem:basis_k}
 Let $\tilde X$ denote the rank-$r^\star$ truncation of $X$, i.e. its projection onto the set of rank-$r^\star$ matrices.
 If $\rho = \|XX^{T}-ZZ^{T}\|_{F} / \lambda_{\min}(Z^TZ)<1$,
then, 
    \begin{equation}
\frac{\|(I-XX^{\dagger})ZZ^{T}(I-XX^{\dagger})\|_{F}}{\|XX^{T}-ZZ^{T}\|_{F}} \le \frac{\|(I-\tilde{X}\tilde{X}^{\dagger})ZZ^{T}(I-\tilde{X}\tilde{X}^{\dagger})\|_{F}}{\|XX^{T}-ZZ^{T}\|_{F}}\le\frac{1}{\sqrt{2}}\frac{\rho}{\sqrt{1-\rho^{2}}}.\nonumber
\end{equation}
\end{lemma}
\begin{proof}
Within the radius $\rho<1/\sqrt{2}$,
we begin by proving that the following incidence angle between $\tilde X$
and $Z$ satisfies 
\begin{equation}
\sin\phi\eqdef\frac{\|(I-\tilde{X}\tilde{X}^{\dagger})Z\|_{F}}{\sigma_{r^{\star}}(Z)} = \frac{\|Z_{2}\|_{F}}{\sqrt{\lambda_{\min}(Z^{T}Z)}}\le\frac{\|XX^{T}-ZZ^{T}\|_{F}}{\lambda_{\min}(Z^{T}Z)}=\rho.\label{eq:sinphi}
\end{equation}
First, note that $\sigma_{\min}(X_{1})>0$
and therefore
$(I-\tilde{X}\tilde{X}^{\dagger})Z  = \begin{bmatrix}0\\
Z_{2}
\end{bmatrix}.$ 
Next, 
\begin{align*}
\|XX^{T}-ZZ^{T}\|_{F}^{2} & =\|X_{1}X_{1}^{T}-Z_{1}Z_{1}^{T}\|_{F}^{2}+2\langle Z_{1}^{T}Z_{1},Z_{2}^{T}Z_{2}\rangle+\|X_{2}X_{2}^{T}-Z_{2}Z_{2}^{T}\|_{F}^{2}\\
 & \ge2\langle Z_{1}^{T}Z_{1},Z_{2}^{T}Z_{2}\rangle\geq2\lambda_{\min}(Z_{1}^{T}Z_{1})\|Z_{2}\|_{F}^{2}\ge\lambda_{\min}(Z^{T}Z)\|Z_{2}\|_{F}^{2}
\end{align*}
where we use $\lambda_{\min}(Z_{1}^{T}Z_{1})\ge\frac{1}{2}\lambda_{\min}(Z^{T}Z)$
because 
\begin{gather}
\lambda_{\min}(Z^{T}Z)=\lambda_{\min}(Z_{1}^{T}Z_{1}+Z_{2}^{T}Z_{2})\le\lambda_{\min}(Z_{1}^{T}Z_{1})+\lambda_{\max}(Z_{2}^{T}Z_{2})\le2\lambda_{\min}(Z_{1}^{T}Z_{1})\label{eq:zmin}
\end{gather}
where we used $\lambda_{\min}(Z_{1}^{T}Z_{1})\ge\lambda_{\max}(Z_{2}^{T}Z_{2})$
via \lemref{ab}. 
 Then, (\ref{eq:bnd}) is true because 
\begin{align}
\frac{\|Z_{2}Z_{2}^{T}\|_{F}^{2}}{\|XX^{T}-ZZ^{T}\|_{F}^{2}} & \overset{\text{(a)}}{\le}\frac{\|Z_{2}\|_{F}^{4}}{2\langle Z_{1}^{T}Z_{1},Z_{2}^{T}Z_{2}\rangle}\overset{\text{(b)}}{\le}\frac{\|Z_{2}\|_{F}^{4}}{2\lambda_{\min}(Z_{1}^{T}Z_{1})\|Z_{2}\|_{F}^{2}}\nonumber \\
 & \overset{\text{(c)}}{\le}\frac{\|Z_{2}\|_{F}^{2}}{2[\lambda_{\min}(Z^{T}Z)-\|Z_{2}\|_{F}^{2}]}\overset{\text{(d)}}{\le}\frac{\sin^{2}\phi}{2[1-\sin^{2}\phi]}\leq\frac{1}{2}\frac{\rho^{2}}{1-\rho^{2}}
\end{align}
Step (a) bounds $\|Z_{2}Z_{2}^{T}\|_{F}\leq\|Z_{2}\|_{F}^{2}$ and
$2\langle Z_{1}^{T}Z_{1},Z_{2}^{T}Z_{2}\rangle\le\|XX^{T}-ZZ^{T}\|_{F}^{2}$.
Step (b) bounds $\langle Z_{1}^{T}Z_{1},Z_{2}^{T}Z_{2}\rangle\geq\lambda_{\min}(Z_{1}^{T}Z_{1})\cdot\tr(Z_{2}Z_{2}^{T})$.
Step (c) uses (\ref{eq:zmin}) and $\|Z_{2}\|_{F}^{2}\ge\lambda_{\max}(Z_{2}^{T}Z_{2})$.
Finally, step (d) substitutes (\ref{eq:sinphi}).
\end{proof}

\section{\label{app:pf_decr}Proof of Gradient Lipschitz (\lemref{decr})}
\begin{proof}
Let $\phi$ be $L$-gradient Lipschitz. Let $M^{\star}=\arg\min\phi$
satisfy $M^{\star}\succeq0$. In this section, we prove that $f(X)\eqdef\phi(XX^{T})$
satisfies
\begin{gather*}
f(X+V)\le f(X)+\inner{\nabla f(X)}V+\frac{L}{2}\gamma_{X,\eta}\|V\|_{X,\eta}^{2}\\
\text{where }\gamma_{X,\eta}=4+\frac{2\|E\|_{F}+4\|V\|_{X,\eta}}{\lambda_{\min}+\eta}+\frac{\|V\|_{X,\eta}^{2}}{(\lambda_{\min}+\eta)^{2}}
\end{gather*}
where $\|V\|_{X,\eta}=\|V(X^{T}X+\eta I)^{-1/2}\|_{F}$ and $\lambda_{\min}\equiv\lambda_{\min}(X^{T}X)$
and $E=XX^{T}-M^{\star}$. First, it follows from the $L$-gradient
Lipschitz property of $\phi$ that
\begin{align*}
\underbrace{\phi((X+V)(X+V)^{T})}_{f(X+V)} & =\underbrace{\phi(XX^{T})}_{f(X)}+\underbrace{\inner{\nabla\phi(XX^{T})}{XV^{T}+VX^{T}}}_{\inner{\nabla f(X)}V}\\
 & +\inner{\nabla\phi(XX^{T})}{VV^{T}}+\frac{L}{2}\|XV^{T}+VX^{T}+VV^{T}\|_{F}^{2}.
\end{align*}
 Substituting the following
\begin{gather*}
\|VX^{T}\|_{F}\le\|(X^{T}X+\eta I)^{-1/2}X^{T}\|\cdot\|V\|_{X,\eta}\le\|V\|_{X,\eta}\\
\|VV^{T}\|_{F}\le\|(X^{T}X+\eta I)^{-1}\|\cdot\|V\|_{X,\eta}^{2}=[\lambda_{\min}(X^{T}X)+\eta]^{-1}\|V\|_{X,\eta}^{2}\\
\|\nabla\phi(XX^{T})\|_{F}=\|\nabla\phi(XX^{T})-\nabla\phi(M^{\star})\|_{F}\le L\|E\|_{F},
\end{gather*}
bounds the error term
\begin{align*}
 & \inner{\nabla\phi(XX^{T})}{VV^{T}}+\frac{L}{2}\|XV^{T}+VX^{T}+VV^{T}\|_{F}^{2}\\
= & \frac{\|\nabla\phi(XX^{T})\|_{F}\|V\|_{X,\eta}^{2}}{\lambda_{\min}+\eta}+\frac{L}{2}\left(4\|V\|_{X,\eta}^{2}+\frac{4\|V\|_{X,\eta}^{3}}{\lambda_{\min}+\eta}+\frac{\|V\|_{X,\eta}^{4}}{(\lambda_{\min}+\eta)^{2}}\right)\\
\le & \frac{L\cdot\|V\|_{X,\eta}^{2}}{2}\left(4+\frac{2\|E\|_{F}+4\|V\|_{X,\eta}}{\lambda_{\min}+\eta}+\frac{\|V\|_{X,\eta}^{2}}{(\lambda_{\min}+\eta)^{2}}\right).
\end{align*}
This completes the proof. 
\end{proof}

\section{\label{app:pf_bndgrad}Proof of Bounded Gradient (\lemref{bndgrad})}
\begin{proof}
Let $\phi$ be $L$-gradient Lipschitz, and let $f(X)\eqdef\phi(XX^{T})$.
Let $M^{\star}=\arg\min\phi$ satisfy $M^{\star}\succeq0$. In this
section, we prove that $V=\nabla f(X)(X^{T}X+\eta I)^{-1}$ satisfies
\[
\|V\|_{X,\eta}=\|\nabla f(X)\|_{X,\eta}^{*}\le2L\|XX^{T}-M^{\star}\|_{F},
\]
where $\|V\|_{X,\eta}=\|VP_{X,\eta}^{1/2}\|_{F}$ and $\|\nabla f(X)\|_{X,\eta}^{*}=\|\nabla f(X)P_{X,\eta}^{-1/2}\|_{F}$
and $P_{X,\eta}=X^{T}X+\eta I$. Indeed, $\|V\|_{X,\eta}=\|\nabla f(X)\|_{X,\eta}^{*}$
can be verified by inspection. We have
\begin{align*}
\|\nabla f(X)\|_{X,\eta}^{*} & =\max_{\|Y\|_{X,\eta}=1}\inner{\nabla\phi(XX^{T})}{XY^{T}+YX^{T}}\\
 & \le\|\nabla\phi(XX^{T})\|_{F}\|XY^{\star T}+Y^{\star}X^{T}\|_{F}\\
 & \le\|\nabla\phi(XX^{T})-\nabla\phi(M^{\star})\|_{F}\left(2\|X(X^{T}X+\eta I)^{-1/2}\|\cdot\|Y^{\star}\|_{X,\eta}\right)\\
 & \le L\|XX^{T}-M^{\star}\|_{F}\cdot2.
\end{align*}
This completes the proof. 
\end{proof}

\section{Proofs of Global Convergence}

In this section, we provide the proofs of Lemmas~\ref{lemma:escape} and~\ref{lem_Lip} which play critical roles in proving the global convergence of PMGD (Theorem~\ref{thm:pmgd}) and PPrecGD (Corollary~\ref{cor_pprecgd}).

\subsection{Proof of Lemma~\ref{lemma:escape}}
To proceed with the proof of Lemma~\ref{lemma:escape}, we define the following quantities to streamline our presentation:
\begin{gather}
\alpha=\frac{{p_\lb}}{2\ell_1},\quad \beta=\frac{\epsilon}{400{L_d}\cdot\iota^{3}},\quad \mathcal{T}=\frac{\ell_1}{p_\lb \sqrt{L_d\epsilon}}\cdot\iota,\quad\mathcal{F}=\frac{p_\lb}{50\iota^3}\sqrt{\frac{\epsilon^{3}}{L_d}},\quad
\mathcal{S} := \frac{1}{5\iota}\sqrt{\frac{\epsilon}{L_d}}\label{eq:para_gd} \\
L_d = \frac{5\max\{\ell_2,L_P \ell_1\sqrt{p_\ub}\}}{p_\lb^{2.5}}, \quad
\iota=c\cdot\log\left(\frac{{p_{\ub}}d\ell_1(f(x_0)-f^*)}{p_\lb\ell_2\epsilon\delta}\right)
\end{gather}
for some absolute constant $c$. Once Lemma~\ref{lem:coupleseq_gd} below is established, the proof of Lemma~\ref{lemma:escape} follows by identically repeating the arguments in the proof of~\cite[Lemma 5.3]{jin2021nonconvex}. 

\begin{sloppypar}
\begin{lemma}
	[Coupling Sequence]\label{lem:coupleseq_gd} Suppose that $\bar{x}$ satisfies $\|\nabla f(\tilde{x})\|_{\tilde{x}}^* \leq \epsilon$ and $\nabla^2 f(\tilde{x})\not\succeq-\sqrt{L_d\epsilon} \cdot P(\bar{x})$.
	Let $\{x_{t}\}_{t=0}^{\mathcal{T}}$ and $\{y_{t}\}_{t=0}^{\mathcal{T}}$ be two sequences generated by PMGD initialized respectively at $x_0$ and $y_0$ which
	satisfy: (1) $\max\{\norm{x_{0}-\bar{x}},\norm{y_{0}-\bar{x}}\}\le\alpha \beta$;
	and (2) $ P(\bar{x})^{1/2}\left(x_0-y_0\right)=\alpha \omega \cdot v$,
	where $v$ is the eigenvector corresponding to the minimum eigenvalue of ${P(\bar{x})^{-1/2}}\nabla^2 f(\bar{x}){P(\bar{x})^{-1/2}}$
	and $\omega>\bar\omega := 2^{3-\iota/4}\cdot\mathcal{S}$.
	Then:
	
	\[
	\min\{f(x_{\mathcal{T}})-f(x_{0}),f(y_{\mathcal{T}})-f(y_{0})\}\le-2\mathcal{F}.
	\]
\end{lemma}
\end{sloppypar}
 We point out that the Lemma~\ref{lem:coupleseq_gd} is \textit{not} a direct consequence of~\cite[Lemma 5.5]{jin2021nonconvex} which shows a similar result but for the perturbed gradient descent. The key difference in our analysis is the precise control of the general metric function as a preconditioner for the gradients. In particular, we show that, while in general $P(x)$ and $P(y)$ can be drastically different for different values of $x$ and $y$, they can essentially be treated as constant matrices in the vicinity of strict saddle points. 
More precisely, according to (\ref{eq_dist}), the term $P^{1/2}_{\bar{x}}\left(x_{t+1}-y_{t+1}\right)$ can be written as
\begin{align}
P^{1/2}_{\bar{x}}\left(x_{t+1}-y_{t+1}\right)
&= \left(I - \alpha P({\bar{x}})^{-1/2}\nabla^2f(\bar{x})P(^{-1/2}\right)P({\bar{x}})^{1/2}(x_{t}-y_{t})+\xi(\bar{x}, x_t,y_t)\nonumber
\end{align}
where $\xi(\bar{x},x_t,y_t)$ is a deviation term defined as 
$$
\xi(\bar{x},  x_t,y_t) = \alpha P^{1/2}_{\bar{x}}\left(P^{-1}_{\bar{x}}\nabla^2f(\bar{x})(x_{t}-y_{t})-\left(P^{-1}_{x_t}\nabla f(x_t)-P^{-1}_{y_t}\nabla f(y_t)\right)\right)
$$
Our  goal is to show that $\xi(\bar{x},  x_t,y_t)$ remains small for every $t\leq\mathcal{T}$.

\begin{lemma}\label{lem_distance}
	Let $f$ be $\ell_1$-gradient and $\ell_2$-Hessian Lipschitz. Let $p_{\lb}I\preceq P(x)\preceq p_\ub I$ and $\left\|P(x)-P(y)\right\|\leq L_P\|x-y\|$ for every $x$ and $y$. Suppose that $u$ satisfies $\|\nabla f(u)\|_u^*\leq \epsilon$. Moreover, suppose that $x$ and $y$ satisfy $\max\{\|x-u\|, \|y-u\|\}\leq \mathcal{S}$ and $\epsilon\leq \mathcal{S}/\sqrt{p_\ub}$ for some $\mathcal{S}\geq 0$. Then, we have 
	\begin{align*}
	\|\zeta(u,x,y)\|\leq L_d\mathcal{S}\left\|P(u)^{1/2}(x-y)\right\|.
	\end{align*}
\end{lemma}
\begin{proof}
	Note that we can write 
	\begin{align*}
	\xi(u,x,y) =& -P(u)^{1/2}P(x)^{-1}\nabla f(x) + P(u)^{1/2}P(y)^{-1}\nabla f(y) + P(u)^{-1/2}\nabla^2 f(u)(x-y)\\ 
	=& \underbrace{-P(u)^{-1/2}\nabla f(x) + P(u)^{-1/2}\nabla f(y) + P(u)^{-1/2}\nabla^2 f(u)(x-y)}_{T_1}\\
	&- \underbrace{P(u)^{1/2}\left((P(x)^{-1}\nabla f(x)-P(u)^{-1}\nabla f(x)) +  (P(u)^{-1}\nabla f(y)-P(y)^{-1}\nabla f(y))\right)}_{T_2}.
	\end{align*}
	We bound the norm of $T_1$ and $T_2$ separately. First, we have 
	\begin{align*}
	\|T_1\| &= \|P(u)^{-1/2} \left(\nabla f(x)-\nabla f(y) - \nabla^2f(u)(x-y)\right)\|\\
	& \leq \frac{1}{\sqrt{p_\lb}}\|\nabla f(x)-\nabla f(y) - \nabla^2f(u)(x-y)\| \\
	& = \frac{1}{\sqrt{p_\lb}}\left\|\int_0^1 \left(\nabla^2 f(y+t(x-y))- \nabla^2f(u)\right)~dt\cdot (x-y) \right\| \\
	&\leq \frac{\ell_2}{\sqrt{p_\lb}}\cdot  \max\{\|x-u\|, \|y-u\|\} \cdot \|x-y\|,
	\end{align*}
	where the last inequality follows from the assumption that the Hessian is Lipscthiz. On the other hand, we have 
	\begin{align*}
	\|T_2\| &\leq \sqrt{p_{\ub}} 	\left\|P(x)^{-1}\nabla f(x)-P(u)^{-1}\nabla f(x) +  P(u)^{-1}\nabla f(y)-P(y)^{-1}\nabla f(y)\right\|\\
	&= \left\|(P(x)^{-1}-P(u)^{-1})(\nabla f(x)-\nabla f(y)) + (P(x)^{-1}-P(y)^{-1})\nabla f(y)\right\| \\
	& \leq \frac{L_P\sqrt{p_{\ub}}}{p_\lb^2}\|x-u\|\cdot \ell_1 \|x-y\| + \frac{L_P\sqrt{p_{\ub}}}{p_\lb^2}\|x-y\| \|\nabla f(y)\|.
	\end{align*}
	Since the gradient is $\ell_1$-Lipschitz, we have 
	\[
	\|\nabla f(y)\| \leq \|\nabla f(u)\| + \ell_1 \|y-u\| \leq p_{\ub}^{1/2}\cdot \|\nabla f(u)\|_u^* + \ell_1 \|y-u\|\leq  p_{\ub}^{1/2}\epsilon + \ell_1 \|y-u\|.
	\]
	As a result, we get 
	\begin{align*}
	\|T_2\| &\leq \frac{L_P\sqrt{p_{\ub}}}{p_\lb^2}\|x-u\|\cdot \ell_1 \|x-y\| + \frac{L_P\sqrt{p_{\ub}}}{p_\lb^2}\|x-y\| \left(\sqrt{p_{\ub}}\epsilon + \ell_1 \|y-u\|\right)
	\end{align*}
	Combining the derived upper bounds for $T_1$ and $T_2$ leads to
		\begin{align*}
		\|\zeta(u,x,y)\|\leq \frac{2\max\{\ell_2,L_P\ell_1\sqrt{p_\ub}\}}{p_\lb^2}\left(\|x-u\|+\|y-u\|\right)\|x-y\|+\frac{L_P{p_\ub}}{p_\lb^2}\epsilon\|x-y\|.
		\end{align*}
		Invoking the assumptions $\max\{\|x-u\|, \|y-u\|\}\leq \mathcal{S}$ and $\epsilon\leq \mathcal{S}/\sqrt{p_\ub}$ yields
		\begin{align*}
		\|\zeta(u,x,y)\|&\leq \frac{5\max\{\ell_2,L_P\ell_1\sqrt{p_\ub}\}}{p_\lb^2}\cdot\mathcal{S}\|x-y\|\\
		&\leq \frac{5\max\{\ell_2,L_P\ell_1\sqrt{p_\ub}\}}{p_\lb^{2.5}}\cdot\mathcal{S}\|P(u)^{1/2}(x-y)\|
		\end{align*}
		This completes the proof.
	
\end{proof}
To prove Lemma~\ref{lem:coupleseq_gd}, we also need  the following lemma, which shows that, for $t\leq \mathcal{T}$, the iterations $\{x_t\}$ remain close to the initial point $x_0$ if $f(x_0) - f(x_t)$ is small. The proof is almost identical to that of  \cite[Lemma 5.4]{jin2021nonconvex} and is omitted for brevity.

\begin{lemma}[Improve or Localize]\label{lem_bound}
	Under the assumptions of Lemma~\ref{lem:escape}, we have for every $t\leq \mathcal{T}$:
	\begin{align*}
	\|x_t-x_0\|\leq \frac{1}{\sqrt{p_\lb}}\sqrt{2\alpha t(f(x_0)-f(x_t))}.
	\end{align*}
\end{lemma}
\begin{proof} [Lemma~\ref{lem:coupleseq_gd}.] By contradiction, suppose that 
\[
\min\{f(x_{\mathcal{T}})-f(x_{0}),f(y_{\mathcal{T}})-f(y_{0})\}>-2\mathcal{F}.
\]
Given this assumption, we can invoke Lemma \ref{lem_bound} to show that both sequences remain close to $\bar{x}$, i.e., for any $t\le\mathcal{T}$: 
\begin{align}
\max\{\norm{x_{t}-\bar{x}},\norm{y_{t}-\bar{x}}\}\le {\frac{1}{\sqrt{p_{\lb}}}}\sqrt{4\alpha\mathcal{T}\mathcal{F}}
= \sqrt{\frac{p_{\lb}\epsilon}{25L_d^2\iota^2\ell_2}}\le  \frac{1}{5\iota}\sqrt{\frac{\epsilon}{L_d}} := \mathcal{S}\label{eq:localization_gd}
\end{align}
where the first equality follows from our choice of $\alpha$, $\mathcal{T}$, $\mathcal{F}$, and $r$. 
Upon defining $z_t = P(\bar{x})^{1/2}(x_t-y_t)$, we have
\begin{align*}
z_{t+1} &=  z_{t}-\alpha{P(\bar{x})^{-1/2}[P(x_t)^{-1}\nabla f(x_{t})-P(y_t)^{-1}\nabla f(y_{t})]} \\
& =(I-\alpha H)z_{t}-\alpha \xi(\bar{x},  x_t,y_t)\\
& =\underbrace{(I-\alpha H)^{t+1}z_{0}}_{p(t+1)}-\underbrace{\alpha\sum_{\tau=0}^{t}(I-\alpha H)^{t-\tau} {\xi(\bar{x},  x_\tau,y_\tau)}}_{q(t+1)},
\end{align*}
where $H = P(\bar{x})^{-1/2}\nabla^2 f(\bar{x})P(\bar{x})^{-1/2}$, and 
$$
\xi(\bar{x},  x_t,y_t) = \alpha P(\bar{x})^{1/2}\left(P(\bar{x})^{-1}\nabla^2f(\bar{x})(x_{t}-y_{t})-\left(P({x_t})^{-1}\nabla f(x_t)-P({y_t})^{-1}\nabla f(y_t)\right)\right)
$$
In the dynamic of $z_{t+1}$, the term $p(t+1)$ captures the effect of the difference in the initial points of the sequences $\{x_t\}_{t=0}^T$ and $\{y_t\}_{t=0}^T$. Moreover, the term $q(t+1)$ is due to the fact that the
function $f$ is not quadratic{ and the metric function $P(x)$ changes along the solution trajectory}. We now use induction to show that the error term $q(t)$ remains smaller than the leading term $p(t)$. In particular, we show
\[
\norm{q(t)}\le\norm{p(t)}/2,\qquad t\in\mathcal{T}.
\]
The claim is true for the base case $t=0$ as $\norm{q(0)}=0\le\norm{z_{0}}/2=\norm{p(0)}/2$. Now suppose the induction hypothesis is true up to $t$. Denote $\lambda_{\min}(H)=-\gamma$ with $\gamma\geq\sqrt{L_d\epsilon}$.
Note that $z_{0}$ lies in the direction of the minimum eigenvector
of ${H}$. Thus, for any $\tau\le t$, we have
\begin{align}\label{eq_induction}
\norm{z_{\tau}}\le\norm{p(\tau)}+\norm{q(\tau)}\le 2\norm{p(\tau)}=2\norm{(I-\alpha H)^{\tau}z_{0}}=2(1+\alpha\gamma)^{\tau}\alpha \omega.
\end{align}
On the other hand, we have
\begin{align*}
\left\|{q(t+1)}\right\| & =\left\|{\alpha\sum_{\tau=0}^{t}(I-\eta H)^{t-\tau}\zeta(\bar{x}, x_t,y_t)}\right\| 
{\le}\alpha \sum_{\tau=0}^{t}\|{( I-\eta H)^{t-\tau}}\|\cdot{L_d \mathcal{S}}\|{z_{\tau}}\|\\
&\leq \alpha \sum_{\tau=0}^{t}\|{( I-\eta H)^{t-\tau}}\|\cdot{L_d \mathcal{S}}\cdot (2(1+\alpha\gamma)^{\tau}\alpha \omega)
\leq 2\alpha L_d\mathcal{S} \sum_{\tau=0}^{t} (1+\alpha\gamma)^{t}\alpha \omega\\
&\leq 2\alpha L_d\mathcal{S}\mathcal{T}p(t+1)
\end{align*}
where in the first inequality we used Lemma~\ref{lem_distance} to bound $\|\zeta(\bar{x}, x_t,y_t)\|$. Moreover, in the second and last inequalities we used (\ref{eq_induction}), $t\leq \mathcal{T}$, and $(1+\alpha\gamma)^{t}\alpha \omega\leq p(t+1)$. Due to our choice of $\alpha$, $L_d$, $\mathcal{S}$, and $\mathcal{T}$, it is easy to see that $2\alpha L_d\mathcal{S}\mathcal{T} = 1/5$. This leads to $\|q(t+1)\|\leq \|p(t+1)\|/5$, thereby completing our inductive argument. Based on this inequality, we have
\begin{align*}
	\max\{\|{x_{\mathcal{T}}-{x}_0}\|,\|{y_{\mathcal{T}}-{x}_0}\|\}
	&\geq \frac{1}{2\sqrt{p_{\ub}}}\|z_{\mathcal{T}}\|
	\geq \frac{1}{2\sqrt{p_{\ub}}}(\|p_{\mathcal{T}}\|-\|q_{\mathcal{T}}\|)
	\geq \frac{1}{4\sqrt{p_{\ub}}}\|p_{\mathcal{T}}\|\\
	&\geq \frac{(1+\alpha\gamma)^{\mathcal{T}}\alpha \omega}{4\sqrt{p_{\ub}}}
	\overset{(a)}{\ge} \frac{2^{\iota/2-3}p_\lb}{\sqrt{p_\ub}\ell_1}\bar{\omega}
	\overset{(b)}{>}\mathcal{S}.
\end{align*}
where in $(a)$, we used the inequality $(1+x)^{1/x}\geq 2$ for every $0<x\leq 1$ and in $(b)$, we used the definition of $\bar{\omega}$. 
The above inequality contradicts with (\ref{eq:localization_gd}) and therefore completes our proof.
\end{proof}

\subsection{Proof of Lemma~\ref{lem_Lip}}
To prove Lemma~\ref{lem_Lip}, first we provide an upper bound on $f(X_t)$.
\begin{lemma}\label{lem_ub}
	For every iteration $X_t$ of \ref{PPrecGD}, we have
	\[
	f(X_t)\leq f(X_0)+2\sqrt{\|X_0\|_F^2+\eta}\cdot \alpha \beta\epsilon
	\]
\end{lemma}
\begin{proof}
	Note that $f(X_{t+1})\leq f(X_t)$, except for when $X_t$ is perturbed with a random perturbation. Moreover, we have already shown that, each perturbation followed by $\mathcal{T}$ iterations of PMGD strictly reduces the objective function. Therefore, $f(X)$ takes its maximum value when it is perturbed at the initial point. This can only happen if $X_0$ is close to a strict saddle point, i.e.,
	\[
	\|\nabla f(X_0)\|_{X_0,\eta}^{*}\le\epsilon,\quad\text{and} \quad \nabla^{2}f(X_0)\not\succeq-\sqrt{L_d\epsilon}\cdot \P_{X_0,\eta},
	\]
	Therefore, we have $X_1 = X_0+\alpha \zeta$, where $\zeta\sim\mathbb{B}(\beta)$. This implies that
	\begin{align*}
		f(X_1)-f(X_0)&\leq \alpha\langle\nabla f(X_0), \zeta\rangle+\frac{\alpha^2L_1}{2}\|\zeta\|_F^2
		\leq \alpha\norm{\P_{X_0,\eta}^{1/2}}_F\norm{\nabla f(X_0)}^*_{X_0,\eta}\norm{\zeta}_F+\frac{\alpha^2L_1}{2}\|\zeta\|_F^2\\
		&\overset{(a)}{\leq} \sqrt{\norm{X_0}_F^2+\eta}\cdot\epsilon\alpha \beta +\frac{L_1}{2}\alpha^2\beta^2
		\overset{(b)}{\leq} 2\sqrt{\norm{X_0}_F^2+\eta}\cdot\epsilon\alpha \beta
	\end{align*}
	where $(a)$ follows from our assumption $\|\nabla f(X_0)\|_{X_0,\eta}^{*}\le\epsilon$, and $(b)$ is due to our choice of $\alpha$ and $\beta$. This implies that $f(X_t)\leq f(X_1)\leq f(X_0)+2\sqrt{\norm{X_0}_F^2+\eta}\cdot\epsilon\alpha \beta$, thereby completing the proof.
\end{proof}
The above lemma combined with the coercivity of $\phi$ implies that 
$$
X_t\in\mathcal{M}\left(\phi(X_0X_0^T)+2\sqrt{\|X_0\|_F^2+\eta}\cdot \alpha \beta\epsilon\right)
$$
for every iteration $X_t$ of \ref{PPrecGD}. Now, we proceed with the proof of Lemma~\ref{lem_Lip}.\vspace{2mm}

\begin{proof} [Lemma~\ref{lem_Lip}.] First, we prove the gradient lipschitzness of $f(X)$. Due to the definition of $\Gamma_F$ and Lemma~\ref{lem_ub}, every iteration of \ref{PPrecGD} belongs to the ball $\{M: \|M\|_F\leq \Gamma_F\}$. For every $X,Y\in\{M: \|M\|_F\leq \Gamma_F\}$, we have
\begin{align*}
	\norm{\nabla f(X)-\nabla f(Y)}_F &= 2\norm{\nabla \phi(XX^\top)X-\nabla \phi(YY^\top)Y}_F\\
	&\leq \norm{\nabla \phi(XX^\top)-\nabla \phi(YY^\top)}_F\norm{X}_F+2\norm{\phi(YY^\top)}_F\norm {X-Y}_F\\
	&\leq 2L_1\Gamma_F\norm {XX^\top-YY^\top}_F+5L_1\Gamma_F^2\norm {X-Y}_F\\
	&\leq 2L_1\Gamma_F\norm {X(X-Y)^\top - (Y-X)Y^\top}_F+5L_1\Gamma_F^2\norm {X-Y}_F\\
	&\leq 9L_1\Gamma_F^2\norm {X-Y}_F
\end{align*}
which shows that $f(X)$ is $9L_1\Gamma_F^2$-gradient Lipschitz within the ball $\{M: \|M\|_F\leq \Gamma_F\}$. 
Next, we prove the Hessian lipschitzness of $f(X)$. For any arbitrary $V$ with $\norm{V}_F=1$, we have
\begin{align*}
	&\left|\left\langle\nabla^2 f(X)[V],V\rangle-\langle\nabla^2 f(Y)[V],V\right\rangle\right| \\
	=& 2\left|\left\langle\nabla \phi(XX^\top)-\nabla \phi(YY^\top),VV^\top\right\rangle\right|\\
	&+\left|\left\langle\nabla^2\phi(XX^\top), XV^\top+VX^\top\right\rangle-\left\langle\nabla^2\phi(YY^\top), YV^\top+VY^\top\right\rangle\right|\\
	\leq & 2\left\|\nabla\phi(XX^\top)-\nabla \phi(YY^\top)\right\|_F\\
	&+\left|\left\langle\nabla^2\phi(XX^\top), XV^\top+VX^\top\right\rangle-\left\langle\nabla^2\phi(YY^\top), XV^\top+VX^\top\right\rangle\right|\\
	&+\left|\left\langle\nabla^2\phi(YY^\top), YV^\top+VY^\top\right\rangle-\left\langle\nabla^2\phi(YY^\top), XV^\top+VX^\top\right\rangle\right|\\
	\leq & 2L_1\left\|XX^\top-YY^\top\right\|_F\\
	&+\left\|\nabla^2\phi(XX^\top)-\nabla^2\phi(YY^\top)\right\|\left\|XV^\top+VX^\top\right\|_F\\
	&+\left\|\nabla^2\phi(YY^\top)\right\|\left\|(Y-X)V^\top+V(Y-X)^\top\right\|_F\\
	\leq& 4L_1\Gamma_F\left\|X-Y\right\|_F+4L_2\Gamma_F^2\left\|X-Y\right\|_F+2L_1\left\|X-Y\right\|_F\\
	=& ((4\Gamma_F+2)L_1+4\Gamma_F^2 L_2)\left\|X-Y\right\|_F.
\end{align*}
Therefore, $f(X)$ is $((4\Gamma_F+2)L_1+4\Gamma_F^2 L_2)$-Hessian Lipschitz within the ball $\{M: \|M\|_F\leq \Gamma_F\}$. Finally, it is easy to verify that the eigenvalues of $\P_{X,\eta} = (X^{T}X+\eta I_{n})\otimes I_{r}$ are between $\eta$ and $\Gamma_2^2+\eta$ within the ball $\{M:\norm {M}\leq \Gamma_2 \}$. Moreover, $\norm {\P_{X,\eta}-\P_{Y,\eta}}\leq \norm {X^\top X - Y^\top Y}\leq 2\Gamma_2\norm {X-Y}$. This completes the proof of this lemma.
\end{proof}

\bibliography{refsnew,reference}

\end{document}